%% file: MESP.tex
\documentclass[11pt]{article}

\input{template}
\usepackage{authblk}

\begin{document}

\title{From Majorization to Scaling: Advancing Convex Relaxations of Maximum Entropy Sampling Problem}

\author[1]{Lingqing Shen}
\author[1]{Fatma K{\i}l{\i}n\c{c}-Karzan}
\affil[1]{Tepper School of Business, Carnegie Mellon University}

\date{\today}

\maketitle

\begin{abstract}
In this paper we study the maximum entropy sampling problem (MESP) and its variants. MESP seeks to identify a small subset of variables that maximizes the determinant of a covariance submatrix, and is a fundamental model in optimal experimental design and information acquisition.  Although MESP is combinatorial and NP-hard, continuous relaxations, most notably linx and $\Gamma$ factorization, provide tractable approximations, yet their derivation, relative strength, and potential for systematic improvement remain poorly understood. 
We address this gap by introducing two main ideas: a unified majorization-based  framework for deriving and analyzing relaxations, and a novel scaling-based bound-enhancement technique, which we call \emph{double-scaling}. 
Our approach is motivated by the observation that the difficulty of MESP arises from two distinct sources: the combinatorial selection structure and the lack of permutation symmetry in the spectral objective. Majorization naturally resolves the latter by symmetrizing the spectral function and yielding its convex envelope. 
In the log-determinant setting, we establish the main theoretical properties of double-scaling and  prove that it strictly dominates previously known scaling bounds. Using our majorization-based alternative characterization of $\Gamma$ factorization relaxation, we also derive, for the first time, formal dominance relations between linx- and $\Gamma$ factorization-bounds, as well as between their scaling-strengthened variants. 
Our numerical results show that our double-scaled linx relaxation consistently and substantially outperforms existing scaling methods and compares quite favorably with other state-of-the-art relaxations in terms of both bound quality and computational efficiency.
\end{abstract}

\input{intro}

\input{majorization}

\input{other_relaxations}

\input{scaling}

\input{experiments}

\section*{Acknowledgements}
This research was supported in part by AFOSR [Grant FA9550-22-1-0365].

\bibliographystyle{abbrvnat}
\bibliography{ref} 

\appendix
\input{appendix}

\input{appendix_subgradient}

\end{document}

%% file: template.tex
\usepackage[margin=1in]{geometry}

\usepackage{parskip}                
\usepackage{amsmath}                
\usepackage{amssymb}                
\usepackage{amsthm}
\usepackage{mathtools}              
\usepackage{bm}                     
\usepackage{array}                  
\usepackage{booktabs}
\usepackage{multirow}
\usepackage{xcolor}                 
\usepackage[square]{natbib}

\usepackage[colorlinks=true,citecolor=blue!75!black,urlcolor=blue!75!black]{hyperref}  
\usepackage[capitalise]{cleveref}
\usepackage{algorithm}
\usepackage{algpseudocode}
\usepackage{tikz}
\usepackage[shortlabels]{enumitem}

\allowdisplaybreaks

\input{macros}

\newtheorem{theorem}{Theorem}
\newtheorem{proposition}{Proposition}
\newtheorem{lemma}{Lemma}
\newtheorem{corollary}{Corollary}

\crefname{assumption}{Assumption}{Assumptions}
\theoremstyle{definition}
\newtheorem{definition}{Definition}
\newtheorem{remark}{Remark}
\newtheorem{example}{Example}

\newcommand{\majorize}{\succcurlyeq_{m}}

\newcommand{\objldlog}{\bar{f}_{\textup{linx-d}}}       
\newcommand{\objld}{f_{\textup{linx-d}}}                
\newcommand{\objgg}{f_{\Gamma\textup{-g}}}              
\newcommand{\objggc}{f_{\Gamma^c\textup{-g}}}           
\newcommand{\objgglog}{\bar{f}_{\Gamma\textup{-g}}}

\newcommand{\objlglog}{\bar{f}_{\textup{linx-g}}}       
\newcommand{\objlg}{f_{\textup{linx-g}}}                

\newcommand{\objlo}{f_{\textup{linx-o}}}                
\newcommand{\objlolog}{\bar{f}_{\textup{linx-o}}}

\newcommand{\xigg}{\bm\xi^{\Gamma\textup{-g}}}          
\newcommand{\xiggc}{\bm\xi^{\Gamma^c\textup{-g}}}       
\newcommand{\xild}{\bm\xi^{\textup{linx-d}}}            
\newcommand{\xiggi}{\xi^{\Gamma\textup{-g}}_i}          
\newcommand{\xiggci}{\xi^{\Gamma^c\textup{-g}}_i}       
\newcommand{\xildi}{\xi^{\textup{linx-d}}_i}            

\newcommand{\linxo}{\textup{linx-o}}                       
\newcommand{\linxg}{\textup{linx-g}}                       
\newcommand{\linxd}{\textup{linx-d}}

\newcolumntype{R}{>{$}r<{$}}
\newcolumntype{L}{>{$}l<{$}}

\allowdisplaybreaks

%% file: macros.tex
\definecolor{gold}{rgb}{0.85,0.65,0}

\newcommand{\ip}[1]{\ensuremath{\left\langle #1 \right\rangle}}

\let\emptyset\varnothing
\newcommand{\set}[1]{\left\{#1\right\}}

\def\K{{\mathbb{K}}}

\def\R{{\mathbb{R}}}
\def\S{{\mathbb{S}}}

\def\bA{{\bm{A}}}
\def\bB{{\bm{B}}}
\def\bC{{\bm{C}}}
\def\bD{{\bm{D}}}
\def\bE{{\bm{E}}}

\def\bI{{\bm{I}}}

\def\bL{{\bm{L}}}
\def\bM{{\bm{M}}}

\def\bP{{\bm{P}}}
\def\bQ{{\bm{Q}}}

\def\bS{{\bm{S}}}

\def\bU{{\bm{U}}}
\def\bV{{\bm{V}}}
\def\bW{{\bm{W}}}
\def\bX{{\bm{X}}}
\def\bY{{\bm{Y}}}
\def\bZ{{\bm{Z}}}

\def\cA{{\cal A}}

\def\cC{{\cal C}}
\def\cD{{\cal D}}

\def\cF{{\cal F}}
\def\cG{{\cal G}}

\def\cP{{\cal P}}
\def\cQ{{\cal Q}}

\def\cU{{\cal U}}

\def\cX{{\cal X}}

\DeclareMathOperator{\Opt}{Opt}

\DeclareMathOperator{\rank}{rank}
\DeclareMathOperator{\Diag}{Diag}
\DeclareMathOperator{\diag}{diag}
\DeclareMathOperator{\tr}{tr}

\DeclareMathOperator{\range}{range}

\DeclareMathOperator{\spann}{span}

\DeclareMathOperator{\cl}{cl}

\DeclareMathOperator{\dom}{dom}

\DeclareMathOperator{\epi}{epi}

\DeclareMathOperator{\conv}{conv}

\DeclareMathOperator{\Proj}{Proj}

\newcommand{\ones}{{\bm{1}}}

\newcommand{\x}{{\bm x}}
\newcommand{\z}{{\bm z}}
\newcommand{\y}{{\bm y}}

\newcommand{\g}{{\bm g}}
\newcommand{\e}{{\bm e}}

\newcommand{\w}{{\bm w}}

\newcommand{\q}{{\bm q}}

\renewcommand{\a}{{\bm a}}

\renewcommand{\u}{{\bm u}}
\renewcommand{\v}{{\bm v}}

\renewcommand{\b}{{\bm b}}

%% file: intro.tex
\section{Introduction}

Dating back to \citep{shewry_maximum_1987}, the \emph{maximum-entropy sampling problem} (MESP) is a fundamental combinatorial optimization problem arising in statistics, experimental design, and machine learning (see \citep{ko_exact_1995, krause_near-optimal_2008} and \cite{fampa_maximum-entropy_2022,fampa_recent_2026} for recent comprehensive treatments). Given a positive semidefinite matrix $\bC\in\S^d_+$, the goal is to select a subset of indices $S\subseteq[d]\coloneqq\{1,\ldots,d\}$ of cardinality $s$ (where $s\le \rank(\bC)$) that maximizes the information content of the corresponding principal submatrix $\bC_{S,S}$. In the classical D-optimal design setting, this amounts to maximizing the log-determinant:
  \begin{align*}
    -\Opt(\bC,s) &:= \max_{S\subseteq[d],|S|=s} \set{\log\det(\bC_{S,S})} .
  \end{align*}
Despite its simple formulation, MESP is NP-hard \citep{ko_exact_1995}, which has motivated extensive research on relaxations and approximation methods.

MESP is part of a broader family of subset-selection problems with spectral objectives. Important variants include constrained MESP (CMESP), which adds linear side constraints, 
and A-optimal design problem (A-MESP), which replaces the determinant-based objective with the trace of the inverse covariance matrix, i.e., $\tr(\bC_{S,S}^{-1})$, as an alternative measure of information content. 
In the language of experimental design, MESP corresponds to the classical D-optimality criterion, while A-optimality and related criteria evaluate information through alternative functions of $\bC_{S,S}$. 
These variants share a common structural foundation that allows for a unified analytical treatment. 
To make this precise, consider a factorization $\bC=\bV^\top\bV$ with $\bV=(\v_1,\dots,\v_d)\in\R^{m\times d}$. For any subset $S\subseteq[d]$, let $\x\in\{0,1\}^d$ denote its indicator vector, so that $\ones^\top\x=s$. 
By letting $\Diag(\x)$ denote the diagonal matrix whose diagonal entries are given by $\x$, the principal submatrix  $\bC_{S,S}$ is given by $\Diag(\x)\bC\Diag(\x)$.
Moreover, it is well known that the nonzero eigenvalues of $\Diag(\x)\bC\Diag(\x)$ (and thus $\bC_{S,S}$) coincide with those of $\bV\Diag(\x)\bV^\top$. Thus, the principal submatrix $\bC_{S,S}$ appearing in MESP and its variants can equivalently be represented by the matrix $\bV\Diag(\x)\bV^\top$, which depends linearly on $\x \in \{0,1\}^d$.
Consequently, MESP can be equivalently written as
  \begin{align*}
    -\Opt(\bC,s)  
    &= \max_{\x\in\set{0,1}^d} \set{\sum_{i\in[s]}\log(\lambda_i(\bV\Diag(\x)\bV^\top)): ~\ones^\top\x=s}, 
  \end{align*}
where $\bm\lambda(\bM)$ denotes the vector of eigenvalues of a symmetric matrix $\bM$. 
This formulation reveals that the problem amounts to selecting a subset that maximizes a spectral functional of a matrix that depends linearly on $\x$.

Motivated by this formulation, given a subset size $s$, we consider the following general class of nonlinear integer programs:
\begin{align}\label{eq:mesp-general}
  \begin{aligned}
   {\Opt} \coloneqq \min_{\x\in\cX\cap\{0,1\}^d} f_s\left( \bm\lambda\left(\bV\Diag(\x)\bV^\top\right) \right),
  \end{aligned}
\end{align}
where 
$f_s(\bm\xi)=\sum_{i\in[s]}\varphi(\xi_i)$ for a closed, convex, and nonincreasing function $\varphi:\R_{++}\to\R$, and $\cX=\set{\x\in[0,1]^d:\ones^\top\x=s,\bA\x\le\b}$. 
Problem~\eqref{eq:mesp-general} unifies several important problems. 
For example, by setting $\varphi(\xi)=-\log(\xi)$ and $\cX=\set{\x\in[0,1]^d:\ones^\top\x=s}$, we arrive at the standardl MESP; including  additional nontrivial linear constraints, i.e., $\bA\x\leq\b$, in $\cX$ leads to the CMESP introduced by \cite{lee_constrained_1998}.
Whenever $\varphi(\xi)= 1/\xi$, we arrive at A-MESP \citep{li_best_2023} and its variants.

Significant effort on MESP research has focused on deriving convex relaxations that provide tight bounds for branch-and-bound algorithms 
 \citep{anstreicher_using_1999, anstreicher_maximum-entropy_2018, chen_mixing_2021}. Approximation algorithms, including local search heuristics \citep{ko_exact_1995} and randomized sampling techniques \citep{nikolov_randomized_2015, li_best_2023}, are often used for large-scale instances for which exact methods quickly become computationally intractable.

\textbf{Relaxation-based bounds:}
A major line of research focuses on deriving strong convex relaxations to obtain bounds on the optimal value of MESP (as our focus is on convex minimization based formulation in \eqref{eq:mesp-general}, all of these relaxations will provide \emph{lower bounds} on $\Opt$ as opposed to upper bounds). Early relaxations include the spectral bound \citep{ko_exact_1995}, masked spectral bounds \citep{anstreicher_masked_2004, burer_solving_2007}, and the Nonlinear Programming (NLP) bound \citep{anstreicher_continuous_1996,anstreicher_using_1999}, 
which is based on a novel continuous convex relaxation.
Later on, through a lifted semidefinite programming framework, the Boolean Quadratic Polytope (BQP) bound \citep{anstreicher_maximum-entropy_2018} was introduced and shown to produce one of the tightest known bounds, but at a significant computational cost. To improve scalability,  \cite{anstreicher_efficient_2020}  proposed the \emph{linx} relaxation by leveraging a determinant identity.  
Based on a low-dimensional factorization of the covariance matrix $\bC=\bV^\top\bV$, \citep{nikolov_randomized_2015} proposed a novel convex relaxation to MESP, later referred to as the $\Gamma$ \emph{factorization} (or \emph{DDFact}) relaxation. Subsequently, \citep{li_best_2023} showed that the $\Gamma$ \emph{factorization}  relaxation admits a dual interpretation as the double Lagrangian dual of the continuous relaxation of the original nonconvex problem. 
More recent work by  \citep{li_augmented_2024} introduced an \emph{augmented factorization} variant that offers a slightly better bound quality while retaining its favorable computational properties.
It is well known that the factorization $\Gamma$ bound dominates the spectral bound for standard MESP and the BQP bound is too expensive for scalability. 
Empirically, the $\Gamma$-factorization and linx relaxations are widely regarded as the state-of-the-art due to their favorable balance between tightness and scalability.

\textbf{Bound enhancement techniques:}
In the case of D-optimal design, i.e., $\varphi(\xi)=-\log(\xi)$, 
several bound-enhancement techniques have been proposed to strengthen existing relaxations without changing their underlying structure; see \citep{anstreicher_continuous_1996,anstreicher_using_1999, anstreicher_masked_2004,chen_mixing_2021} and \citep{fampa_maximum-entropy_2022} for a comprehensive survey. These include \emph{scaling} (rescaling the covariance matrix $\bC$), \emph{complementation} (relying on a determinant identity related to complementary subsets), \emph{masking} (using diagonal selection matrices), and \emph{mixing} (combining multiple relaxation objectives and taking their pointwise maximum to leverage the tightest bound attainable among them). 

Among these bound enhancement techniques, \emph{scaling} technique \citep{anstreicher_continuous_1996,anstreicher_using_1999} has received particular attention due to its empirical effectiveness.
The \emph{ordinary scaling} (o-scaling) technique introduces a scalar parameter $\gamma_0>0$ to scale the covariance matrix $\bC \to \gamma_0\bC$, thereby modifying the curvature of the objective function of the relaxation. Although theoretically simple, when applied judiciously, o-scaling is empirically shown to substantially tighten some relaxation bounds (see \citep{chen_mixing_2021,chen_generalized_2024} and references therein), especially linx and BQP, which are both non-invariant under scaling \citep{anstreicher_maximum-entropy_2018, anstreicher_efficient_2020}. 
On the other hand, it is known that the spectral bound and the $\Gamma$ factorization relaxation are invariant under o-scaling. 

Recently, 
\cite{chen_generalized_2024} introduced \emph{generalized scaling} (g-scaling), extending o-scaling to a vector of positive scaling coefficients and thus enabling component-wise rescaling of $\bC$. They demonstrated that, for the linx and BQP bounds, the strengthened relaxation obtained from g-scaling is concave in the logarithm of the scaling vector, thus ensuring that the optimal g-scaled linx bound can be obtained efficiently through convex optimization. They further extended g-scaling to the $\Gamma$ factorization relaxation (recall that  $\Gamma$ factorization relaxation is invariant to o-scaling). As they could not establish a convex-concave structural result for g-scaled $\Gamma$ factorization relaxation, they instead relied on generalized differentiability results to theoretically justify the use of quasi-Newton and active-set methods in optimizing the g-scaled $\Gamma$ factorization bound, and provided quite promising numerical results for g-scaled linx and g-scaled $\Gamma$ factorization bounds. 
While g-scaling presented a major advance with notable practical impact, 
its existing analysis lacks a unified theoretical explanation, particularly regarding its convexity properties and invariance behavior for $\Gamma$ factorization relaxation. Moreover, so far no theoretical connection between the quality of the state-of-the-art linx and $\Gamma$ relaxations or their scaled variants is known.

In this paper we address these gaps and advance the theoretical foundation and practical effectiveness of convex relaxations for the MESP and its variants by introducing two main ideas: a \emph{unified majorization-based  framework} for deriving and analyzing convex relaxations, and a \emph{novel scaling-based bound enhancement} technique.
Our developments are motivated by the observation that the nonconvexity of MESP arises from two different sources. The first is the cardinality-constrained selection structure. The second is more subtle: 
even when the underlying scalar function is convex, the induced spectral objective need not be convex because it lacks permutation-symmetry. 
We show that majorization naturally addresses this second difficulty.  
By symmetrizing the spectral function, it produces the corresponding convex envelope and leads to a tractable convex relaxation. 
Based on this perspective, we then place classical relaxations and scaling techniques under a common umbrella and also provide the first systematic comparison of their relative theoretical strength in the literature.
Our key contributions, along with the organization of the paper, are as follows:
\begin{itemize}

\item 
We begin in \cref{sec:prelim} with preliminaries on majorization and spectral functions. In \cref{sec:majorization}, we introduce a majorization-based convex relaxation for \eqref{eq:mesp-general}, applicable to general convex non-increasing functions $\varphi$ and general polyhedral domains $\cX$. We analyze its convex hull and exactness guarantees in \cref{sec:majorization:properties}. 
For the cases of $\varphi(\xi)=-\log(\xi)$ and $\varphi(\xi)={1/\xi}$, we show in \cref{sec:equivalent} that our majorization-based relaxation coincides exactly with the $\Gamma$ (factorization) relaxation. 
This yields a unified and more transparent derivation of $\Gamma$ factorization relaxation,
while immediately extending the framework beyond the cases previously studied. 
Moreover, this majorization-based perspective plays a crucial role in our developments of \cref{sec:gamma-g-scaling,sec:connection}.

\item We next focus on the determinant-based case, i.e., $\varphi(\xi)=-\log(\xi)$, and study scaling techniques for the linx and $\Gamma$ factorization relaxations in \cref{sec:scaling,sec:gamma-g-scaling}, respectively. Building on the classical linx relaxation and the existing ordinary and generalized scaling techniques (see \cref{sec:o-g-scaling}), we introduce a novel \emph{double-scaling} approach in   \cref{sec:linx-double-scaling} that strictly generalizes the earlier methods. Furthermore, in \cref{prop:d-scaling-linx-convexity}, we show that our double-scaled linx relaxation admits a convex-concave saddle point formulation. 

In \cref{sec:gamma-g-scaling}, we analyze the g-scaled $\Gamma$  relaxation. 
Leveraging our machinery for establishing the convex-concave structure of double-scaled  linx relaxation, we prove in \cref{prop:gamma-g-scaling-convexity} that g-scaled $\Gamma$ relaxation also possesses such a structure, thereby resolving an  open question of  \cite{chen_generalized_2024}. 
These saddle-point representations provide the missing theoretical foundation for solving the scaling-enhanced linx and $\Gamma$ relaxations with modern first-order methods. 

Furthermore, using the convex-concave reformulation of g-scaled $\Gamma$ relaxation, we show in \cref{prop:g-scaling-Gamma-invariant} that when $\cX=\{\x\in[0,1]^d:\ones^\top\x=s\}$, g-scaling does not strengthen the $\Gamma$  relaxation. This then resolves another open question from  \cite{chen_generalized_2024} and 
provides a theoretical explanation for previously observed empirical behavior.

\item In \cref{sec:connection}, we establish the first theoretical relationship between the  bounds obtained from the linx and $\Gamma$  factorization relaxations and their respective scaled variants. 
In particular, we prove that the double-scaled linx bound dominates the average of the g-scaled $\Gamma$ and g-scaled complementary $\Gamma^c$ bounds, and the o-scaled linx bound dominates the average of the standard $\Gamma$ and  complementary $\Gamma^c$ bounds (see \cref{prop:connection,cor:linx-gamma-connection}). Utilizing this relation, in \cref{cor:linx-integrality-gap} we derive the first explicit integrality gap bound for the o-scaled linx relaxation as well. 

\item Finally, in Section \ref{sec:numerical}, we provide a numerical study on standard MESP instances from the literature with $\varphi(\xi)=-\log(\xi)$ and $\cX=\{\x\in[0,1]^d:\ones^\top\x=s\}$. The results show that our double-scaled linx relaxation substantially strengthens the relaxation bound compared to existing o- and g-scaling techniques, while remaining computationally efficient. Moreover, across nearly all tested instances, it 
consistently achieves the strongest bounds among state-of-the-art relaxations, including the $\Gamma$ factorization relaxation, its complementary $\Gamma^c$ variant, and their mixing strengthened version.

\end{itemize}
Collectively, our results unify and extend prior approaches based on convex relaxations and scaling-based bound enhancement techniques, establish new structural and dominance relations among major relaxations, and contribute both theoretical insights and computational advances for MESP and its variants.

\subsection{Notation}

We let $\overline{\R}:=\R\cup\set{+\infty}$. Given a positive integer $d$, let $[d]\coloneqq\set{1,2,\dots,d}$. For $k\in[d]$, we define $(k,d]:=\set{k+1,\ldots,d}$ and $[k,d):=\set{k,\ldots,d-1}$. 
For $i\in[d]$, let $\e_i$ be the $i$\textsuperscript{th} standard basis vector in $\R^d$.
We denote the vectors of all zeros and all ones in $\R^d$ with ${\bm 0}$ and $\ones$, respectively.
Given $\x\in\R^d$, for any $i\in[d]$, we let $x_i$ be the $i$\textsuperscript{th} component of $\x$, we define $x_{[i]}$ to denote the $i$\textsuperscript{th} largest component of $\x$, and we let $\x_{1:s}\in\R^s$ be the vector formed by the first $s$ coordinates of $\x$. We use $\|\x\|_0$ to represent the sparsity of $\x$, i.e., the number of nonzero entries in $\x$. 
With slight abuse of notation, we denote $\exp(\x)\coloneq (\exp(x_1), \dots, \exp(x_d))$ to be the vector obtained by applying the exponential component-wise, and for $\alpha\in\R$, we denote $\x^{\alpha}\coloneq(x_1^{\alpha}, \dots, x_d^{\alpha})$ analogously. 
Given two vectors $\u,\v \in \R^d$, we use the notation $\u \majorize \v$ 
to denote that $\u$ majorizes 
$\v$. 
We define the polyhedral \emph{ordering} cone to be $\K_{o,d}\coloneqq\set{\x\in\R^d:x_1\geq\cdots\geq x_d}$. 

We let $\S^d$ be the vector space of $d\times d$ real symmetric matrices. In $\S^d$, we denote the identity matrix, the matrix of all zeros, and the cone of positive semidefinite matrices by $\bI_d$,  $\bm 0$, and $\S^d_+$, respectively. 
For $\bX\in\S^d$, we write $\bX\succeq 0$ (resp.\ $\bX\succ 0$) to denote that $\bX$ is positive semidefinite (resp.\ positive definite).
Given $\bX\in\S^d$, let $\lambda_i(\bX)$ be the $i$\textsuperscript{th} largest eigenvalue of $\bX$, and $\bm\lambda(\bX)\coloneqq(\lambda_1(\bX),\dots,\lambda_d(\bX))\in\R^d$ be the vector of eigenvalues sorted in non-increasing order. For $\bX,\bY\in\S^d$, we use the Frobenious inner product, i.e., $\langle\bX,\bY\rangle\coloneqq\tr(\bX\bY)$. We denote the Hadamard (element-wise) product of $\bX,\bY\in\S^d$ by $\bX\circ\bY$. 
For $\x\in\R^d$, we represent the diagonal matrix with diagonal entries $\x$ by $\Diag(\x)\in\S^d$. 
We define $\bE_{ii}\coloneqq\Diag(\e_i)$. 
For a square matrix $\bM\in\R^{d\times d}$, we denote the vector consisting of its diagonal entries by $\diag(\bM)\in\R^d$ 
and 
its trace by $\tr(\bM)$. 
Given an index set $S\subseteq [d]$, a vector $\x\in\R^d$, and a matrix $\bM\in\R^{d\times d}$, we let 
$\bM_{S,S}$ denote the principal submatrix of $\bM$ indexed by $S$. Likewise, for any vector $\bm\gamma\in\R^d$, let $\bm\gamma_{S}$ be the subvector consisting of entries indexed by $S$.

For a function $f:X\to\overline{\R}$, 
we denote its domain by $\dom(f)\coloneqq\set{\x\in X:\ f(x)\in\R }$ and 
its epigraph by $\epi(f)\coloneqq\set{(t,\x)\in \R \times X:\ t\ge f(x) }$. 
Given a convex function $f:\R^d\to \overline{\R}$ 
and a vector $\x\in\R^d$, we let $\partial f(\x)$ be the subdifferential of $f$ at the point $\x$, and we denote  a {gradient} of $f$ at $\x$ by $\nabla f(\x)$.

%% file: majorization.tex
\section{Preliminaries on majorization and spectral functions}\label{sec:prelim}
In this section, we review key concepts and results from majorization theory and spectral functions (see \citep{marshall_inequalities_2011} and \citep{lewis_convex_1996} for comprehensive treatments). We also present a new {convex hull result for a spectral set that will play an important role in our subsequent developments.}

\begin{definition}\label{def:majorization}
  For any $\u,\w\in\R^d$, we say that $\u$ \emph{majorizes} $\w$ (denoted by $\u\majorize\w$) if $\ones^\top\u=\ones^\top\w$ and
  \begin{align*}
    \sum_{i\in[j]}u_{[i]} \geq \sum_{i\in[j]}w_{[i]},\quad\forall j\in[d-1].
  \end{align*}
\end{definition}

Majorization relation among two vectors in fact admits the following well-known characterization as well. 
\begin{lemma}[{\citep[Theorem I.2.B.2, Corollary I.2.B.3]{marshall_inequalities_2011}}]\label{thm:char:majorization}
Given two vectors $\u,\w\in\R^d$, we have $\u\majorize \w$ if and only if 
$ 
\w\in\conv\set{\bP\u:\,\bP\in\R^{d\times d}\textup{ is a permutation matrix}},
$ 
which holds if and only if there exists a doubly stochastic matrix $\bQ\in\R^{d\times d}$ such that $\w=\bQ\u$.
\end{lemma}
\cref{thm:char:majorization} states that any vector majorized by $\u\in\R^d$ must come from the convex hull of all permutations of the given vector $\u$. 

A generalization of \cref{thm:char:majorization} is possible in the matrix majorization case as well. 
\begin{theorem}[{\citep[Exercise III.30]{kilinc-karzan_essential_2025}}]\label{thm:char:majorization:matrix}
Given two matrices $\bU,\bW\in\S^d$, we have $\bU\majorize\bW$, i.e., $\bm\lambda(\bU)\majorize\bm\lambda(\bW)$, if and only if 
\[ 
\bW\in\conv\set{\bQ^\top\bU\bQ:\,\bQ\in\R^{d\times d}\textup{ is an orthogonal matrix}}.
\] 
\end{theorem}

\begin{lemma}\label{lem:majorization-convex-representation}
Both sets $\cA_v \coloneqq \set{(\u,\w)\in\K_{o,d}\times\R^d:\,\u\majorize\w}$ and $\cA_m \coloneqq \set{(\xi,\bX)\in \K_{o,d}\times\S^d:\,\xi\majorize\bm\lambda(\bX)}$ are closed and convex.
\end{lemma}
\begin{proof}
Recall that, for any $j\in[d]$, the function defined as the sum of $j$ largest entries in $\w$, i.e., $q_j(\w)\coloneqq \sum_{i\in[j]}w_{[i]}$ is a closed convex function of $\w$, and also the function defined as the sum of $j$ largest eigenvalues of $\bX\in\S^d$, i.e., $Q_j(\bX)\coloneqq q_j(\bm\lambda(\bX))$, is a closed convex function of $\bX$ (see  \cite[B.1.3.(e)]{hiriart-urruty_fundamentals_2001}). 
Then, $\cA_v$ is closed and convex because for any $\u\in\K_{o,d}$ and any $j\in[d]$, we have $\sum_{i\in[j]}u_{[i]}=\sum_{i\in[j]}u_{i}$ and thus the constraint $\u\majorize\w$ can be represented equivalently by 
the system of convex constraints $\sum_{i\in[j]} u_i \ge q_j(\w)$, for all $j\in[d-1]$, and $\sum_{i\in[d]} u_i = \sum_{i\in[d]} w_i$. Similarly, $\cA_m$ is closed and convex because we also have $\sum_{i\in[d]} \lambda_i(\bX) = \tr(\bX)$, which is a linear function of $\bX$. 
\end{proof}

\cref{thm:char:majorization:matrix} states that any matrix majorized by $\bU\in\S^d$ must come from the convex hull of all rotations of the given matrix $\bU$. 
This characterization also yields a convenient convex hull result that will be instrumental in analyzing the convex hull structure of sets arising in MESP problems.  
\begin{proposition}\label{prop:conv-general}
  Suppose $\cU\subseteq\R^l$ and $\cC\subseteq  \cU\times\K_{o,d} $ are convex sets and $\cU$ is closed. Then, for any set of the form 
  \[
      \overline{\cA} \coloneqq \set{
     (\u,\bm\xi,\bX) \in \cU \times \K_{o,d} \times \S_+^d  :
     	\begin{array}{>{\displaystyle}l} 
       		\bm\xi=\bm\lambda(\bX),\\ 
		(\u,\bm\xi)\in\cC
	\end{array}
	} , 
  \]
  we have $\conv(\overline{\cA}) = \widehat{\cA}$ and $\cl \conv(\overline{\cA}) =\cA^{\cl}$, where
  \begin{align*}
    \widehat{\cA} &\! \coloneqq \set{ \! (\u,\bm\xi,\bX) \in \cU \times \K_{o,d} \times \S_+^d  :
    \begin{array}{>{\displaystyle}l} 
    \bm\xi\majorize\bm\lambda(\bX),\\
    (\u,\bm\xi)\in\cC
    \end{array}
  \! }, \! \\
     \cA^{\cl} &\! \coloneqq \set{\!  (\u,\bm\xi,\bX) \in \cU \times \K_{o,d} \times \S_+^d : 
     \begin{array}{>{\displaystyle}l} 
      \bm\xi\majorize\bm\lambda(\bX),\\ 
      (\u,\bm\xi)\in\cl(\cC) 
     \end{array}
  \! } \! . 
  \end{align*}
\end{proposition}
\begin{proof}
In general, $\overline{\cA}$ is nonconvex because, as a function, each eigenvalue function $\lambda_i(\bX)$ is a nonconvex function of $\bX$. 
In contrast, {by \cref{lem:majorization-convex-representation},} $\widehat\cA$ is convex.  Moreover, $\overline{\cA}\subseteq \widehat\cA$. Hence, $\conv(\overline{\cA})\subseteq\widehat\cA$. To complete the proof, we will show that $\widehat\cA\subseteq\conv(\overline{\cA})$. Consider any $(\bar\u, \bar{\bm\xi}, \bar\bX)\in\widehat\cA$.  As $\bar{\bm\xi}\majorize\bm\lambda(\bar\bX)$, we have $\Diag(\bar{\bm\xi})\majorize \bar\bX$ and thus by \cref{thm:char:majorization:matrix} 
there exist orthogonal matrices $\bQ^{\ell}\in\R^{d\times d}$ and convex combination weights $\alpha_\ell\ge0$ such that $\sum_{\ell}\alpha_\ell=1$ and $\bar\bX=\sum_{\ell}\alpha_{\ell}(\bQ^{\ell})^\top\Diag(\bar{\bm\xi})\bQ^{\ell}$. Let $\bX^{\ell}:=(\bQ^{\ell})^\top\Diag(\bar{\bm\xi})\bQ^{\ell}$. Then, for all $\ell$, we have $(\bar\u,\bar{\bm\xi},\bX^{\ell})\in\overline{\cA}$ as $\bm\lambda(\bX^\ell)=\bar{\bm\xi}$. Moreover, $(\bar\u,\bar{\bm\xi},\bar\bX) = \sum_{\ell}\alpha_{\ell}(\bar{\u},\bar{\bm\xi},\bX^{\ell})$. Thus, $(\bar\u,\bar{\bm\xi},\bar\bX)\in\conv(\overline{\cA})$, as desired. 

Finally, taking the closure of both sides of $\conv(\overline{\cA})=\widehat{\cA}$, we get  $\cl\conv(\overline{\cA})=\cl(\widehat{\cA}) =\cA^{\cl}$, where the last inequality follows from the definitions of the sets $\widehat{\cA}$ and $\cA^{\cl}$ and the premise that $\cU$ is closed.
\end{proof}

We next recall two classical results from \citet{marshall_inequalities_2011} that link majorization and convex functions. Recall that $f_s(\bm\xi)=\sum_{i\in[s]}\varphi(\xi_i)$, where $\varphi(\cdot)$ is convex, and hence $f_s(\cdot)$ is convex as well.

\begin{lemma}[{{\cite[Proposition I.3.C.1]{marshall_inequalities_2011}}}]\label{lem:monotone}
For any $\bm\xi,\bm\xi'\in\R^d$ such that  $\bm\xi_{1:s}\majorize\bm\xi_{1:s}'$, we have $f_s(\bm\xi)\geq f_s(\bm\xi')$. 
\end{lemma}

We close this section with a fundamental fact on spectral functions. 
Recall that a function  $f:\R^d\to\overline{\R}$ is permutation-symmetric if $f(\x)=f(\bP\x)$ for any permutation matrix $\bP\in\R^{d\times d}$ and any $\x\in\R^d$.
\begin{proposition}[{\citep[Theorem 3.1]{lewis_convex_1996}}]\label{prop:eigenvalueFunctions}\label{lem:conjugate-subdifferential}
Let $f:\R^d\to\overline{\R}$ be a convex and permutation-symmetric function. Then, for any $\bX\in\S^d$, the spectral function $F(\bX) \coloneqq f(\lambda(\bX))$ satisfies 
$ 
F(\bX) = \max_{\bQ\in\R^{n\times n}}\set{f(\diag(\bQ\bX\bQ^\top)) :~ \bQ^\top \bQ = \bI_n}.
$ 
Thus, $F(\bX)$ is a convex function of $\bX\in\S^d$,
{and $\bY\in\partial F(\bX)$ if and only if $\bm\lambda(\bY)\in\partial f(\bm\lambda(\bX))$ and there exists an orthogonal matrix $\bQ$ with $\bQ^\top\bX\bQ=\Diag(\bm\lambda(\bX))$ and $\bQ^\top\bY\bQ=\Diag(\bm\lambda(\bY))$.}
\end{proposition}

\section{Majorization-based convex relaxation}
\label{sec:majorization}

In this section, we will show how majorization can be used as a tool for deriving a convex relaxation for \eqref{eq:mesp-general}. Recall that  \eqref{eq:mesp-general} is given by
\begin{align*}
    &\min_{\x\in\cX\cap\{0,1\}^d} f_s\left( \bm\lambda\left(\bV\Diag(\x)\bV^\top\right) \right) \\
    &= \min_{\x\in\cX\cap\{0,1\}^d} F_s\left(\bV\Diag(\x)\bV^\top\right)  ,
\end{align*}
where $F_s(\bX)\coloneqq f_s(\bm\lambda(\bX))$. Note that the domain of $F_s$ is the set of positive semidefinite matrices $\bX$ with $\rank(\bX) \ge s$ (recall that $\varphi:\R_{++}\to\R$ and $\cX\subseteq \{\x\in[0,1]^d:\, \ones^\top\x=s\}$), and thus it is indeed convex. 
There are two sources of nonconvexities in this problem. 
First, the domain of $\x$, i.e., $\cX\cap\{0,1\}^d$, is nonconvex due to the integrality restriction. Second, even though $f_s(\cdot)$ is a convex function, this function is not  permutation symmetric and as a result the function $F_s(\bX)$ 
is not necessarily  convex either. Indeed, 
\citet[Propositions 1 and 7]{li_best_2023} 
prove the nonconvexity of this function in the cases of $\varphi(\xi)=-\log(\xi)$ and $\varphi(\xi)=\frac{1}{\xi}$, respectively. Our focus will be on addressing this second source of nonconvexity. 

Note that our problem, i.e., \eqref{eq:mesp-general}, can be equivalently written as 
  \begin{align*}
    & \min_{t\in\R,\x\in\cX\cap\{0,1\}^d} \set{t:~ (t,\x)\in\cP(\overline{\Omega}) } 
   , \\\
\text{where }\quad 
    \overline{\Omega} \! & \coloneqq \!\set{\! (\bm\xi,\bX)\in\R^d_+\times\S^d_+:\,\bm\xi = \bm\lambda\left(\bX\right),\ \bm\xi \in\K_{o,d} \! },\!  \\
        \cP (\Omega)\! &\coloneqq \! \set{ \!
      (t,\x) \in \R \times \cX : \!
       \begin{array}{>{\displaystyle}l}
        \exists \bm\xi \in\R^d_+ \textup{ s.t. } t\geq f_s(\bm\xi), \!\! \\
        (\bm\xi, \bV\Diag(\x)\bV^\top) \in\Omega \!
       \end{array} \!
    } \! .
  \end{align*}
Here, the set $\overline{\Omega}$ is introduced based on the fact that the vector of eigenvalues are naturally sorted.   
Note that $\overline{\Omega}$ is nonconvex due to the constraints $ \bm\xi = \bm\lambda\left(\bX\right)$.

Given that $f_s$ is a convex function, it is immediate to observe that the set $\cP (\Omega)$ is convex if and only if the set $\Omega$ is convex. As the set $\overline{\Omega}$ is nonconvex in general, so is $\cP (\overline{\Omega})$.  
In the light of \cref{prop:conv-general}, we  introduce the set
\[
    \widehat{\Omega} \coloneqq \set{ (\bm \xi,\bX)\in\R^d_+\times\S^d_+:\,\bm \xi  \majorize \bm\lambda\left(\bX\right),\ \bm{\xi}\in\K_{o,d} }. 
\]
Note that $ \widehat{\Omega}$ is in fact closed and convex {(see \cref{lem:majorization-convex-representation})}. 
Then, by applying \cref{prop:conv-general} 
with $\cU$ and $\cC$ taken as $\cU=\R^0$ and $\cC=\R^0 \times (\K_{o,d}\cap\R^d_+)$ 
which is closed, we conclude that
\[
\cl\conv(\overline{\Omega}) = \conv(\overline{\Omega}) = \widehat{\Omega}  .
\]
This then leads to the natural convex relaxation $\cP( \widehat{\Omega}) \supseteq \cl \conv\left(\cP(\overline{\Omega}) \right)$, and by dropping the integrality requirement on $\x$ we arrive at the following convex relaxation of \eqref{eq:mesp-general} given by
\begin{align}
&\min_{t\in\R,\x\in\cX} \set{t:~ (t,\x)\in\cP(\widehat{\Omega}) } \notag \\
&= \min_{\x\in\cX,\bm\xi\in\R^d_+} \set{f_s(\bm\xi): \begin{array}{l}
\bm \xi  \majorize \bm\lambda\left(\bV\Diag(\x)\bV^\top\right),\!\!\!\! \\
\bm{\xi}\in\K_{o,d} 		
\end{array}
 }   \label{eq:Gamma_rel} \\
&= \min_{\x\in\cX}\set{g_s(\bm\lambda\left(\bV\Diag(\x)\bV^\top\right) )}
, \notag
\end{align}
where
\begin{align} \label{eq:h-rep:non-sparse}
     g_s(\bm\lambda) = \min_{\bm\xi\in\R_+^d}\set{f_s(\bm\xi):~ \bm\xi\majorize\bm\lambda,\ \bm\xi\in\K_{o,d} }. 
   \end{align}
Note that $\dom(g_s)=\{\bm\lambda\in\R^d_+: \|\bm\lambda\|_0\ge s\}$.

Now note that the majorization relations $\bm\xi\majorize\bm\lambda$ are invariant with respect to permuting the coordinates of $\bm\lambda$. Thus, we immediately observe that $g_s(\cdot)$ is a permutation symmetric function. Moreover, it is convex as we establish next.
\begin{proposition}\label{lem:h-convex}
The function  $g_s$ defined in \eqref{eq:h-rep:non-sparse} is convex. Thus, the function $G_s(\bX)\coloneqq g_s(\bm\lambda(\bX))$, where $\dom(G_s)$ is the set of positive semidefinite matrices $\bX$ with $\rank(\bX)\ge s$, is a convex function of $\bX$. Consequently, $\x \mapsto g_s\left( \bm\lambda\left(\bV\Diag(\x)\bV^\top\right) \right)$ is a convex function of $\x$. 
\end{proposition}
\begin{proof}
As $f_s(\bm\xi)=\sum_{i\in[s]}\varphi(\xi_i)$ and $\varphi$ is convex, the function $\tilde{g}_s(\bm\lambda,\bm\xi) \coloneq f_s(\bm\xi)$  is a convex function of $(\bm\lambda,\bm\xi)$.
  Then, by definition of $g_s$, we have 
 \[
  g_s(\bm\lambda)=\min_{\bm\xi\in\R_+^d}\set{\tilde{g}_s(\bm\lambda,\bm\xi):~  \bm\xi\majorize\bm\lambda,\ \bm\xi\in\K_{o,d} } .
 \] 
Thus, $\epi(g_s)=\left\{ (t,\bm\lambda):\, \exists \bm\xi\in \R^d_+ \textup{ s.t. } t\ge \tilde{g}_s(\bm\lambda,\bm\xi),\  \bm\xi\majorize\bm\lambda,\ \bm\xi\in\K_{o,d}  \right\}$. Moreover, by \cref{lem:majorization-convex-representation}  the constraints $\bm\xi\in\K_{o,d}$ and $\bm\xi\majorize\bm\lambda$ together admit a convex representation. So, $\epi(g_s)$ is the projection of a convex set, and as the projection operation preserves convexity, we conclude $\epi(g_s)$ is convex and so $g_s$ is a convex function. 
 Now, as $g_s$ is a convex permutation-symmetric function, by \cref{prop:eigenvalueFunctions},  $G_s(\bX)=g_s(\bm\lambda(\bX))$ is a convex function of $\bX$. Moreover, $\x\mapsto \bV\Diag(\x)\bV^\top$ is a linear map, and thus the composition of $G_s(\bX)$ with this linear map, i.e., the function $\x\mapsto g_s(\bm\lambda(\bV\Diag(\x)\bV^\top))$ results in a convex function of $\x$, proving the last claim.
\end{proof}

\begin{remark}\label{rem:sparse}
Recall that in our problems of interest the domain for $\x$ is given as $\cX\cap\{0,1\}^d \subseteq \{\x\in\{0,1\}^d:\,  \sum_{i\in[d]} x_i = s\}$. Then, whenever $\x \in \{0,1\}^d$ the requirement $ \sum_{i\in[d]} x_i = s$ enforces that $\rank\left(\bV\Diag(\x)\bV^\top \right) \leq s$ and so  $\left\|\bm\lambda\left(\bV\Diag(\x)\bV^\top \right) \right\|_0 \leq s$ must hold. 
Therefore, one may wonder if there is any benefit to studying the smaller-yet-valid relaxation set 
\[
   \widetilde{\Omega}_s  \coloneqq \set{\! (\bm\xi,\bX)\in\R^d_+\times\S^d_+:
    \begin{array}{>{\displaystyle}l} 
    \bm\xi \majorize \bm\lambda\left(\bX\right),\ \bm\xi \in\K_{o,d},\\
    \|\bm\xi\|_0 \leq s 
    \end{array} \!
    }  
\]
instead of $\widehat{\Omega}$. 
Note also that the seemingly nonconvex requirement $\|\bm\xi\|_0 \leq s$ in $\widetilde{\Omega}_s$ admits a convex representation of the form $\xi_{s+1}=\ldots = \xi_d=0$ due to the presence of the conic constraint $\bm\xi \in\K_{o,d}$ and the sign restrictions $\bm\xi \in\R^d_+$. Hence, we can equivalently write $\widetilde{\Omega}_s$ as the convex set 
\[
\widetilde{\Omega}_s = \set{\! (\bm\xi,\bX)\in\R^d_+\times\S^d_+:
      \begin{array}{>{\displaystyle}l}
      \bm\xi \majorize \bm\lambda\left(\bX\right), \, \bm\xi \in\K_{o,d} ,\\
      \xi_{s+1}=\ldots=\xi_d=0 
      \end{array} \!
}.
\]

When considering $\widetilde{\Omega}_s$, the function analogous to $g_s$ defined in \eqref{eq:h-rep:non-sparse} is given by
\begin{align}\label{eq:h-rep}
  \bar{g}_s(\bm\lambda) \coloneqq \min_{\bm\xi\in\R_+^d}\set{f_s(\bm\xi):
   \begin{array}{>{\displaystyle}l}
    \bm\xi\majorize\bm\lambda,\ \bm\xi\in\K_{o,d},\\ 
    \|\bm\xi\|_0\leq s 
    \end{array}
}. 
\end{align}
We next show that whenever $\varphi$ is nonincreasing, which is our case, there is no benefit in including the additional restriction $\|\bm\xi\|_0 \leq s$.
 That is, whenever $\varphi$ is nonincreasing, we have $ \bar{g}_s(\bm\lambda)=g_s(\bm\lambda)$ for all $\bm\lambda\in\R^d$. 
  To see this, consider any $\bm\lambda\in\R^d$. Then, it is immediate that $\bar{g}_s(\bm\lambda) \geq g_s(\bm\lambda)$. To see the reverse direction, consider any $\bm\xi$ feasible to \eqref{eq:h-rep:non-sparse}, and define $\bar{\bm\xi}$ by 
  \begin{align*}
    \bar\xi_j = \begin{cases}
      \xi_1 + \sum_{i\in(s,d]}\xi_i, & \text{if } j\in\{1\}, \\
      \xi_j, & \text{if } j\in(1,s], \\
      0, & \text{if } j\in(s,d]. 
    \end{cases}
  \end{align*}
  Then, $\|\bar{\bm\xi}\|_0\leq s$ and $\bar{\bm\xi}\majorize\bm\lambda$ which follows from $\bm\xi\majorize\bm\lambda$. Thus, $\bar{\bm\xi}$ is feasible to \eqref{eq:h-rep}. Moreover, as $\varphi(\cdot)$ is nonincreasing and we have $\bar{\bm\xi}_{1:s}\geq\bm\xi_{1:s}$, we conclude that  $f_s(\bar{\bm\xi})\leq f_s(\bm\xi)$. This shows $\bar{g}_s(\bm\lambda) \leq g_s(\bm\lambda)$. 
  
Finally, if we further assume $\varphi$ is strictly decreasing, then for any $\|\bm\xi\|_0>s$ we have $\bar\xi_1>\xi_1$, thus $f_s(\bar{\bm\xi})<f_s(\bm\xi)$. This means any $\bm\xi$ that is not $s$-sparse cannot achieve the minimum in \eqref{eq:h-rep:non-sparse}. 
\end{remark}

We close this section with a simple observation. 
\begin{lemma}\label{lem:g_d:exactness}
When $s=d$, we have $g_s(\bm\lambda)=f_s(\bm\lambda)$ for all $\bm\lambda\in\R^d_+$.
\end{lemma}
\begin{proof}
Given any $\bm\lambda\in\R^d_+$, by definition of $g_s$ in \eqref{eq:h-rep:non-sparse}   we have 
$
     g_d(\bm\lambda) = \min_{\bm\xi\in\R_+^d}\set{f_d(\bm\xi):~ \bm\xi\majorize\bm\lambda,\ \bm\xi\in\K_{o,d} }. 
$ 
In particular, every $\bm\xi$ feasible to this problem satisfies $\bm\xi\majorize\bm\lambda$  and thus \cref{lem:monotone} implies $f_d(\bm\xi)\geq f_d(\bm\lambda)$. Therefore, $\bm\lambda$ is an optimum solution to this problem and $g_d(\bm\lambda)=f_d(\bm\lambda)$ holds. 
\end{proof}

\subsection{Properties of the majorization-based relaxation}\label{sec:majorization:properties}

In this section, we examine some convex hull and exactness properties of certain sets and functions related to  $\widehat{\Omega}$. All proofs for this section are given in  \cref{sec:app:proofs}.

We start by establishing that the function $G_s$ provides the convex envelope of $F_s$, i.e., $G_s$ is the best convex under estimator of $F_s$. 
\begin{corollary}\label{cor:f+Omega_s:conv}
Given $F_s(\bX)=f_s(\bm\lambda(\bX))$ and $G_s(\bX)=g_s(\bm\lambda(\bX))$, we have $\conv(\epi(F_s)) = \epi(G_s)$, where their epigraphs are given by
$ 
\epi(F_s) 
= \{(t,\bX)\in\R\times\S^d:~ \exists \bm\xi\in\R^d \textup{ s.t. } t\ge f_s(\bm\xi),~  (\bm\xi,\bX)\in\overline{\Omega} \}, 
$ 
and 
$ \epi(G_s) 
 = \{(t,\bX)\in\R\times\S^d:~ \exists \bm\xi\in\R^d \textup{ s.t. } t\ge f_s(\bm\xi),~  (\bm\xi,\bX)\in\widehat{\Omega} \} . 
$ 
\end{corollary}

Investigation of the strength of $\cP( \widehat{\Omega})$, and thus of \eqref{eq:Gamma_rel}, as well as its relation to the existing relaxations for MESP and its variants brings some quite natural questions. To this end, our next result states that the integer points of $\cP( \widehat{\Omega})$ and of $\cP( \overline{\Omega})$ match exactly. 
\begin{proposition}\label{prop:P-relax-integrality}\label{prop:relaxation-quality}
$
      \cP(\overline{\Omega}) \cap \left(\R\times\{0,1\}^d\right) 
      = \cP(\widehat{\Omega}) \cap \left(\R\times\{0,1\}^d\right). 
$
Thus, for any $\x\in\cX\cap\{0,1\}^d$, we have $g_s\left( \bm\lambda\left(\bV\Diag(\x)\bV^\top\right) \right) = f_s \left( \bm\lambda\left(\bV\Diag(\x)\bV^\top\right) \right)$.
\end{proposition}

In spite of \cref{prop:relaxation-quality}, even when $\varphi(\cdot)=-\log(\cdot)$, we have $\conv\left(\cP(\overline{\Omega})\cap(\R\times\set{0,1}^d)\right) \neq \cP(\widehat{\Omega} )$ (see \cref{ex:P+Omega_s:not_conv} for a numerical example illustrating this in \cref{sec:app:proofs}). 
To better understand the relaxation defined by $\cP(\widehat{\Omega})$, and to emphasize that the negative example in \cref{ex:P+Omega_s:not_conv} is not an isolated instance, we next examine the lifted sets associated with $\cP(\overline{\Omega})$ and $\cP(\widehat{\Omega})$.
Recall that for any $\Omega\subseteq\R^d\times\S^d$, we have defined 
$
        \cP (\Omega) 
        \coloneqq \set{
      (t,\x) \in \R \times \cX :\ \exists \bm\xi \in\R^d_+ \textup{ s.t. }
        t\geq f_s(\bm\xi),\ 
        (\bm\xi, \bV\Diag(\x)\bV^\top) \in\Omega
    } ,
$ 
and now we also define 
 \begin{align}\label{eq:cQ-Omega}
     \cQ(\Omega) \! &\coloneq \!\set{ \!
      (t,\bm\xi,
      \x)\in\R\times\R_+^d\times
      \cX : 
      \begin{array}{>{\displaystyle}l}
        t\ge f_s(\bm\xi), \\
        (\bm\xi, \bV\Diag(\x)\bV^\top 
        )\in \Omega 
      \end{array} \!
    } . 
 \end{align}
 Then, 
 $\cP(\overline{\Omega}) = \Proj_{t,\x}\left( \cQ(\overline{\Omega}) \right)$ and $\cP(\widehat{\Omega}) = \Proj_{t,\x}\left(\cQ(\widehat{\Omega}) \right)$. 
Although we have established favorable convex hull properties for certain sets associated with $\widehat{\Omega}$, the next result demonstrates that, in general, $\conv(\cQ(\overline{\Omega})) \neq \cQ(\widehat{\Omega})$, and that this discrepancy arises in many settings.
\begin{proposition}\label{prop:Q+Omega_s:not_conv}
For any matrix $\bC$ (and in fact even when $\bC$ is diagonal) with $\rank(\bC)=d$, where $d\ge 2$, and any $s\in\set{2,3,\ldots,d}$, we have $\conv(\cQ(\overline{\Omega})) \neq \cQ(\widehat{\Omega})$. 
\end{proposition}

Finally, we close this section by providing a condition under which the relaxation \eqref{eq:Gamma_rel} is exact. 

\begin{lemma}\label{lem:exact-s-1}
  Suppose that $\bA$ is a totally unimodular matrix and $\b$ is a vector of integers such that $\cX\neq\emptyset$. If $s=1$, then {any optimum solution of \eqref{eq:Gamma_rel} is an optimum solution of the original integer program \eqref{eq:mesp-general} and vice versa.} 
\end{lemma}

%% file: other_relaxations.tex
\subsection{Relaxations from literature}

\subsubsection{The \texorpdfstring{$\Gamma$}{Gamma} (factorization) relaxation}
\label{sec:equivalent}
  
One of the strongest relaxations for MESP  and its variants is the \emph{$\Gamma$ relaxation} (or the \emph{factorization bound}), which is introduced by \cite{nikolov_randomized_2015} and \cite{li_best_2023}.    
To introduce this relaxation, we first recall a preliminary result and define some notation. 
\begin{lemma}{\citep[Lemma 14]{nikolov_randomized_2015}}
  \label{lem:index-k}
  For any $\bm\lambda\in\R^d_+$ 
  and $1\leq s\leq d$, there exists a unique integer $0\leq k<s$ such that 
  \begin{align*}
    \lambda_{[k]} > \frac{1}{s-k}\sum_{i\in(k,d]}\lambda_{[i]} \geq \lambda_{[k+1]}, 
  \end{align*}
  where $\lambda_{[0]}\coloneqq \infty$ by convention. 
\end{lemma}

We define the function $h_s$ over the domain $\R_+^d$ as 
\begin{align}\label{def:h}
  h_s(\bm\lambda) \coloneqq \sum_{i\in[k]} \varphi(\lambda_{[i]})+(s-k)\varphi\left(\frac{1}{s-k}\sum_{i\in(k,d]}\lambda_{[i]}\right), 
\end{align}
where for a given $\bm\lambda\in\R^d_+$ the index $k$ is defined as in \cref{lem:index-k} for this $\bm\lambda$. 
Based on this function, in the case of $\varphi(\xi)=-\log(\xi)$, \cite{nikolov_randomized_2015} and \cite{li_best_2023} defined the \emph{$\Gamma$ relaxation} (or the \emph{factorization bound})  as 
\begin{align}
& \Opt_{\Gamma}\coloneqq \min_{\x\in\cX}\set{h_s\left( \bm\lambda \left(\bV\Diag(\x)\bV^\top\right) \right)  }. \label{eq:opt-gamma} 
\end{align}

Note that by its definition $h_s$ is permutation-symmetric.
Establishing that $h_s(\bm\lambda)$ is a convex function of $\bm\lambda$ requires a few steps due to its reliance on the ordered components $\lambda_{[i]}$, and to the best of our knowledge was not done in the literature. 
In constrast, in the case of $\varphi(\cdot)=-\log(\cdot)$, using a connection to a dual problem   and a subdifferential characterization, \cite[Theorem 15]{nikolov_randomized_2015} established that $h_s(\bm\lambda(\bX))$ is a convex function of $\bX$. Later on, using a Lagrangian dual perspective \cite[Lemma 3 and Theorem 2]{li_best_2023} provided an alternative proof of the convexity of $h_s(\bm\lambda(\bX))$ in $\bX$ for the case of $\varphi(\xi)=-\log(\xi)$; this was also extended to the case of $\varphi(\xi)=\frac{1}{\xi}$ in \cite[{{Theorem 9}}]{li_best_2023}. 
Hence, these results collectively imply that $\x \mapsto h_s\left( \bm\lambda \left(\bV\Diag(\x)\bV^\top\right) \right)$ is a convex function of $\x$ whenever $\varphi(\xi)=-\log(\xi)$ or $\varphi(\xi)=\frac{1}{\xi}$.

We next show that in fact the convex relaxation~\eqref{eq:Gamma_rel} is equivalent to the  $\Gamma$ relaxation given by \eqref{eq:opt-gamma}, and this is so for a broader class of  general convex functions $\varphi(\xi)$ beyond the special cases of $\varphi(\xi)=-\log(\xi)$ or $\varphi(\xi)=\frac{1}{\xi}$. To this end, we establish the following result that provides the closed-form solution for the minimization problem \eqref{eq:h-rep:non-sparse} 
defining $g_s(\bm\lambda)$.

\begin{proposition}\label{prop:h-rep}
  Consider any $\bm\lambda\in\R_{+}^d\cap\K_{o,d}$. Define 
  \begin{align}\label{eq:optimal-xi}
    \lambda_i' \coloneq \begin{cases}
      \lambda_i, & \textup{if } i\in[k], \\
      \frac{1}{s-k}\sum_{j\in(k,d]}\lambda_j, & \textup{if } i\in(k,s], \\
      0, & \textup{if } i\in(s,d], 
    \end{cases}
  \end{align}
  where $k$ is the index as defined in \cref{lem:index-k}.
  Then, {$\bm\lambda'\majorize\bm\lambda$, and} for any $\bm\xi\in\R^d_+\cap\K_{o,d}$ such that $\|\bm\xi\|_0\leq s$ and $\bm\xi\majorize \bm\lambda$, we have $\bm\xi\majorize \bm\lambda'$ as well, which implies $f_s(\bm\xi)\geq f_s(\bm\lambda')$.
  Consequently, for the given $\bm\lambda\in\R_{+}^d\cap\K_{o,d}$, the vector $\bm\lambda'$ is an optimum solution to the optimization problem \eqref{eq:h-rep:non-sparse} 
   defining $g_s$, and $g_s(\bm\lambda)=f_s(\bm\lambda')=h_s(\bm\lambda)$. 
\end{proposition}
\begin{proof}
  Clearly, $\bm\lambda'$ is $s$-sparse and by the choice of $k$ we also have $\R_{+}^d\cap\K_{o,d}$. Moreover, $\bm\lambda'\majorize\bm\lambda$ follows from the definition of majorization and the choice of $k$. Thus, $\bm\lambda'$ is feasible to \eqref{eq:h-rep:non-sparse}. 
  
  Next, consider any $s$-sparse vector $\bm\xi$ that is feasible to \eqref{eq:h-rep:non-sparse}. We claim that $\bm\xi\majorize\bm\lambda'$, or equivalently $\bm\xi_{1:s}\majorize\bm\lambda'_{1:s}$ as $\bm\xi$, $\bm\lambda'$ are both $s$-sparse. To see this, first note that as both $\bm\xi$ and $\bm\lambda'$ majorize $\bm\lambda$ and are both $s$-sparse, and so $\sum_{i\in[s]}\xi_i=\sum_{i\in[d]}\lambda_i=\sum_{i\in[s]}\lambda_i'$. Also, {from the majorization relation $\bm\xi\majorize\bm\lambda$ and the definition of $\bm\lambda'$,} for any $j\in[k]$, we have $\sum_{i\in[j]}\xi_i\geq\sum_{i\in[j]}\lambda_i=\sum_{i\in[j]}\lambda_i'$. 
  {
  Hence, we deduce $\sum_{i\in[k+1,s]}\xi_i \leq \sum_{i\in[k+1,s]}\lambda_i'$. Therefore, for any $j\in[k,s-1]$, we have
    \begin{align*}
    \frac{1}{s-j} \sum_{i\in[j+1,s]}\lambda_i' &= \frac{1}{s-k} \sum_{i\in[k+1,s]}\lambda_i' \\
    & \ge \frac{1}{s-k} \sum_{i\in[k+1,s]} \xi_i \\ 
    & \ge \frac{1}{s-j} \sum_{i\in[j+1,s]} \xi_i ,
    \end{align*}
  where the equality follows from the definition of $\bm\lambda'$, and the last inequality holds because $\bm\xi\in\K_{o,d}$. Then, once again using $\sum_{i\in[s]}\xi_i=\sum_{i\in[s]}\lambda_i'$ on this relation, we conclude that
      \begin{align*}
 	\sum_{i\in[j]}\lambda_i' & \leq \sum_{i\in[j]}\xi_i,~\forall j\in[k,s-1],  
      \end{align*}
      i.e., the remaining majorization relations also hold.
    }
  Therefore, we have shown $\bm\xi_{1:s}\majorize\bm\lambda'_{1:s}$ as desired.   
  By \cref{lem:monotone}, the relation $\bm\xi_{1:s}\majorize\bm\lambda'_{1:s}$ implies $f_s(\bm\xi)\geq f_s(\bm\lambda')$. 
  
    Finally, as $\varphi$ is a nonincreasing function, by \cref{rem:sparse} 
   without loss of generality we can assume that the optimum solution of   \eqref{eq:h-rep:non-sparse} is $s$-sparse.
Since the relation $f_s(\bm\xi)\geq f_s(\bm\lambda')$ holds for any $s$-sparse vector $\bm\xi$ feasible to \eqref{eq:h-rep:non-sparse}, we conclude that $\bm\lambda'$ is optimal to \eqref{eq:h-rep:non-sparse}, and $g_s(\bm\lambda)=f_s(\bm\lambda')=h_s(\bm\lambda)$. 
\end{proof}

\begin{corollary}\label{cor:relaxation-h}
  The function $\x \mapsto h_s \left( \bm\lambda\left(\bV\Diag(\x)\bV^\top\right) \right)$ is a convex function of $\x$, and the convex relaxations given by \eqref{eq:Gamma_rel} and \eqref{eq:opt-gamma}  are the same, i.e., 
  \begin{align}\label{eq:opt-gamma:equivalence}
    \min_{\x\in\cX} g_s\left( \bm\lambda\left(\bV\Diag(\x)\bV^\top\right) \right) 
    &=  \min_{\x\in\cX} h_s \left( \bm\lambda\left(\bV\Diag(\x)\bV^\top\right) \right) = \Opt_{\Gamma}. 
  \end{align}
\end{corollary}

Note that the definition of $h_s$ depends on a special index $k$, which complicates both its analysis and that of the $\Gamma$ (factorization) relaxation. In contrast, the equivalent representation of $h_s$ given by $g_s$ in \cref{prop:h-rep} and \cref{cor:relaxation-h} allows us to easily conclude that $h_s$ is convex and also yields a cleaner representation of $\Opt_{\Gamma}$ given by the majorization-based relaxation \eqref{eq:Gamma_rel}. As shown later in \cref{sec:connection}, using both representations interchangeably enables us to study the strengths of several well-known MESP relaxations effectively.

\subsubsection{Linx relaxation}\label{sec:linx}

The most well-studied case of \eqref{eq:mesp-general} is when $\varphi(\cdot)=-\log(\cdot)$. 
In this setting, 
{
exploiting the following simple determinant identity which holds for any $\x\in\cX\cap\{0,1\}^d$  and $S\coloneqq\set{i\in[d]:\, x_i=1}$,
 \begin{align}
-\log\det(\bC_{S,S}) =
-\sum_{i\in[s]} \log\left( \lambda_i( \Diag(\x)\bC\Diag(\x) )\right) = -\frac12 \log\det\left(\bC\Diag(\x)\bC + \bI_d - \Diag(\x)\right) , 
\label{eq:linx:relaxation-relation}
\end{align}
}
\citep{anstreicher_efficient_2020} suggested the following high quality convex relaxation called \emph{linx}: 
\[
\Opt_{\textup{linx}}\coloneqq \min_{\x\in\cX} \set{-\frac12 \log\det\left(\bC\Diag(\x)\bC + \bI_d - \Diag(\x)\right) } .
\]

Whenever $\varphi(\cdot)=-\log(\cdot)$,  for any $\bZ\in\S^d_+$, we have $f_d(\bm\lambda(\bZ)) = \sum_{i\in[d]} \varphi(\lambda_i(\bZ)) = -\sum_{i\in[d]} \log(\lambda_i(\bZ)) = -\log \det (\bZ)$. Thus, together with \cref{lem:g_d:exactness} we conclude that
 \begin{align}\label{eq:linx}
   \Opt_{\textup{linx}} = \min_{\x\in\cX}\set{\frac12\, g_d\left(\bm\lambda\left(\bC\Diag(\x)\bC + \bI_d - \Diag(\x)\right)\right)}. 
  \end{align}
  {
  As the linx relaxation is due to the determinant identity \eqref{eq:linx:relaxation-relation}, we immediately observe that it is exact at the integer points, i.e., for any $\bar\x\in\cX\cap\{0,1\}^d$ and when $\varphi(\cdot)=-\log(\cdot)$, we have $\frac12\, g_d\left(\bm\lambda\left(\bC\Diag(\bar\x)\bC + \bI_d - \Diag(\bar\x)\right)\right)  =-\sum_{i\in[s]}\log(\lambda_i(\Diag(\bar\x)\bC\Diag(\bar\x)) = f_s \left( \bm\lambda\left(\bV\Diag(\bar\x)\bV^\top\right) \right)$.
  }

\subsubsection{Complementary problem}\label{sec:complementary}

We next introduce the notion of the complementary problem in the case of $\varphi(\cdot)=-\log(\cdot)$ and $\bC$ is invertible. 
The complementary problem was first introduced by \cite{anstreicher_continuous_1996,anstreicher_using_1999} to derive lower bounds on $\Opt(\bC,s)$ via its complementary counterpart. 
The main idea arises from a fundamental relationship between the determinants of certain principal submatrices of $\bC$ and  $\bC^{-1}$: For an invertible matrix $\bC\in\S_{++}^d$, the determinant of its principal submatrix $\bC_{S,S}$ is linked to that of the complementary principal submatrix $(\bC^{-1})_{[d]\setminus S,[d]\setminus S}$ of the inverse matrix $\bC^{-1}$ through the following identity (see \cite[Section 0.8.4]{horn_matrix_1985})
\begin{align*}
  \log\det(\bC_{S,S}) = \log\det(\bC) + \log\det((\bC^{-1})_{[d]\setminus S,[d]\setminus S}). 
\end{align*}
Based on this observation, when $\varphi (\cdot)=-\log(\cdot)$ 
and $\cX=\set{\x\in[0,1]^d:\ones^\top\x=s}$, 
the objective value of \eqref{eq:mesp-general} also satisfies  
\begin{align*}
  \Opt(\bC,s) = \Opt(\bC^{-1},d-s) - \log\det(\bC). 
\end{align*}
As a result, we can derive lower bounds for $\Opt(\bC,s)$ from any lower bound on $\Opt(\bC^{-1},d-s)$ adjusted by the constant $- \log\det(\bC)$. 
Formally, let $\cX_{d-s}\coloneqq\{\check\x\in[0,1]^d:\ones^\top \check\x=d-s\}$ (by properly defining the set $\cX_{d-s}$ we can also handle the case when $\cX=\set{\x\in[0,1]^d:\ones^\top\x=s,\, \bA \x \le \b}$ as well, but to ease the exposition we are leaving this discussion out) and $\bC^{-1}=\bW^\top\bW\in\S^d_{+}$ where $\bW=(\w_1,\dots,\w_d)\in\R^{m\times d}$, then the \emph{complementary problem} for the given data $(\bC,s)$ is given by
  \begin{align*}
    \Opt(\bC^{-1},d-s) &:= \min_{S\subseteq[d],|S|=d-s} \set{-\log\det((\bC^{-1})_{S,S})} \\
    &= \min_{\check\x\in\set{0,1}^d} \set{-\sum_{i\in[d-s]}\log(\lambda_i(\bW\Diag(\check\x)\bW^\top)): \ones^\top \check\x=d-s}. 
  \end{align*}

Now, since $\Opt(\bC^{-1},d-s)$ has the same structure as $\Opt(\bC,s)$, we can apply any of the relaxation techniques developed for $\Opt(\bC,s)$ to $\Opt(\bC^{-1},d-s)$ as well. In particular, the lower bound on  $\Opt(\bC,s)$ derived from applying the $\Gamma$ relaxation to the complementary problem is given by 
\begin{align}\label{eq:opt-gamma-complementary} 
 \Opt_{\Gamma^c} &\coloneqq \min_{\check\x\in\cX_{d-s}}\set{g_{d-s}\left( \bm\lambda \left(\bW\Diag(\check\x)\bW^\top\right) \right) }  -\log\det(\bC) , 
\end{align}
and the lower bound derived from the linx relaxation applied to the complementary problem is
  \begin{align}\label{eq:linx-complementary}
  \Opt_{\textup{linx}^c} &\coloneqq \min_{\check\x\in\cX_{d-s}} \left\{-\frac12\log\det(\bC^{-1}\Diag(\check\x)\bC^{-1}+(\bI_d-\Diag(\check\x))) \right\} - \log\det(\bC) . 
 \end{align}
 While \citep[Lemma 2]{anstreicher_efficient_2020} show that, when $\cX=\set{\x\in[0,1]^d:\ones^\top\x=s}$, the linx relaxation is invariant to the complementary technique, i.e., $ \Opt_{\textup{linx}} = \Opt_{\textup{linx}^c}$, 
in general, the bounds obtained from the $\Gamma$ relaxation applied to MESP and its complementary problem, i.e., $\Opt_{\Gamma}$ and $\Opt_{\Gamma^c}$, differ, even for this choice of $\cX$.

%% file: scaling.tex
\section{Scaling methods for linx relaxation}
\label{sec:scaling}

{
From this section onward, we focus on the case $\varphi(\cdot)=-\log(\cdot)$. 
}

\subsection{Ordinary and generalized scaling methods}\label{sec:o-g-scaling}

\citep{anstreicher_efficient_2020} introduced a \emph{scaling} technique based on the observation that the MESP objective  is not invariant under positive scaling of $\bC$. Indeed, for any $\gamma\in\R_{++}$ and any $S$ satisfying $S\subseteq[d]$, we have $[\gamma\bC]_{S,S}=\gamma \bC_{S,S}$  as well as $\lambda_i(\gamma \bC) =\gamma \lambda_i(\bC)$ for all $i\in[d]$, and thus  $\log \det \left( \bC_{S,S} \right) = \log \det \left( [\gamma \bC]_{S,S} \right)  - s \log(\gamma)$.
Consequently, we deduce that for every $\gamma\in\R_{++}$ we have 
\[
\Opt(\bC,s) = \Opt(\gamma\bC,s) + s \log(\gamma) , 
\]
where recall that $ \Opt(\bC,s)= \min_{S\subseteq[d],|S|=s} \set{-\log\det(\bC_{S,S})}$ is the optimum objective value of the MESP problem \eqref{eq:mesp-general} associated with the matrix $\bC$, $\varphi(\cdot)=-\log(\cdot)$, and $\cX=\set{\x\in[0,1]^d:\ones^\top\x=s}$. 

By a change of variables of $\gamma=\gamma_0^{1/2}$ for some $\gamma_0\in\R_{++}$, we obtain 
$-\log \det \left( \bC_{S,S} \right) 
= -\log \det \left( [\gamma_0^{1/2} \bC]_{S,S} \right) + {1\over 2} s \log(\gamma_0)$. 
Also, by setting $x_i=1$ if $i\in S$ and $x_i=0$ otherwise and 
recalling the linx relaxation {determinant identity~\eqref{eq:linx:relaxation-relation},} 
we arrive at  
\begin{align*}
-\log\det(\bC_{S,S}) & =  -\log\det([\gamma_0^{1/2}\bC]_{S,S}) + {1\over 2} s \log(\gamma_0) \\ & {=}  -\frac12 \log\det\left(\gamma_0\bC\Diag(\x)\bC + \bI_d - \Diag(\x)\right) + {1\over 2} s \log(\gamma_0) .
\end{align*}

This observation motivates the \emph{scaling} enhancement of the linx relaxation introduced by \citep{anstreicher_efficient_2020}, which optimizes over the scaling parameter $\gamma_0\in\R_{++}$ to maximize the relaxation bound obtained by applying  linx  to the scaled matrix $\gamma_0^{1/2}\bC$: 
\begin{align}
&\Opt_{\linxo}\coloneqq \min_{\x\in\cX} \max_{\gamma_0\in\R_{++}} \set{ \objlolog(\x,\gamma_0) },  \label{eq:o-scaling} \\
\text{where } & \objlolog(\x,\gamma_0) \coloneqq  -\frac12\log\det \left(\gamma_0\bC\Diag(\x)\bC+(\bI_d-\Diag(\x)) \right) \vphantom{\sum_{i\in[d]}} 
+ \frac12 s\log(\gamma_0). \label{eq:o-scaling:f}
\end{align}
We refer to the formulation in \eqref{eq:o-scaling} as the \emph{o-scaling} of the linx relaxation. 
It is well known that the $\Gamma$ relaxation is invariant under o-scaling \citep[Theorem 3.4.14]{fampa_maximum-entropy_2022}. 
Moreover, as discussed in \cref{sec:complementary}, the linx relaxation applied to the complementary problem coincides with the linx relaxation applied to the original problem, i.e., $\Opt_{\textup{linx}} = \Opt_{\textup{linx}^c}$. In addition, 
the same invariance holds under o-scaling as well \cite[Proposition 3.3.16]{fampa_maximum-entropy_2022}.

More recently, \citep[Section 3]{chen_generalized_2024} introduced a more general  scaling technique based on the following observation, which relies on the well-known Schur complement determinant formula and extends  \citep[Lemma 1]{anstreicher_efficient_2020}.
\begin{lemma}[{{\citep[Proof of Theorem 2]{chen_generalized_2024}}}]\label{lem:linx:main-idea}
 Let $\bQ\in\R^{d \times d}$ (not necessarily symmetric). For any nonempty set $S\subseteq[d]$, and by letting $x_i=1$ if $i\in S$ and $x_i=0$ otherwise,  we have 
  \begin{align*}
    \log\det(\bQ_{S,S}\bQ_{S,S}^\top) =  \log\det(\bQ\Diag(\x)\bQ^\top + \bI_d - \Diag(\x)) . 
  \end{align*}
\end{lemma}
Then, for any $\bm\gamma\in\R^d_{++}$, by considering the scaled matrix $\bQ\coloneqq \Diag(\bm\gamma)\bC$ in \cref{lem:linx:main-idea}, we arrive at
\[
    \log\det\left(\Diag(\bm\gamma)_{S,S}\bC_{S,S}^2\Diag(\bm\gamma)_{S,S} \right) =  \log\det\left(\Diag(\bm\gamma)_{S,S}\bC_{S,S}\Diag(\x)\bC_{S,S} \Diag(\bm\gamma)_{S,S}+ \bI_d - \Diag(\x)\right) . 
\]
 Now, by the determinant product rule, we have $\log\det\left(\Diag(\bm\gamma)_{S,S}\bC_{S,S}^2\Diag(\bm\gamma)_{S,S} \right) =  2\log\det(\bC_{S,S}) +2 \sum_{i\in S} \log(\gamma_i) $. Combining these two leads to
\[
\log\det(\bC_{S,S}) = \frac12 \log\det\left(\Diag(\bm\gamma)_{S,S}\bC_{S,S}\Diag(\x)\bC_{S,S} \Diag(\bm\gamma)_{S,S}+ \bI_d - \Diag(\x)\right) -   \sum_{i\in S} \log(\gamma_i). 
\]
Once again, to obtain the best bound for $\Opt(\bC,s)$ we can consider  the linx relaxation applied to the scaled matrix $\Diag(\bm\gamma)\bC$ and optimize over the scaling vector $\bm\gamma\in\R^d_{++}$. 
The resulting formulation is given by  
\begin{align}
&\Opt_{\linxg} \coloneqq \min_{\x\in\cX} \max_{\bm\gamma\in\R^d_{++}}  \set{ \objlglog(\x,\bm\gamma) }, \label{eq:g-scaling} \\
\text{where } & \objlglog(\x,\bm\gamma) \coloneqq -\frac12\log\det \left(\Diag(\bm\gamma)\bC\Diag(\x)\bC\Diag(\bm\gamma) + (\bI_d-\Diag(\x)) \right) + \sum_{i\in[d]}x_i\log\gamma_i. \label{eq:g-scaling:f}
\end{align}
We refer to  \eqref{eq:g-scaling} as \emph{g-scaling} of the linx relaxation. 
Taking $\bm\gamma=\gamma_0^{1/2}\ones$ shows that o-scaling is a special case of g-scaling, and thus $\Opt_{\linxg} \ge \Opt_{\linxo}$.
\cite[Appendix]{chen_generalized_2024} shows on a small numerical instance that linx relaxation enhanced by g-scaling can strictly outperform the one with o-scaling, i.e., $\Opt_{\linxg} > \Opt_{\linxo}$, and also quite notable improvements were reported in the numerical study in \cite[Section 6]{chen_generalized_2024}, especially for larger instances.

Finally, analogous to $\Opt_{\textup{linx}}$ and $\Opt_{\linxo}$, the g-scaled linx bound $\Opt_{\linxg}$  is invariant under complementation \citep[page 32]{chen_generalized_2024}.

\subsection{Double-scaling method}
\label{sec:linx-double-scaling}
In this section we introduce a new scaling technique for the linx relaxation that we refer to as \emph{double-scaling}. 

{
Let $\bD,\hat{\bD}$ be diagonal positive definite matrices. For any nonempty set $S\subseteq [d]$, we have
\begin{align}
   \log \det\left( [\bD \bC \hat{\bD}]_{S,S} [\bD \bC \hat{\bD}]_{S,S}^\top \right) 
  &= \log\det \left( \bD_{S,S} \bC_{S,S} \hat{\bD}_{S,S}^2 \bC_{S,S} \bD_{S,S}\right) \notag \\
  &= 2 \log \det \left( \bC_{S,S} \right) + 2\log \det \left( \bD_{S,S} \right)  + 2\log \det \left( \hat{\bD}_{S,S} \right), \label{eq:d-scaling-derivation1}
\end{align}
where 
the first 
equality follows from the observation that $[\bD \bC \hat{\bD}]_{S,S} = \bD_{S,S} \bC_{S,S} \hat{\bD}_{S,S}$ for diagonal matrices $\bD,\hat{\bD}\in\S_+^d$, 
and the last one from the multiplicativity of the determinant. 
By defining $x_i=1$ if $i\in S$ and $x_i=0$ otherwise and using \cref{lem:linx:main-idea} with $\bQ=\bD \bC \hat{\bD}$, we also observe 
\begin{align}
 &\log \det\left( [\bD \bC \hat{\bD}]_{S,S} [\bD \bC \hat{\bD}]_{S,S}^\top \right)  \notag \\
& = \log\det \left( [\bD \bC \hat{\bD}]\Diag(\x)[\bD \bC \hat{\bD}]^\top  + \bI_d - \Diag(\x) \right) \notag \\
& =  \log\det \left( \bD \bC \hat{\bD}^2 \Diag(\x) \bC \bD   + \bI_d - \Diag(\x) \right)  \notag  \\
& = \log\det \left( \bC \hat{\bD}^2 \Diag(\x) \bC    + \bD^{-2} (\bI_d - \Diag(\x) ) \right) +  2 \log\det(\bD)  ,  \label{eq:d-scaling-derivation2}
\end{align}
where the second equality is just rearrangement due to the matrices $\hat{\bD}$ and $\Diag(\x)$ both being diagonal,
and the last relation again follows from the property of the determinant. 
Putting together \eqref{eq:d-scaling-derivation1} and \eqref{eq:d-scaling-derivation2} leads to
\begin{align*}
& \log \det \left( \bC_{S,S} \right) \\
& = {1\over 2} \log\det \left( \bC \hat{\bD}^2 \Diag(\x) \bC    + \bD^{-2} (\bI_d - \Diag(\x) ) \right) +  \log\det(\bD) - \log \det \left( \bD_{S,S} \right)  - \log \det \left( \hat{\bD}_{S,S} \right) .
\end{align*}

Now set $\bD\coloneqq (\Diag(\bm\mu))^{-1/2},\hat{\bD}\coloneqq (\Diag(\bm\gamma))^{1/2}$ where $\bm\gamma,\bm\mu\in\R^d_{++}$. Then, for any $\x\in\cX\cap\{0,1\}^d$ and $\bm\gamma,\bm\mu\in\R^d_{++}$, the MESP objective $-\log\det(\bC_{S,S})$ is equal to the function $\objldlog(\x,\bm\gamma,\bm\mu)$ defined as
}
\begin{align}
\objldlog(\x,\bm\gamma,\bm\mu) 
& \coloneq 
 -\frac12\log\det(\bC\Diag(\bm\gamma)\Diag(\x)\bC+\Diag(\bm\mu)(\bI_d-\Diag(\x))) \vphantom{\sum_{i\in[d]}} \notag \\
 &\qquad\qquad\qquad + \frac12\sum_{i\in[d]}x_i\log\gamma_i + \frac12\sum_{i\in[d]}(1-x_i)\log\mu_i  .  \label{eq:double-scaling} 
\end{align}
This motivates the double-scaling for linx relaxation, which is defined as
\[
\Opt_{\linxd} \coloneq \min_{\x\in\cX} \max_{\bm\gamma,\bm\mu\in\R^d_{++}} \set{ \objldlog(\x,\bm\gamma,\bm\mu) }.
\]

The preceding discussion immediately shows that the double-scaling for linx \eqref{eq:double-scaling}  is a valid relaxation for MESP  \eqref{eq:mesp-general}  and the objective function values of \eqref{eq:mesp-general} and \eqref{eq:double-scaling}  at integer feasible solutions match as well. We state this formally in \cref{prop:double-scaling-valid}. 
\begin{proposition}\label{prop:double-scaling-valid}
  Double-scaling for linx \eqref{eq:double-scaling} is a valid relaxation of MESP \eqref{eq:mesp-general} where $\varphi(\cdot)=-\log(\cdot)$. Moreover, at any binary $\x\in\set{0,1}^d$ and for any $\bm\gamma,\bm\mu\in\R^d_{++}$, we have $\objldlog(\x,\bm\gamma,\bm\mu) = f_s\left( \bm\lambda\left(\Diag(\x)\bC\Diag(\x)\right) \right) = f_s\left( \bm\lambda\left(\bV\Diag(\x)\bV^\top\right) \right)$ where $\bC=\bV^\top\bV$,   
  (that is, the objective function values of \eqref{eq:double-scaling} and \eqref{eq:mesp-general} are equal). 
\end{proposition}

\begin{remark}\label{rem:scaling-comparisons}
Both o-scaling \citep{anstreicher_efficient_2020} and g-scaling \citep{chen_generalized_2024} methods for linx, given respectively in \eqref{eq:o-scaling} and \eqref{eq:g-scaling}, are special cases of double-scaling method. 
In particular, by setting $\bm\gamma=\gamma_0\ones$, $\bm\mu=\ones$ for some $\gamma_0\in\R_{++}$, we observe that for any $\x\in\cX$ (and thus $\ones^\top \x=s$) the double-scaling objective function reduces to the o-scaling one, i.e., $\objldlog(\x,\gamma_0\ones,\ones)=\objlolog(\x,\gamma_0)$ (see \eqref{eq:double-scaling} and \eqref{eq:o-scaling:f}).

To see that double-scaling also generalizes g-scaling,  by using the multiplicative prperty of the determinant 
and rearranging terms in \eqref{eq:double-scaling}, we first observe that  
\begin{align}\label{eq:double-scaling-alt} 
\objldlog(\x,\bm\gamma,\bm\mu) 
& =
  -\frac12 \log \det\left( \Diag(\bm\mu)^{-1/2} \bC \Diag(\bm\gamma) \Diag(\x) \bC  \Diag(\bm\mu)^{-1/2}  + \bI_d - \Diag(\x)  \right) \notag \\
 &\qquad \qquad\qquad + {1\over 2} \sum_{i\in [d]} x_i \log \gamma_i  - {1\over 2} \sum_{i\in [d]} x_i \log \mu_i .
\end{align}
Now, by setting $\bm\gamma=\ones$ and $\bm\mu$ such that $\mu_i=\Upsilon^{-2}_i$ for some $\bm\Upsilon\in\R_{++}^d$, we have $(\Diag(\bm\mu))^{-1/2} = \Diag(\bm\Upsilon)$, $\log\mu_i = - 2 \log \Upsilon_i$ and so 
\begin{align*}
& \objldlog(\x,\bm\gamma,\bm\mu)  
=  \objldlog(\x,\ones,\bm\Upsilon^{-2})  \notag \\
  &\qquad = -\frac12\log\det\left( \Diag(\bm\Upsilon)\bC\Diag(\x)\bC\Diag(\bm\Upsilon)+\bI_d-\Diag(\x) \right) + \sum_{i\in[d]}x_i\log\Upsilon_i
  =  \objlglog(\x,\bm\Upsilon)
  , 
  \end{align*}
where the last equality follows from \eqref{eq:g-scaling:f}. 
Thus, double-scaling of the linx relaxation yields a bound at least as strong as its g-scaling, i.e., $\Opt_{\linxd} \ge \Opt_{\linxg}$. Recall also that $\Opt_{\linxg} \ge \Opt_{\linxo}$ and  g-scaling can strictly improve upon o-scaling in certain cases, we obtain the hierarchy 
\[
\Opt_{\linxd} \ge \Opt_{\linxg} \ge \Opt_{\linxo} \ge \Opt_{\textup{linx}}.
\] 
Our numerical results in \cref{sec:numerical} confirm that these inequalities can be strict, with substantial performance differences among the scaling variants. 
\end{remark}

We next observe that, when $\bC$ is full rank, applying double-scaling to the linx relaxation of the complementary problem (given in \eqref{eq:linx-complementary}) yields no further improvement: it produces the same bound as the double-scaled linx relaxation itself; see \cref{sec:linx-complementary} for the proof. 

\begin{proposition}\label{prop:linx-complementary}
If $\bC$ is full rank, then applying double-scaling to the linx relaxation yields exactly the same formulation as applying double-scaling to the linx relaxation of the complementary problem; that is, both lead to \eqref{eq:double-scaling}.
\end{proposition}

\subsection{Convex-concave saddle point formulation for double-scaled linx relaxation}\label{sec:d-scaling:convex-concave}

{In this section, we will establish that the double-scaled linx relaxation admits a convex-concave saddle point formulation (see \cref{prop:d-scaling-linx-convexity}). To do so, we will first provide a general structural result in \cref{lem:gamma-g-scaling-hessian}, which establishes the convexity of a composition function involving the matrix function $G_s(\bX)\coloneqq g_s(\bm\lambda(\bX))$, where $g_s$ is defined in \eqref{eq:h-rep:non-sparse}. }

A key ingredient for \cref{lem:gamma-g-scaling-hessian} is the characterization of the subdifferential of $G_s$ when $\varphi(\cdot)=-\log(\cdot)$, which was established in \cite[Proposition 2]{li_best_2023}. To simplify the notation, we introduce a linear mapping $\R_+^d\ni\bm\lambda\mapsto\bm\beta(\bm\lambda)$: 
\begin{align}\label{eq:beta-map}
  \beta_\ell = \beta_\ell(\bm\lambda) \coloneq \begin{cases} 
    \lambda_\ell, &\text{if }\ell\in[k], \\ 
    \frac{1}{s-k}\sum_{\iota\in[k+1,d]}\lambda_{\iota}, &\text{if }\ell\in[k+1,d],  
  \end{cases} 
\end{align}
where $k$ is the unique index determined by $\bm\lambda$ according to \cref{lem:index-k}. We use a slightly different notation from \cite{li_best_2023}, specifically, our $\beta_\ell$ is the reciprocal of theirs, to make the relationship between $\bm\lambda$ and $\bm\beta$ more transparent.

\begin{theorem}\label{lem:gamma-g-scaling-hessian}
{Suppose $\varphi(\cdot)=-\log(\cdot)$.}  
  Given $\bS_i\succeq0$ for $i\in[m]$ such that $\sum_{i\in[m]}\bS_i\succ0$, define the function $\Psi(\bm\nu)$ as  
  \begin{align}\label{eq:gs-sum-exp}
    \Psi(\bm\nu) \coloneqq -G_s\left(\sum_{i\in[m]}e^{\nu_i}\bS_i\right). 
  \end{align}
  If $\bm\nu\in\R^m$ is such that the matrix $\bM(\bm\nu)\coloneqq\sum_{i\in[m]}e^{\nu_i}\bS_i$ has distinct eigenvalues, then the Hessian $\nabla^2\Psi(\bm\nu)\succeq0$. Moreover, if there exists $\bar{\bm\nu}\in\R^m$ such that $\bM(\bar{\bm\nu})$ has distinct eigenvalues, then $\Psi$ is convex over $\R^m$. 
\end{theorem}

\begin{proof}
  {As $\bS_i\succeq0$ for $i\in[m]$ and $\sum_{i\in[m]}\bS_i\succ0$,  we immediately deduce that $\bM(\bm\nu)=\sum_{i\in[m]}e^{\nu_i}\bS_i\succ0$ for any $\bm\nu\in\R^d$. 
  
  Let us first assume that $\bM(\bm\nu)$  has distinct eigenvalues.
}
  Let $\bQ\Diag(\bm\lambda)\bQ^\top=\bM(\bm\nu)$ be the eigenvalue decomposition of $\bM(\bm\nu)$ with orthogonal matrix $\bQ$ and eigenvalues $\bm\lambda=(\lambda_1,\dots,\lambda_d)$ in {sorted in a decreasing} order and corresponding eigenvectors $\q_1,\ldots,\q_d$. In addition, define $\bm\beta\coloneqq\bm\beta(\bm\lambda)$ as in \eqref{eq:beta-map}. 
  Since $\lambda_1>\cdots>\lambda_d>0$, by its definition, the index $k$ remains constant for all $\bm\nu'$ in a small neighborhood of $\bm\nu$. 
  By \cite[Proposition 2]{li_best_2023}, the first derivative of $\Psi$ with respect to $\nu_i$ is given by 
  \begin{align*}
    \frac{\partial\Psi}{\partial\nu_i} 
    = \left\langle\frac{\partial G_s}{\partial\bM}, \frac{\partial\bM}{\partial\nu_i}\right\rangle
    &= \left\langle\bQ\Diag(\bm\beta)^{-1}\bQ^\top, e^{\nu_i}\bS_i\right\rangle 
    = \left\langle\sum_{\ell\in[d]}\frac{1}{\beta_\ell}\q_\ell\q_\ell^\top, e^{\nu_i}\bS_i\right\rangle 
    = \sum_{\ell\in[d]}\frac{1}{\beta_\ell}\q_\ell^\top(e^{\nu_i}\bS_i)\q_\ell. 
  \end{align*}
  
  According to \cite[Theorem 6.3.12]{horn_matrix_1985}, 
  \begin{align}\label{eq:grad-eigenvalue}
    \frac{\partial\lambda_\ell}{\partial\nu_j} 
    = \q_\ell^\top\frac{\partial\bM}{\partial\nu_j}\q_\ell
    = \q_\ell^\top(e^{\nu_j}\bS_j)\q_\ell. 
  \end{align}
  According to \cite{magnus_differentiating_1985}, 
  \begin{align}\label{eq:grad-eigenvector}
    {\frac{\partial\q_\ell}{\partial\nu_j}} 
    = \sum_{\iota\in[d],\iota\neq\ell}\frac{1}{\lambda_\ell-\lambda_\iota}\q_\iota\q_\iota^\top\frac{\partial\bM}{\partial\nu_j}\q_\ell
    = \sum_{\iota\in[d],\iota\neq\ell}\frac{1}{\lambda_\ell-\lambda_\iota}\q_\iota\q_\iota^\top(\e^{\nu_j}\bS_j)\q_\ell. 
  \end{align}
  Thus, the second derivative of $\Psi$ can be computed as 
  \begin{align*}
    \frac{\partial^2\Psi}{\partial\nu_j\partial\nu_i} 
    &= \frac{\partial}{\partial\nu_j}\left(\sum_{\ell\in[d]}\frac{1}{\beta_\ell}\q_\ell^\top(e^{\nu_i}\bS_i)\q_\ell\right) \\
    &= \underbrace{\sum_{\ell\in[d]}\frac{1}{\beta_\ell}\q_\ell^\top\frac{\partial}{\partial\nu_j}(e^{\nu_i}\bS_i)\q_\ell}_{\eqcolon A_{ij}^{(1)}} + \underbrace{\sum_{\ell\in[d]}\frac{\partial}{\partial\nu_j}\left(\frac{1}{\beta_\ell}\right)\q_\ell^\top(e^{\nu_i}\bS_i)\q_\ell}_{\eqcolon A_{ij}^{(2)}} + \underbrace{2\sum_{\ell\in[d]}\frac{1}{\beta_\ell}\q_\ell^\top(e^{\nu_i}\bS_i)\frac{\partial}{\partial\nu_j}(\q_\ell)}_{\eqcolon A_{ij}^{(3)}}. 
  \end{align*}
  We compute each term separately. The first term is 
  \begin{align*}
    A_{ij}^{(1)} 
    &= \delta_{ij}\sum_{\ell\in[d]}\frac{1}{\beta_\ell}\q_\ell^\top(e^{\nu_i}\bS_i)\q_\ell, 
  \end{align*}
  where $\delta_{ii}=1$ for all $i\in[d]$ and $\delta_{ij}=0$ for all $i\neq j$. 
  The second component involves the definition of $\bm\beta$ and thus the derivative of eigenvalues \eqref{eq:grad-eigenvalue}: 
  \begin{align*}
    A_{ij}^{(2)} 
    &= -\sum_{\ell\in[k]}\frac{1}{\lambda_\ell^2}\frac{\partial\lambda_\ell}{\partial\nu_j}\q_\ell^\top(e^{\nu_i}\bS_i)\q_\ell - \sum_{\ell\in[k+1,d]}\frac{1}{\bar\lambda^2}\frac{\partial\bar\lambda}{\partial\nu_j}\q_\ell^\top(e^{\nu_i}\bS_i)\q_\ell \\
    &= -\sum_{\ell\in[k]}\frac{1}{\lambda_\ell^2}[\q_\ell^\top(e^{\nu_j}\bS_j)\q_\ell][\q_\ell^\top(e^{\nu_i}\bS_i)\q_\ell] - \sum_{\ell\in[k+1,d]}\frac{1}{\bar\lambda^2}\left(\frac{1}{s-k}\sum_{\iota\in[k+1,d]}\q_\iota^\top(e^{\nu_j}\bS_j)\q_\iota\right)\q_\ell^\top(e^{\nu_i}\bS_i)\q_\ell \\
    &= -\sum_{\ell\in[k]}\frac{1}{\lambda_\ell^2}[\q_\ell^\top(e^{\nu_j}\bS_j)\q_\ell][\q_\ell^\top(e^{\nu_i}\bS_i)\q_\ell] - \sum_{\ell\in[k+1,d]}\sum_{\iota\in[k+1,d]}\frac{1}{(s-k)\bar\lambda^2}[\q_\iota^\top(e^{\nu_j}\bS_j)\q_\iota][\q_\ell^\top(e^{\nu_i}\bS_i)\q_\ell]. 
  \end{align*}
  The last term involves the derivative of eigenvectors \eqref{eq:grad-eigenvector}: 
  \begin{align*}
    A_{ij}^{(3)} 
    &= 2\sum_{\ell\in[d]}\frac{1}{\beta_\ell}\q_\ell^\top(e^{\nu_i}\bS_i)\frac{\partial}{\partial\nu_j}(\q_\ell) \\
    &= 2\sum_{\ell\in[d]}\frac{1}{\beta_\ell}\q_\ell^\top(e^{\nu_i}\bS_i)\left(\sum_{\iota\in[d],\iota\neq\ell}\frac{1}{\lambda_\ell-\lambda_\iota}\q_\iota\q_\iota^\top(e^{\nu_j}\bS_j)\q_\ell\right) \\
    &= 2\sum_{\ell\in[d]}\sum_{\iota\in[d],\iota\neq\ell}\frac{1}{\beta_\ell(\lambda_\ell-\lambda_\iota)}[\q_\ell^\top(e^{\nu_i}\bS_i)\q_\iota][\q_\iota^\top(e^{\nu_j}\bS_j)\q_\ell] \\
    &= \sum_{\ell\in[d]}\sum_{\iota\in[d],\iota\neq\ell}\frac{1}{\beta_\ell(\lambda_\ell-\lambda_\iota)}[\q_\ell^\top(e^{\nu_i}\bS_i)\q_\iota][\q_\iota^\top(e^{\nu_j}\bS_j)\q_\ell]  \\
    &\qquad - \sum_{\iota\in[d]}\sum_{\ell\in[d],\ell\neq\iota}\frac{1}{\beta_\ell(\lambda_\iota-\lambda_\ell)}[\q_\ell^\top(e^{\nu_i}\bS_i)\q_\iota][\q_\iota^\top(e^{\nu_j}\bS_j)\q_\ell] \\
    &= \sum_{\ell\in[d]}\sum_{\iota\in[d],\iota\neq\ell}\left(\frac{1}{\beta_\ell}-\frac{1}{\beta_\iota}\right)\frac{1}{\lambda_\ell-\lambda_\iota}[\q_\ell^\top(e^{\nu_i}\bS_i)\q_\iota][\q_\iota^\top(e^{\nu_j}\bS_j)\q_\ell]. 
  \end{align*}
  To show that $\nabla^2\Psi(\bm\nu)=\bA^{(1)}+\bA^{(2)}+\bA^{(3)}\succeq0$, we consider its quadratic form. For any $\a\in\R^m$, denote $\bP(\a)\coloneqq\sum_{i\in[m]}a_ie^{\nu_i}\bS_i$, then
    \begin{align*}
    \a^\top \bA^{(1)}\a 
    &= \sum_{i\in[m]}\sum_{j\in[m]}A_{ij}^{(1)}a_ia_j 
    = \sum_{i\in[m]}A_{ii}^{(1)}a_i^2\\
    &= \sum_{\ell\in[d]}\frac{1}{\beta_\ell}\q_\ell^\top\left(\sum_{i\in[m]}a_i^2e^{\nu_i}\bS_i\right)\q_\ell \sum_{\iota\in[d]}\frac{1}{\lambda_\iota}\q_\iota^\top\left(\sum_{i\in[m]}e^{\nu_i}\bS_i\right)\q_\iota \\
    &= \sum_{\ell\in[d]}\sum_{\iota\in[d]}\frac{1}{\beta_\ell\lambda_\iota}\left(\sum_{i\in[m]}a_i^2e^{\nu_i}(\q_\ell^\top\bS_i\q_\ell)\right)\left(\sum_{i\in[m]}e^{\nu_i}(\q_\iota^\top\bS_i\q_\iota)\right) \\
    &\geq \sum_{\ell\in[d]}\sum_{\iota\in[d]}\frac{1}{\beta_\ell\lambda_\iota}\left(\sum_{i\in[m]}a_ie^{\nu_i}(\q_\ell^\top\bS_i\q_\iota)\right)^2 \\
    &= \sum_{\ell\in[d]}\sum_{\iota\in[d]}\frac{1}{\beta_\ell\lambda_\iota}[\q_\ell^\top\bP(\a)\q_\iota]^2 \\
    &{=  \sum_{\ell\in[d]}\sum_{\iota\in[d],\iota\neq \ell}\frac{1}{\beta_\ell\lambda_\iota}[\q_\ell^\top\bP(\a)\q_\iota]^2 + \sum_{\ell\in[d]}\frac{1}{\beta_\ell\lambda_\ell}}[\q_\ell^\top\bP(\a)\q_\ell]^2, 
  \end{align*}
where the third equality follows from the using the identity $\sum_{\iota\in[d]}\lambda_\iota^{-1}\q_\iota^\top\bM(\bm\nu)\q_\iota = \Diag(\bm\lambda)^{-1}\bQ^\top\bM(\bm\nu)\bQ=\bI_d$ and the inequality follows from applying Cauchy-Schwarz inequality {(note that the terms inside the sums are nonnegative)}. 

{
In addition, we have
  \begin{align*}
    \a^\top \bA^{(2)}\a 
    &= \sum_{i\in[m]}\sum_{j\in[m]}A_{ij}^{(2)}a_ia_j \\
    &= - \sum_{i\in[m]}  \sum_{j\in[m]} \sum_{\ell\in[k]}\frac{1}{\lambda_\ell^2}[\q_\ell^\top(a_j e^{\nu_j}\bS_j)\q_\ell][\q_\ell^\top(a_i e^{\nu_i}\bS_i)\q_\ell] \\ 
    &\qquad - \sum_{i\in[m]}  \sum_{j\in[m]} \sum_{\ell\in[k+1,d]}\sum_{\iota\in[k+1,d]}\frac{1}{(s-k)\bar\lambda^2}[\q_\iota^\top(a_j e^{\nu_j}\bS_j)\q_\iota][\q_\ell^\top(a_i e^{\nu_i}\bS_i)\q_\ell] \\ 
    &= -\sum_{\ell\in[k]}\frac{1}{\lambda_\ell^2}[\q_\ell^\top\bP(\a)\q_\ell]^2 - \frac{1}{(s-k)\bar\lambda^2} \sum_{\ell\in[k+1,d]}\sum_{\iota\in[k+1,d]}[\q_\iota^\top\bP(\a)\q_\iota][\q_\ell^\top\bP(\a)\q_\ell] \\
    &= -\sum_{\ell\in[k]}\frac{1}{\lambda_\ell^2}[\q_\ell^\top\bP(\a)\q_\ell]^2 - \frac{1}{(s-k)\bar\lambda^2} \left( \sum_{\ell\in[k+1,d]} \q_\ell^\top\bP(\a)\q_\ell \right)^2 ,
  \end{align*}
as well as
  \begin{align*}
    \a^\top \bA^{(3)}\a 
    &= \sum_{i\in[m]}\sum_{j\in[m]}A_{ij}^{(3)}a_ia_j \\
    &=  \sum_{i\in[m]}  \sum_{j\in[m]} \sum_{\ell\in[d]}\sum_{\iota\in[d],\iota\neq\ell}\left(\frac{1}{\beta_\ell}-\frac{1}{\beta_\iota}\right)\frac{1}{\lambda_\ell-\lambda_\iota}[\q_\ell^\top( e^{\nu_i}\bS_i)\q_\iota][\q_\iota^\top(e^{\nu_j}\bS_j)\q_\ell] a_i a_j \\    
    &=  \sum_{\ell\in[d]}\sum_{\iota\in[d],\iota\neq\ell}\left(\frac{1}{\beta_\ell}-\frac{1}{\beta_\iota}\right)\frac{1}{\lambda_\ell-\lambda_\iota}[\q_\ell^\top\bP(\a)\q_\iota][\q_\iota^\top\bP(\a)\q_\ell]  \\
    &=  \sum_{\ell\in[d]}\sum_{\iota\in[d],\iota\neq\ell}\left(\frac{1}{\beta_\ell}-\frac{1}{\beta_\iota}\right)\frac{1}{\lambda_\ell-\lambda_\iota}[\q_\ell^\top\bP(\a)\q_\iota]^2 .       
  \end{align*}
}
  
Finally, putting things together and denoting $\bP\coloneq \bP(\a)$ we arrive at
  \begin{align*}
    &\a^\top(\bA^{(1)} + \bA^{(3)} + \bA^{(2)})\a \\
    &\geq 
    {\sum_{\ell\in[d]}\sum_{\iota\in[d], \iota\neq\ell}\frac{1}{\beta_\ell\lambda_\iota}[\q_\ell^\top\bP\q_\iota]^2 + \sum_{\ell\in[d]}\frac{1}{\beta_\ell\lambda_\ell}[\q_\ell^\top\bP\q_\ell]^2 }
    +\sum_{\ell\in[d]}\sum_{\iota\in[d],\iota\neq\ell}\left(\frac{1}{\beta_\ell}-\frac{1}{\beta_\iota}\right)\frac{1}{\lambda_\ell-\lambda_\iota}[\q_\ell^\top\bP\q_\iota]^2 \\
    &\qquad  -\sum_{\ell\in[k]}\frac{1}{\lambda_\ell^2}[\q_\ell^\top\bP\q_\ell]^2 
    - \frac{1}{(s-k)\bar\lambda^2} \left( \sum_{\ell\in[k+1,d]} \q_\ell^\top\bP\q_\ell \right)^2
    \\
    &{\geq \sum_{\ell\in[k+1,d]}\frac{1}{\bar\lambda\lambda_{\ell}}[\q_\ell^\top\bP\q_\ell]^2 - \frac{1}{(s-k)\bar\lambda^2}\left(\sum_{\ell\in[k+1,d]}\q_\ell^\top\bP\q_\ell\right)^2} \\
    &\geq 0, 
  \end{align*}
  where the second inequality follows from {${1\over \beta_\ell \lambda_\iota} = {1\over \lambda_\ell^2}$} for all $\ell=\iota \in[k]$ and the observation that $\frac{1}{\beta_\ell\lambda_\iota} + \left(\frac{1}{\beta_\ell}-\frac{1}{\beta_\iota}\right)\frac{1}{\lambda_\ell-\lambda_\iota} = \frac{\lambda_\ell/\beta_\ell - \lambda_\iota/\beta_\iota}{\lambda_{\iota}(\lambda_\ell-\lambda_\iota)} \geq 0$ for any $\ell\neq\iota$ as $\lambda_\ell\neq\lambda_\iota>0$ and both $(\lambda_1,\dots,\lambda_d)$ and $(\lambda_1/\beta_1, \dots, \lambda_d/\beta_d)=(1,\dots,1,\lambda_{k+1}/\bar\lambda,\ldots,\lambda_d/\bar\lambda)$ are sorted in a nonincreasing order {(recall that $\lambda_{k+1}\le \bar{\lambda}$ as $k$ is the unique index from \cref{lem:index-k})}. 
  Finally, the last inequality follows from Cauchy-Schwarz inequality and noting that $(s-k)\bar\lambda=\sum_{\iota\in[k+1,d]}\lambda_\iota$. 
This then shows that $\nabla^2\Psi(\bm\nu)\succeq0$ holds for any $\bm\nu\in\R^m$ such that $\bM(\bm\nu)$ has distinct eigenvalues. 

  To establish convexity of $\Psi$ on $\R^m$, it suffices to show that (i) the set of such $\bm\nu$ where $\bM(\bm\nu)$ has distinct eigenvalues is dense in $\R^m$, and (ii) $\Psi$ is continuously differentiable. Convexity of $\Psi$ then follows from \cite[Theorem 3.3]{dudley_second_1977}. The denseness claim follows because the map $\bm\nu\mapsto\bM(\bm\nu)$ is analytic, and hence the discriminant $\Delta(\bm\nu)$ of its characteristic polynomial is analytic. 
  By the premise of the theorem, there exists $\bar{\bm\nu}\in\R^m$ such that $\bM(\bar{\bm\nu})$ has distinct eigenvalues, and thus $\Delta(\bar{\bm\nu})\neq0$. By the identity theorem, the set $\set{\bm\nu\in\R^m:\Delta(\bm\nu)\neq0}$ is dense in $\R^m$, which coincides with the set where  $\nabla^2\Psi(\bm\nu)\succeq0$. 
  Finally, as $G_s$ is convex and also differentiable over $\S_{++}^d$ (see  \citep[Proposition 2]{li_best_2023}), by \cite[Remark 6.2.6]{hiriart-urruty_fundamentals_2001}, it is continuously differentiable. Therefore, $\Psi(\bm\nu)=G_s(\bM(\bm\nu))$ is continuously differentiable, completing the argument.
\end{proof}

We next present another structural result that will be useful in establishing the desired convex-concave saddle point form for the double-scaled  linx relaxation. The proof of \cref{lem:distinct-eigenvalues} is given in \cref{sec:distinct-eigenvalues}. 

\begin{lemma}\label{lem:distinct-eigenvalues}
  Suppose $\bS_i=\bB^\top\e_i\e_i^\top\bB$ for $i\in[d]$, where $\bB\in\R^{d\times d}$ is invertible, and $\bS_i\in\S_+^d$ for $i\in[d+1,m]$ {if $m>d$}. Then, there exists $\bar{\bm\nu}\in\R^m$ such that $\bM(\bar{\bm\nu})=\sum_{i\in[m]}e^{\bar\nu_i}\bS_i$ has distinct eigenvalues. 
\end{lemma}

We are now ready to prove that the relaxation obtained from applying double-scaling to linx given in \eqref{eq:double-scaling} indeed admits a convex-concave saddle point problem formulation after a change of variables of $(\bm\rho,\bm\omega)\coloneqq(\log\bm\gamma,\log\bm\mu)\in\R^d\times\R^d$. 
To this end, let us define 
\begin{align}
  &\objld(\x,\bm\rho,\bm\omega) 
  \coloneqq \objldlog(\x,\exp(\bm\rho),\exp(\bm\omega)) \notag\\
  &= -\frac12\log\det(\bC\Diag(\exp(\bm\rho))\Diag(\x)\bC+\Diag(\exp(\bm\omega)){\Diag(\ones-\x)} ) + \frac12\sum_{i\in[d]}x_i\rho_i + \frac12\sum_{i\in[d]}(1-x_i)\omega_i. \label{eq:double-linx-convex-concave}
\end{align}
After this change of variables, the relaxation \eqref{eq:double-scaling} becomes 
\begin{equation}\label{eq:d-scaling:reformulated}
\Opt_{\linxd} = \min_{\x\in\cX} \max_{\bm\rho,\bm\omega\in\R^d} \objld(\x,\bm\rho,\bm\omega).
\end{equation} 
It is straightforward to see that $\objld(\x,\bm\rho,\bm\omega)$ is convex in $\x$ for every fixed $(\bm\rho,\bm\omega)$. To establish that this function is also concave in $(\bm\rho,\bm\omega)$ for every fixed $\x\in[0,1]^d$, we apply the structural result from \cref{lem:gamma-g-scaling-hessian} with $s=d$ and a specific choice of $\bS_1,\dots,\bS_{2d}$.

\begin{proposition}\label{prop:d-scaling-linx-convexity}
  Suppose $\bC$ is full rank.
  The function $\objld(\x,\bm\rho,\bm\omega)$ is convex in $\x$ for every fixed $(\bm\rho,\bm\omega)$ and concave in $(\bm\rho,\bm\omega)$ for every fixed $\x\in[0,1]^d$. 
\end{proposition}
\begin{proof}
  Suppose $(\bm\rho,\bm\omega)$ is fixed. Then, 
  $\objld(\x,\bm\rho,\bm\omega)$ is the sum of a $-\log\det(\cdot)$ term and a linear function, and the $-\log\det(\cdot)$ term is a composite of convex matrix function $-\log\det(\cdot)$ with an affine mapping of $\x$ and so it is a convex function of $\x$. 
  Therefore, $\objld(\x,\bm\rho,\bm\omega)$ is convex in $\x$ for any fixed $(\bm\rho,\bm\omega)$. 

  Now, suppose $\x\in(0,1]^d$ is fixed. Let us set $m=2d$ and define the matrices
  \begin{align*}
    \bS_i &\coloneq \bC\e_i\e_i^\top\Diag(\x)\bC \succeq0,\quad \forall i\in[d], \\
    \bS_{d+i} &\coloneq \e_i\e_i^\top\Diag(\ones-\x) \succeq0,\quad \forall i\in[d].
  \end{align*}
  Then,
  \[
  	\bC\Diag(\exp(\bm\rho))\Diag(\x)\bC+\Diag(\exp(\bm\omega)) {\Diag(\ones-\x)}
	=\sum_{i\in[d]} \exp(\rho_i) S_i + \sum_{i\in[d]} \exp(\omega_i) S_{i+d}, 
  \]
  and thus $\objld(\x,\bm\rho,\bm\omega)=  -\frac12\log\det(\sum_{i\in[d]} \exp(\rho_i) S_i + \sum_{i\in[d]} \exp(\omega_i) S_{i+d} ) + \frac12\sum_{i\in[d]}x_i\rho_i + \frac12\sum_{i\in[d]}(1-x_i)\omega_i$, i.e.,
   the $\log\det(\cdot)$ term in $\objld(\x,\bm\rho,\bm\omega)$ is an instance of $\Psi(\bm\nu)$ in \cref{lem:gamma-g-scaling-hessian} with $s=d$ up to a factor of $\frac12$ in terms of the variables $\bm\nu\coloneqq (\bm\rho,\bm\omega)\in\R^{2d}$. 
  In addition, we have $\sum_{i\in[2d]}\bS_i=\bC\Diag(\x)\bC+\Diag(\ones-\x)\succ0$ for any $\x\in[0,1]^d$ (see \citep[Lemma 3.3.2]{fampa_maximum-entropy_2022}). Note also that the matrices $\bS_i$ satisfy the structural assumption of \cref{lem:distinct-eigenvalues} by letting $\bB:=\Diag(\x)^{1/2}\bC$ which is invertible {as $\x\in(0,1]^d$ and $\bC$ is full rank}.
  Then, by \cref{lem:gamma-g-scaling-hessian,lem:distinct-eigenvalues}, we conclude that $\objld(\x,\bm\rho,\bm\omega)$ in concave in $(\bm\rho,\bm\omega)$ for $\x\in(0,1]^d$. By \cite[Theorem 10.8]{rockafellar_convex_1970} and the continuity of $\objld$, we conclude that it is concave in $(\bm\rho,\bm\omega)$ for all $\x\in[0,1]^d$. 
\end{proof}

As an immediate consequence of \cref{prop:d-scaling-linx-convexity}, together with the observation in \cref{rem:scaling-comparisons}, we recover that the objective functions arising from o-scaling and g-scaling of the linx relaxation also admit a convex–concave structure after an appropriate change of variables.  
For ease of notation, we denote these o-scaling and g-scaling objective functions by
\begin{align}
  \objlo(\x,\rho_0) &\coloneqq \objlolog(\x,\exp(\rho_0)) \label{eq:f-linx-o} \\
  &= -\frac12\log\det(\exp(\rho_0)\bC\Diag(\x)\bC+ {\Diag(\ones-\x)} ) + \frac12 s\rho_0,  \notag \\
  \objlg(\x,\bm\rho) &\coloneqq \objlglog(\x,\exp(\bm\rho)) \label{eq:f-linx-g} \\
  &= -\frac12\log\det \left(\Diag(\exp(\bm\rho))\bC\Diag(\x)\bC\Diag(\exp(\bm\rho)) + {\Diag(\ones-\x)} \right) + \sum_{i\in[d]}x_i\rho_i.  \notag
\end{align}
Then, \cref{prop:d-scaling-linx-convexity} implies that $\objlo$ is convex in $\x$ for every fixed $\rho_0\in\R$ and concave in $\rho_0$ for every fixed $\x\in[0,1]^d$. Similarly, $\objlg$ is convex in $\x$ for every fixed $\bm\rho\in\R^d$ and concave in $\bm\rho$ for every fixed $\x\in[0,1]^d$.

\section{Scaling methods for \texorpdfstring{$\Gamma$}{Gamma} relaxation}
\label{sec:gamma-g-scaling}

In this section, we study the scaling enhancement of $\Gamma$ relaxation \eqref{eq:opt-gamma:equivalence} when $\varphi(\cdot)=-\log(\cdot)$. Recall that the $\Gamma$ relaxation is invariant to o-scaling {(see \cite[Theorem 3.4.14]{fampa_maximum-entropy_2022})}. 
To further strengthen it, \cite{chen_generalized_2024} proposed the following \emph{g-scaling} (general scaling) enhancement: 
\begin{align}\label{eq:gamma-g-scaling}
  \Opt_{\Gamma\text{-g}} \coloneqq \min_{\x\in\cX} \max_{\bm\gamma\in\R^d_{++}}\set{\objgglog (\x,\bm\gamma)  
  \coloneqq G_s\left(\bV\Diag(\bm\gamma)\Diag(\x)\bV^\top\right) + \sum_{i\in[d]}x_i\log\gamma_i}.
\end{align}

As in the case of g-scaled linx relaxation, we perform the change of variables $\bm\rho\coloneqq\log\bm\gamma\in\R^d$ and define
\begin{align}\label{eq:f-Gamma}
  \objgg(\x,\bm\rho) 
  &\coloneqq G_s\left(\bV\Diag(\exp(\bm\rho))\Diag(\x)\bV^\top\right) + \sum_{i\in[d]}x_i\rho_i . 
\end{align}
Thus, the g-scaled $\Gamma$ relaxation can be written as $\min_{\x\in\cX}\max_{\bm\rho\in\R^d}\objgg(\x,\bm\rho)$. 
When $\bm\rho=\bm0$, this reduces to the original $\Gamma$ relaxation. 

We first analyze the convexity properties of the g-scaled $\Gamma$ relaxation in \cref{sec:SP:g-scalled-Gamma}. This structural result further allows us to show in \cref{sec:g-scalled-Gamma-invariance} that, when $\cX=\set{\x\in[0,1]^d:\, \ones^\top \x=s}$, the $\Gamma$ relaxation is invariant under g-scaling.

\subsection{Convex-concave saddle point formulation for g-scaled \texorpdfstring{$\Gamma$}{Gamma} relaxation}\label{sec:SP:g-scalled-Gamma}

To obtain a tractable saddle point formulation, we need $\objgg(\x,\bm\rho)$ to be convex in $\x$ and concave in $\bm\rho$.
For each fixed $\bm\gamma\in\R^d_{++}$ (equivalently, fixed $\bm\rho\in\R^d$), \cite[Theoerem 6.ii]{chen_generalized_2024} shoved that the function 
$\objgglog(\x,\bm\gamma)$ (and thus $\objgg(\x,\bm\rho)$) is convex in $\x$. However, concavity of $\objgg(\x,\bm\rho)$ with respect to the g-scaling vector $\bm\rho$ was left unresolved, leaving the computational study of g-scaling for the $\Gamma$ relaxation relaxation without a theoretical foundation. In \cref{prop:gamma-g-scaling-convexity}, we close this gap by establishing the full convex–concave structure of $\objgg(\x,\bm\rho)$.

\begin{proposition}\label{prop:gamma-g-scaling-convexity}
{Suppose $\varphi(\cdot)=-\log(\cdot)$ and $\bV^\top$ is invertible.}
 For any fixed $\bm\rho\in\R^d$, the function $\objgg(\x,\bm\rho)$ is convex in $\x\in\R^d$, and  for any fixed $\x\in\cX$, the function $\objgg(\x,\bm\rho)$ is concave in $\bm\rho\in\R^d$. 
\end{proposition}

\begin{proof}
  Let $\bm\rho$ be fixed. Then, 
  $\objgg(\x,\bm\rho)$ is the sum of a function of the form $G_s(\cdot)$ and a linear function, where the $G_s$ term is a composition of the convex matrix function $G_s$ (see \cref{lem:h-convex}) and an affine mapping of $\x$ and so it is a convex function of $\x$. 
  Therefore, $\objgg(\x,\bm\rho)$ is convex in $\x$ for any fixed $\bm\rho$. 

  Now, suppose $\x\in\R_{++}^d$ is fixed. Let us define the matrices
  \begin{align*}
    \bS_i &\coloneq \bV\e_i\e_i^\top\Diag(\x)\bV^\top \succeq0,\quad \forall i\in[d]. 
  \end{align*}
  Then,
  \[
  	\bV\Diag(\exp(\bm\rho))\Diag(\x)\bV^\top
	=\sum_{i\in[d]} \exp(\rho_i) \bS_i, 
  \]
  and thus $\objgg(\x,\bm\rho) = G_s \left(\sum_{i\in[d]} \exp(\rho_i) \bS_i \right) + \sum_{i\in[d]}x_i\rho_i$. In addition, for any $\x\in\R_{++}^d$, we have $\sum_{i\in[d]}\bS_i=\bV\Diag(\x)\bV^\top\succ0$ and $\bS_i=\bB^\top\e_i\e_i^\top\bB$ with $\bB=\Diag(\x)^{1/2}\bV^\top$ invertible (as $\x\in\R_{++}^d$ and $\bV^\top$ is invertible by the premise of the proposition). Then, by \cref{lem:gamma-g-scaling-hessian,lem:distinct-eigenvalues}, we conclude that $\objgg(\x,\bm\rho)$ is concave in $\bm\rho$ for any $\x\in\cX\cap\R^d_{++}$. 
  Finally, recall that $\objgg$ is continuous in $\bm\rho$ by \cite[Corollary 9]{chen_generalized_2024}, and thus from \cite[Theorem 10.8]{rockafellar_convex_1970} we conclude that it is concave in $\bm\rho$ for all $x\in\cX$.
\end{proof}

\subsection{G-scaling invariance of \texorpdfstring{$\Gamma$}{Gamma} relaxation}\label{sec:g-scalled-Gamma-invariance}

It is well-known that o-scaling does not improve the $\Gamma$ relaxation bound $\Opt_{\Gamma}$ given in \eqref{eq:opt-gamma:equivalence}. This raises the natural question of whether g-scaling can strengthen it. In fact, \cite[Remark 5, Section 7]{chen_generalized_2024} posed this as an open question for the standard MESP with $\cX=\set{\x\in[0,1]^d:\, \ones^\top \x=s}$.
In \cref{prop:g-scaling-Gamma-invariant}, we resolve this question by showing that, for the standard MESP, g-scaling does not improve upon the $\Gamma$ relaxation bound. 
Our proof relies critically on the convex-concave structure of $\objgg(\x,\bm\rho)$ established in \cref{prop:gamma-g-scaling-convexity};
see \cref{sec:g-scaling-Gamma-invariant} for the proof.

\begin{theorem}\label{prop:g-scaling-Gamma-invariant}
  For the standard MESP, where $\cX=\set{\x\in[0,1]^d:\, \ones^\top \x=s}$ and $\varphi(\cdot)=-\log(\cdot)$, applying g-scaling to the $\Gamma$ relaxation yields no improvement; that is,  
  \[
\Opt_{\Gamma\text{-g}} = \min_{\x\in\cX}\max_{\bm\rho\in\R^d} \objgg(\x,\bm\rho)
=\min_{\x\in\cX}\objgg(\x,\bm0) =\Opt_{\Gamma}.
\]
\end{theorem}

Although in \cref{prop:g-scaling-Gamma-invariant} we have shown that g-scaling does not improve $\Gamma$ relaxation when $\cX=\{\x\in[0,1]^d:\ones^\top\x=s \}$, the numerical study in \cite{chen_generalized_2024} shows that g-scaling may improve $\Gamma^c$ in the case of CMESP, i.e., when additional linear constraints $\bA\x\le \b$ are present in $\cX$. 

Recall that $\bC=\bV^\top\bV$. When $\bV$ is full rank, we can write $\bC^{-1}=\bW^\top\bW\in\S^d_{+}$ by selecting $\bW=\bV^{-\top}$. In this case, the complementary  $\Gamma$ relaxation (see \eqref{eq:opt-gamma-complementary}) is then itself  an instance of $\Gamma$ relaxation (up to an additive constant), and g-scaling can be applied to it in the same manner as in \eqref{eq:gamma-g-scaling}--\eqref{eq:f-Gamma} to obtain 
\begin{align*}
  \min_{\check\x\in\cX_{d-s}} \max_{\bm\omega\in\R^d}\set{G_{d-s}\left(\bW\Diag(\exp(\bm\omega))\Diag(\check\x)\bW^\top\right) {- \log\det(\bC)} + \sum_{i\in[d]}\check x_i\omega_i}.
\end{align*}
With a change of variables $\x=\ones-\check\x\in\cX$ and fixing $\bW=\bV^{-\top}$, 
 we define the g-scaled complementary relaxation as 
\begin{align}
  \Opt_{\Gamma^c\text{-g}} &\coloneqq  \min_{\x\in\cX} \max_{\bm\omega\in\R^d}\set{  \objggc(\x,\bm\omega) },  \label{eq:Opt-Gamma-c} \\ 
  \text{where }~  \objggc(\x,\bm\omega) &\coloneqq G_{d-s}\left(\bV^{-\top}\Diag(\exp(\bm\omega))\Diag(\ones-\x)\bV^{-1}\right) - \log\det(\bC) + \sum_{i\in[d]}(1-x_i)\omega_i . \label{eq:f-Gamma-c}
\end{align}
In the next section, we will establish a connection between double-scaling for linx and g-scaling for both $\Gamma$ and $\Gamma^c$.

\section{Connections between linx, \texorpdfstring{$\Gamma$}{Gamma}, \texorpdfstring{$\Gamma^c$}{Gamma\^{}c} relaxations, and their scalings}
\label{sec:connection}

In this section, we explore the relationships among the linx, $\Gamma$, $\Gamma^c$ relaxations and their scaled variants. Numerical evidence in \citep{li_best_2023,chen_computing_2023,chen_generalized_2024} suggests that none of these relaxations uniformly dominates the others. Nevertheless, in \cref{prop:connection} and \cref{cor:linx-gamma-connection}, we establish precise connections between their bounds. In particular, we show that the bound obtained from applying o-scaling technique to linx is at least as strong as the average of the $\Gamma$ and $\Gamma^c$ bounds. We further prove an analogous result for double-scaling of linx and g-scaling of $\Gamma$ and $\Gamma^c$ relaxations. Both results rely on the following key observation.

\begin{proposition}\label{prop:connection} 
  Suppose {$\bC=\bV^\top\bV$ with $\bV$} full rank. Then, for
  any $\x\in\cX$ and any $\bm\rho,\bm\omega\in\R^d$, there exists $\kappa\in\R$ such that $\objld(\x,\bm\rho + \kappa\ones,\bm\omega) \geq \frac12(\objgg(\x,\bm\rho) + \objggc(\x,\bm\omega))$.  
\end{proposition}
\begin{proof}
  Fix $\x\in\cX$, $\bm\rho,\bm\omega\in\R^d$. 
  Denote $\bM\coloneqq\bV\Diag(\exp(\bm\rho))\Diag(\x)\bV^\top$ and $\bP\coloneqq\bV^{-\top}\Diag(\exp(\bm\omega))(\bI_d - \Diag(\x))\bV^{-1}$. Observe that $\bV^\top(e^\kappa\bM+\bP)\bV=e^\kappa\bC\Diag(\exp(\bm\rho))\Diag(\x)\bC+\Diag(\exp(\bm\omega))(\bI_d-\Diag(\x))=\bC\Diag(\exp(\bm\rho+\kappa\ones))\Diag(\x)\bC+\Diag(\exp(\bm\omega))(\bI_d-\Diag(\x))$. 
  By the definition of $\objld$ in \eqref{eq:double-linx-convex-concave}, we have 
  \begin{align*}
    &\objld(\x,\bm\rho + \kappa\ones,\bm\omega) \\
    &= -\frac12\log\det(\bV^\top(e^\kappa\bM+\bP)\bV) + \frac12\sum_{i\in[d]}x_i(\rho_i+\kappa) + \frac12\sum_{i\in[d]}(1-x_i)\omega_i \\
    &= -\frac12\log\det(e^\kappa\bM+\bP) - \frac12\log\det\bV^\top - \frac12\log\det\bV + \frac12\sum_{i\in[d]}x_i\rho_i + \frac12\sum_{i\in[d]}(1-x_i)\omega_i + \frac12s\kappa \\
    &= \frac12G_d(e^\kappa\bM+\bP) - \frac12\log\det\bC + \frac12\sum_{i\in[d]}x_i\rho_i + \frac12\sum_{i\in[d]}(1-x_i)\omega_i + \frac12s\kappa \\
    &= \min_{\bm\xi\in\R^d_+}\set{-\frac12\sum_{i\in[d]}\log\xi_i - \frac12\log\det\bC + \frac12\sum_{i\in[d]}x_i \rho_i + \frac12\sum_{i\in[d]}(1-x_i)\omega_i + \frac12s\kappa: \right. \\
    &\hspace{10em}\left.\vphantom{\frac12\sum_{i\in[d]}} \bm\xi\majorize\bm\lambda\left(e^\kappa\bM + \bP \right),\ \bm\xi \in \K_{o,d}
    }. 
  \end{align*}
  Let $\xild\coloneq \xild(\x,\bm\rho+\kappa\ones,\bm\omega)$ be the optimal solution to the above minimization problem given by \cref{prop:h-rep}. 
  Consider the optimization problems defining $\objgg(\x,\bm\rho)$ (see \eqref{eq:f-Gamma}) and $\objggc(\x,\bm\omega)$ (see \eqref{eq:f-Gamma-c}), respectively. That is, in our notation, these problems are given by 
  \begin{align*}
    \objgg(\x,\bm\rho)
    &= \min_{\bm\xi\in\R^d_+}\set{-\sum_{i\in[s]}\log\xi_i + \sum_{i\in[d]}x_i\rho_i: \,  
    \bm\xi\majorize\bm\lambda(\bM),\ \bm\xi \in \K_{o,d}}, \\
    \objggc(\x,\bm\omega)
    &= \min_{\bm\xi\in\R^d_+}\set{-\sum_{i\in[d-s]}\log\xi_i + \sum_{i\in[d]}(1-x_i)\omega_i - \log\det\bC: \, 
    \bm\xi\majorize\bm\lambda(\bP),\ \bm\xi \in \K_{o,d}}.
  \end{align*}
  Let $\xigg:=\xigg(\x,\bm\rho)$ and $\xiggc:=\xiggc(\x,\bm\omega)$ be the optimal solutions corresponding to these two problems given by \cref{prop:h-rep}. 
  
  We claim that there exists $\kappa\in\R$ such that $\xild\majorize\begin{pmatrix}
    e^\kappa \xigg_{1:s} \\ \xiggc_{1:d-s}
  \end{pmatrix}$. Recall by the definition of majorization (see \cref{def:majorization}) that {the vectors $\bm\xi,\bm\zeta\in \R^d$ satisfy} $\bm\xi\majorize\bm\zeta$ if and only if 
  \begin{align}
    \sum_{i\in[j]}\xi_{[i]} &\geq \sum_{i\in[j]}\zeta_{[i]}, \quad j\in[d-1], \label{eq:majorization-def-ineq}\\
    \sum_{i\in[d]}\xi_{[i]} &= \sum_{i\in[d]}\zeta_{[i]}. \label{eq:majorization-def-eq}
  \end{align}

  By \cref{prop:h-rep}, the closed-form expressions for $\xild,\xigg,\xiggc$ are given as 
  \begin{gather*}
    \xild = \bm\lambda(e^\kappa\bM+\bP), \\
    \xigg = \left(\lambda_1(\bM), \cdots, \lambda_k(\bM), 
    \tau(\bM), \ldots, \tau(\bM), 
    0, \dots, 0\right)^\top, \\
    \xiggc = \left(\lambda_1(\bP), \cdots, \lambda_{k^c}(\bP), 
     \tau(\bP), \ldots, \tau(\bP), 
      0, \dots, 0\right)^\top, 
  \end{gather*}
  where the indices $k$ and $k^c$ are as defined in \cref{lem:index-k} for the vectors $\bm\lambda(\bM)$ and $\bm\lambda(\bP)$, respectively and we define $\tau(\bM)\coloneqq \frac{1}{s-k}\sum_{i\in[k+1,d]}\lambda_i(\bM)$ and $\tau(\bP) \coloneqq \frac{1}{d-s-k^c}\sum_{i\in[k^c+1,d]}\lambda_i(\bP)$.
  Here, the first $s$ coordinates of $\xigg$ and the first $d-s$ coordinates of $\xiggc$ are nonzero. In addition, $k\in[s-1]$ and $k^c\in[d-s-1]$ are such that $\xigg,\xiggc\in\K_{o,d}$. 
  In addition, note that for any  $\kappa\in\R$ we have 
  \begin{align}\label{eq:stacked-majorization} 
  \begin{pmatrix}\xild \\ \bm0\end{pmatrix} = 
  \begin{pmatrix}\bm\lambda(e^\kappa\bM+\bP) \\ \bm0 \end{pmatrix}
  \majorize \bm\lambda \begin{pmatrix} e^\kappa \bM & \bm0  \\ \bm0 & \bP \end{pmatrix}
  =\begin{pmatrix} e^\kappa\bm\lambda(\bM) \\ \bm\lambda(\bP) \end{pmatrix} ,
  \end{align} 
  where the majorization relation follows from \citep[Theorem 1.1]{lin_eigenvalue_2011} which is applicable since  $\bM,\bP\succeq0$. 
  
  Let $\kappa\in\R$ be such that $e^\kappa\cdot \tau(\bM) = \tau(\bP)$, i.e., 
  \[
    e^\kappa \cdot \frac{1}{s-k}\sum_{i\in[k+1,d]}\lambda_i(\bM) = \frac{1}{d-s-k^c}\sum_{i\in[k^c+1,d]}\lambda_i(\bP) . 
  \]
  Define 
  \[
    \widehat{\bm\xi} \coloneqq \begin{pmatrix}
      e^\kappa\xigg_{1:s} \\ \xiggc_{1:d-s}
    \end{pmatrix} 
    = \begin{pmatrix}
      e^\kappa\, \bm\lambda_{1:k}(\bM) \\ e^\kappa\ \tau(\bM) \ones_{s-k} \\ \bm\lambda_{1:k^c}(\bP) \\ \tau(\bP) \ones_{d-s - k^c} 
    \end{pmatrix}   
    .
  \] 
  Hence, by our choice of $\kappa$, for any $\ell\in [k+k^c]$, the $\ell$-th largest coordinate of $\widehat{\bm\xi} 
  $ coincides with the $\ell$-th largest coordinate of $\begin{pmatrix}
    e^\kappa\bm\lambda(\bM) \\ \bm\lambda(\bP)
  \end{pmatrix}$, and the rest $d-(k+k^c)$ coordinates of $\widehat{\bm\xi} 
  $ are all equal to $e^\kappa\cdot \tau(\bM) = \tau(\bP)$. 
  We claim that $\xild\majorize\widehat{\bm\xi}$ holds. 
  Using the definition of $\widehat{\bm\xi}$ and \eqref{eq:stacked-majorization}, we see that  the first $k+k^c$ inequalities of \eqref{eq:majorization-def-ineq} in the definition of the majorization relation $\xild\majorize\widehat{\bm\xi}
  $ hold.
  In addition,  the equation \eqref{eq:majorization-def-eq} in the definition of majorization relation $\xild\majorize\widehat{\bm\xi}$ is satisfied because 
  \begin{align*}
    \sum_{i\in[d]}\xildi
    &= \tr(e^\kappa\bM+\bP) \\
    &= e^\kappa\tr(\bM)+\tr(\bP) \\
    &= e^\kappa\sum_{i\in[d]}\lambda_i(\bM) + \sum_{i\in[d]}\lambda_i(\bP) \\
    &= e^\kappa\sum_{i\in[s]}\xiggi + \sum_{i\in[d-s]}\xiggci \\
    &=  \sum_{i\in[d]} \widehat{\xi}_i ,
  \end{align*}
  which holds by the definitions of $\xild,\xigg,\xiggc$. 
    Moreover,  as $e^\kappa\cdot \tau(\bM) = \tau(\bP)$, the $d-(k+k^c)$ smallest coordinates of $ \widehat{\bm\xi}
  $ 
  are all equal. 
  {
    Since the first $k+k^c$ inequalities of \eqref{eq:majorization-def-ineq} hold, and moreover the equality \eqref{eq:majorization-def-eq} holds, we have 
  \begin{align*}
    \sum_{i\in[k+k^c+1,d]}\xild_i \leq \sum_{i\in[k+k^c+1,d]}\widehat{\bm\xi}_i. 
  \end{align*}
  Therefore, for any $j\in[k+k^c,d-1]$, we have
  \begin{align*}
    \frac{1}{d-j}\sum_{i\in[j+1,d]}\widehat{\bm\xi}_i 
    &= \frac{1}{d-k-k^c}\sum_{i\in[k+k^c+1,d]}\widehat{\bm\xi}_i \\
    &\geq \frac{1}{d-k-k^c}\sum_{i\in[k+k^c+1,d]}\xild_i \\
    &\geq \frac{1}{d-j}\sum_{i\in[j+1,d]}\xild_i. 
  \end{align*}
  Here, the equality follows from the choice of $\kappa$ which guarantees that $\widehat{\bm\xi}_i$'s are all equal for $i\in[k+k^c+1,d]$, and the last inequality follows from the fact that $\xild\in\K_{o,d}$. Combined with \eqref{eq:majorization-def-eq}, this implies 
  \begin{align*}
    \sum_{i\in[j]}\widehat{\bm\xi}_i \leq \sum_{i\in[j]}\xild_i, ~~\forall j\in[k+k^c,d-1],
  \end{align*}
  i.e., the remaining inequalities of \eqref{eq:majorization-def-ineq} hold true as well. 
  Therefore, we indeed have $\xild\majorize\widehat{\bm\xi}$ as desired. 
  } 
  Then, by {\cref{lem:monotone} and the} Schur-convexity of the function $\bm\xi\mapsto-\frac12\sum_{i\in[d]}\log(\xi_i)$, the majorization relation $\xild\majorize\widehat{\bm\xi}$ implies 
  \begin{align*}
    -\frac12\sum_{i\in[d]}\log(\xildi) 
    &\geq -\frac12\sum_{i\in[s]}\log(e^\kappa\xiggi) - \frac12\sum_{i\in[d-s]}\log(\xiggci) \\ 
    &= -\frac12\sum_{i\in[s]}\log(\xiggi) - \frac12\sum_{i\in[d-s]}\log(\xiggci) - \frac12s\kappa. 
  \end{align*}
  {Recalling that $\xild$, $\xigg$, and $\xiggc$ are the optimal solutions corresponding to the minimization problems defining $\objld(\x,\bm\rho+\kappa\ones,\bm\omega)$, $\objgg(\x,\bm\rho)$, and $\objggc(\x,\bm\omega)$, respectively, we conclude that this inequality }
  shows $\objld(\x,\bm\rho+\kappa\ones,\bm\omega)\geq\frac12(\objgg(\x,\bm\rho)+\objggc(\x,\bm\omega))$ for the value of $\kappa$ chosen as $\kappa = \log(\tau(\bP) / \tau(\bM)) $. 
\end{proof}

As \cref{prop:connection} holds for all $\x\in\cX$ and all $\bm\rho,\bm\omega\in\R^d$, we arrive at the following corollary, 
which reveals an interesting  connection between the bounds obtained by applying double-scaling (or o-scaling) to linx and the average of bounds from  g-scaled (or no scaling) $\Gamma$ and $\Gamma^c$ relaxations; see \cref{sec:linx-gamma-connection} for the proof of this corollary. 

\begin{corollary}\label{cor:linx-gamma-connection}
 Suppose {$\bC=\bV^\top\bV$ with $\bV$} full rank. Then, the following relations hold for both MESP and CMESP:
\begin{enumerate}[(i)]
\item $\Opt_{\linxo} \ge \frac12 \left[ \Opt_{\Gamma} +\Opt_{\Gamma^c} \right]$,
\item $\Opt_{\linxd} \ge \frac12 \left[ \Opt_{\Gamma\textup{-g}} +\Opt_{\Gamma^c\textup{-g}} \right]$. 
\end{enumerate}
\end{corollary}

The relations in \cref{cor:linx-gamma-connection} apply to both MESP and CMESP. While g-scaling does not improve the $\Gamma$ relaxation bound in the standard MESP setting, it can yield strictly stronger bounds in certain CMESP instances {\citep{chen_generalized_2024}. 

It is worth emphasizing a key distinction between double-scaling and g-scaling. In light of \cref{cor:linx-gamma-connection}, one might ask whether the inequality  $\Opt_{\linxg} \ge \frac12 \left[ \Opt_{\Gamma\textup{-g}} +\Opt_{\Gamma^c\textup{-g}} \right]$ also holds. We suspect that this is not the case.   
As reflected in the proof of \cref{cor:linx-gamma-connection}, the $\Gamma$ and $\Gamma^c$ relaxations may attain their optimal bounds at different g-scaling vectors. The double-scaling formulation for linx can simultaneously accommodate both scaling directions, whereas g-scaling applied to linx cannot. This highlights the added flexibility, and strength, of double-scaling relative to g-scaling for linx.

To the best of our knowledge, no integrality gap has been established in the literature for the linx relaxation or its scaling variants {even in the standard MESP case}. 
This contrasts with the $\Gamma$ relaxation and its complementary problem 
{for standard MESP}, 
for which \cite[Corollary 1]{li_best_2023} proved that (note that their objective function is the negative of ours)
  \begin{align*}
    \Opt &\geq \Opt_{\Gamma} 
    \geq \Opt - s\log s + s\log d - \log{\binom{d}{s}}, \\
    \Opt &\geq \Opt_{\Gamma^c} 
    \geq \Opt - (d-s)\log(d-s) + (d-s)\log d - \log{\binom{d}{d-s}}. 
  \end{align*}
  (Recall $\Opt_{\Gamma}=\min_{\x\in\cX} {\objgg(\x,\bm0)}$ and $\Opt_{\Gamma^c}=\min_{\x\in\cX} {\objggc(\x,\bm0)}$.)
  Leveraging the connection between o-scaled linx relaxation and the $\Gamma$ relaxations given in \cref{cor:linx-gamma-connection}(a),  these preceding integrality gaps for $\Gamma$ and $\Gamma^c$ immediately yield an integrality gap for the o-scaled linx relaxation {for the standard MESP}.
    
\begin{corollary}\label{cor:linx-integrality-gap}
  Suppose $\cX=\set{\x\in[0,1]^d:\,\ones^\top \x=s}$. Then, the integrality gap of the o-scaled linx relaxation satisfies: 
  \begin{align*}
    \Opt_{\linxo} 
    \geq \Opt - \frac12 \left[s\log s + (d-s)\log(d-s) - d\log d \right] - \log{\binom{d}{s}}. 
  \end{align*} 
\end{corollary}

%% file: experiments.tex
\section{Numerical study}
\label{sec:numerical}

In this section, we present a preliminary numerical study comparing the quality of the bounds from various relaxations discussed in this paper.

We use two covariance matrices $\bC\in\S^n_+$ from the MESP literature with dimensions $n=90$ and $n=124$. 
The matrix with $n=90$ matrix is constructed from temperature measurements recorded at monitoring stations across the Pacific Northwest region of the United States \citep{anstreicher_efficient_2020}.
The matrix with $n=124$ 
originate from benchmark datasets 
derived from an environmental monitoring network redesign study 
\cite{hoffman_new_2001}.  
Both matrices are nonsingular and have been widely used as standard benchmarks in the development and testing of MESP algorithms 
{\citep{ko_exact_1995,lee_constrained_1998,anstreicher_using_1999,hoffman_new_2001,lee_linear_2003,anstreicher_masked_2004,burer_solving_2007,anstreicher_maximum-entropy_2018,li_best_2023,chen_generalized_2024}}. 
For each covariance matrix, we generate a series of standard MESP test instances by varying the subset size parameter $s$. 
{The factorization $\bC=\bV^\top\bV$ is computed via the Cholesky decomposition in  NumPy, which returns a lower-triangular matrix $\bV$ satisfying $\bC=\bV^\top\bV$. This factorization is numerically stable for symmetric positive definite matrices and is also used by \cite{li_best_2023}.}

We test the quality of various scaling techniques applied to the linx relaxation, i.e., our double-scaling from \cref{sec:linx-double-scaling} and o-scaling from \cite{anstreicher_efficient_2020} as well as g-scaling developed by 
\cite{chen_generalized_2024} (see \cref{sec:o-g-scaling}). 
In addition, we also compare the quality of scaling bounds with the other state-of-the-art MESP relaxations, namely $\Gamma$ relaxation \citep{nikolov_randomized_2015,li_best_2023} and its variants $\Gamma^c$ and $\Gamma^*$.

Recall that the saddle point  (SP) formulation of the o-scaled linx relaxation leads to the bound of $\Opt_{\textup{linx-o}}$ 
(see \eqref{eq:f-linx-o}). Similarly, $\Opt_{\textup{linx-g}}$ corresponds to the g-scaled linx relaxation 
(see \eqref{eq:f-linx-g}), which also admits a convex-concave SP form. 
For double-scaled linx relaxation, {for practical considerations, instead of using \eqref{eq:double-linx-convex-concave},} we {use a reformulated convex-concave SP representation based on the change of variables $(\bm\rho,\bm\omega) \mapsto (\exp(\bm\rho),\exp(\bm\omega))$ applied to} \eqref{eq:double-scaling-alt}. 
In addition to the scaled linx relaxations, we evaluate the $\Gamma$ relaxation given by \eqref{eq:Gamma_rel} (see \cref{cor:relaxation-h} and \eqref{eq:opt-gamma}), the complementary $\Gamma^c$ relaxation defined in \eqref{eq:opt-gamma-complementary},  as well as the $\Gamma^*$ relaxation, 
which  is defined as
\begin{align*}
  \Opt_{\Gamma^*} \coloneq\min_{\x\in\cX}\max\set{ {\objgg(\x,\bm0)} , {\objggc(\x,\bm0)} }
  = \min_{\x\in\cX}\max_{\alpha\in[0,1]}\set{\alpha {\objgg(\x,\bm0)} + (1-\alpha) {\objggc(\x,\bm0)} }. 
\end{align*}
Note that $\Gamma^*$ relaxation is a mixing of $\Gamma$ and $\Gamma^c$ relaxations (see 
\cite{chen_computing_2023}).
Also, recall that for standard MESP, $\Gamma$ relaxations do not benefit from scaling; see \cite[Theorem 3.4.14]{fampa_maximum-entropy_2022} and \cref{prop:g-scaling-Gamma-invariant}.

Both  $\Opt_{\Gamma}$ and $\Opt_{\Gamma^c}$ are standard convex minimization problems, whereas  $ \Opt_{\Gamma^*}$ and all scaled linx relaxations admit convex-concave SP formulations. We solve the convex minimization problems ($\Opt_{\Gamma}$ and $\Opt_{\Gamma^c}$) using the Mirror Descent algorithm  \citep{nemirovski_problem_1983}. The saddle point problems ($ \Opt_{\Gamma^*}$ and all scaled linx relaxations) are solved using the \emph{parameter-free non-ergodic extragradient} (PF-NE-EG) 
SP algorithm with non-monotone backtracking line search proposed in \cite[Algorithm 2]{shen_parameter-free_2026}; see \cref{sec:numerical:implementation} for full implementation details.

\subsection{Performance of scaling methods on the linx relaxations}\label{sec:linx-scalings}
We first compare the quality of 
o-scaling  ($\Opt_{\textup{linx-o}}$) from \cite{anstreicher_efficient_2020} and g-scaling ($\Opt_{\textup{linx-g}}$) developed by \cite{chen_generalized_2024}  (see \cref{sec:o-g-scaling})  against our double-scaling technique applied to the linx relaxation ($\Opt_{\textup{linx-d}}$) from \cref{sec:linx-double-scaling}. 
\cref{fig:scaling-linx-90,fig:scaling-linx-124,tab:scaling-linx-d90,tab:scaling-linx-d124}  
summarize our numerical results. In these figures, we assess relaxation quality by reporting the integrality gaps of the three methods, $\Opt_{\textup{linx-o}}$, $\Opt_{\textup{linx-g}}$, and $\Opt_{\textup{linx-d}}$, relative to the corresponding integer optimal values reported in \cite{anstreicher_efficient_2020}. 
We also present the algorithm’s solution times (in seconds) for each instance. The detailed numerical results underlying these figures, including the full data tables, are provided in \cref{tab:scaling-linx-d90,tab:scaling-linx-d124} in \cref{sec:app:numerical}.

\begin{figure}[htbp]
  \centering
  \includegraphics[scale=.4]{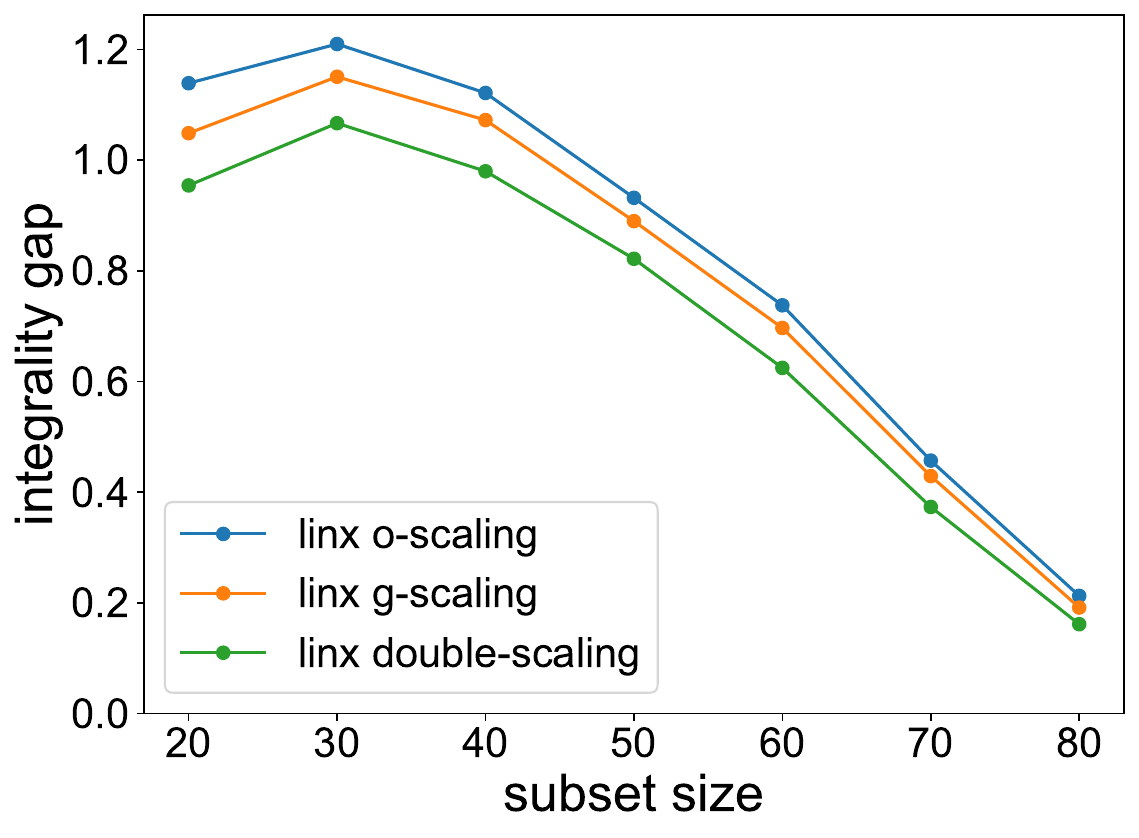}
  \includegraphics[scale=.4]{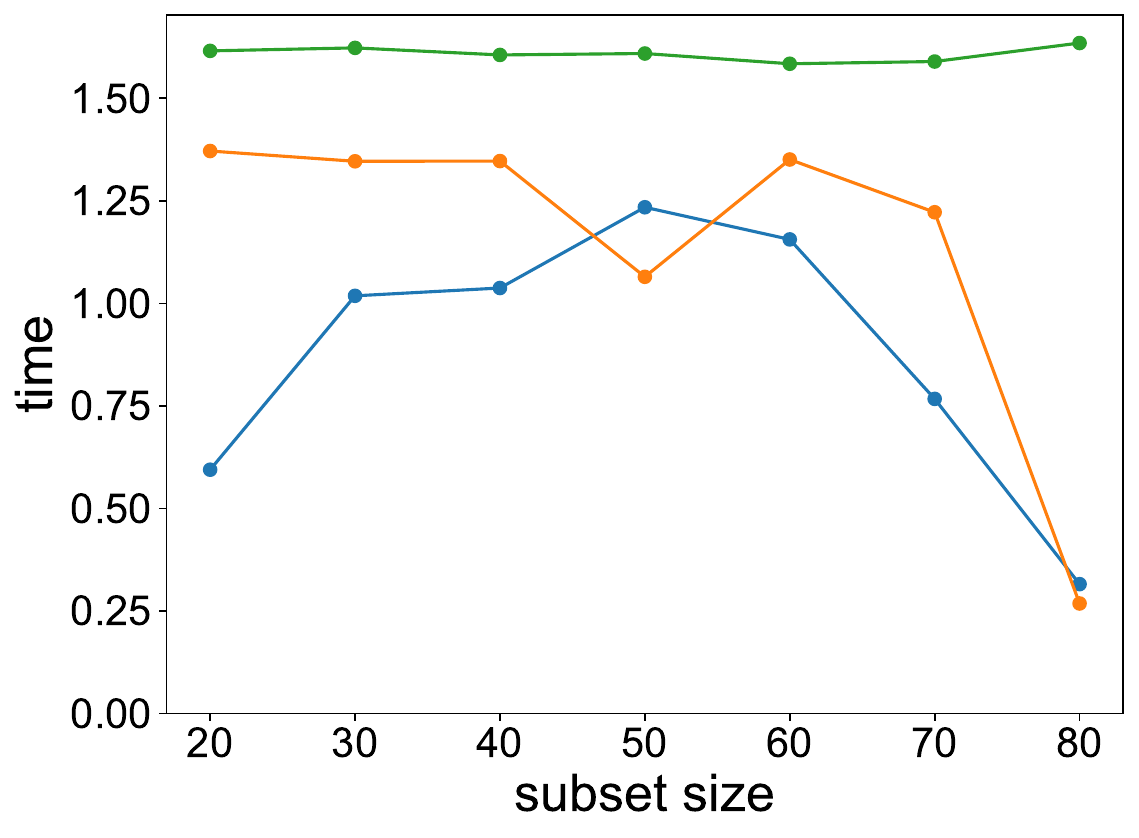}
  \caption{Performance of scaling methods applied to the linx relaxation on instances with $d=90$ and varying subset size $s$. }
  \label{fig:scaling-linx-90}
\end{figure}

\begin{figure}[htbp]
  \centering
  \includegraphics[scale=.4]{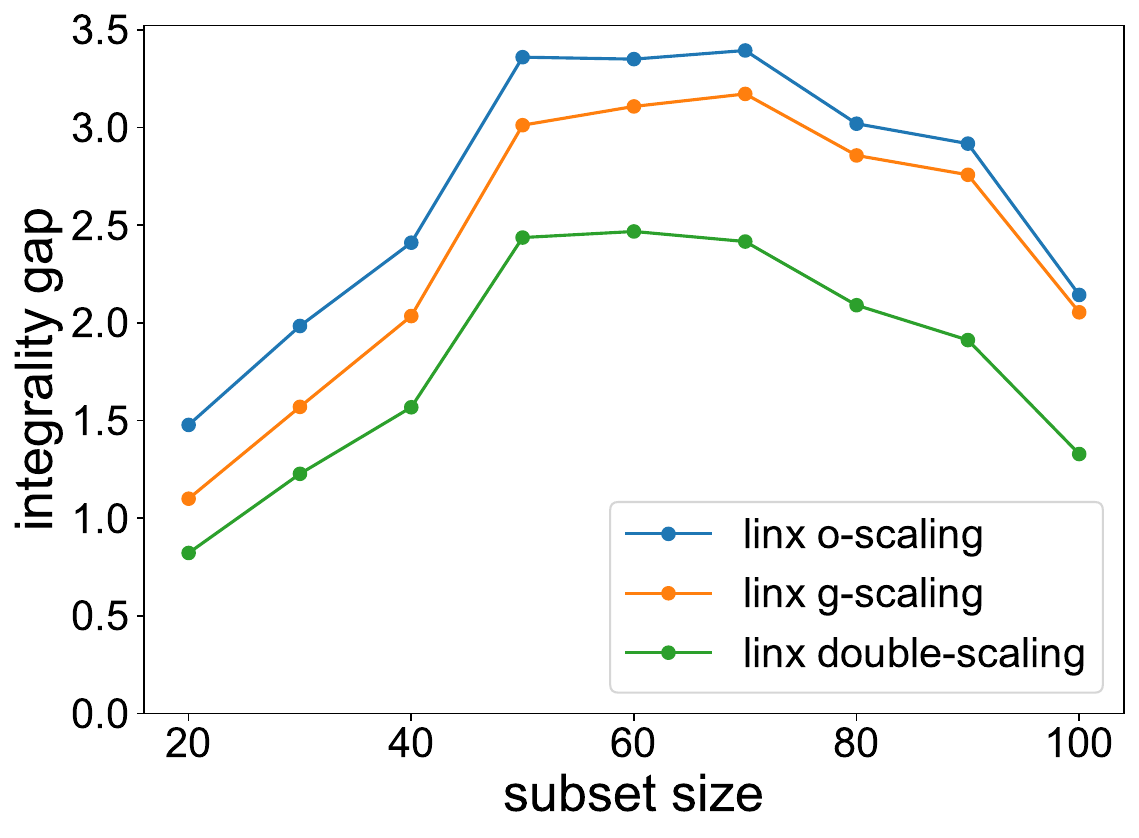}
  \includegraphics[scale=.4]{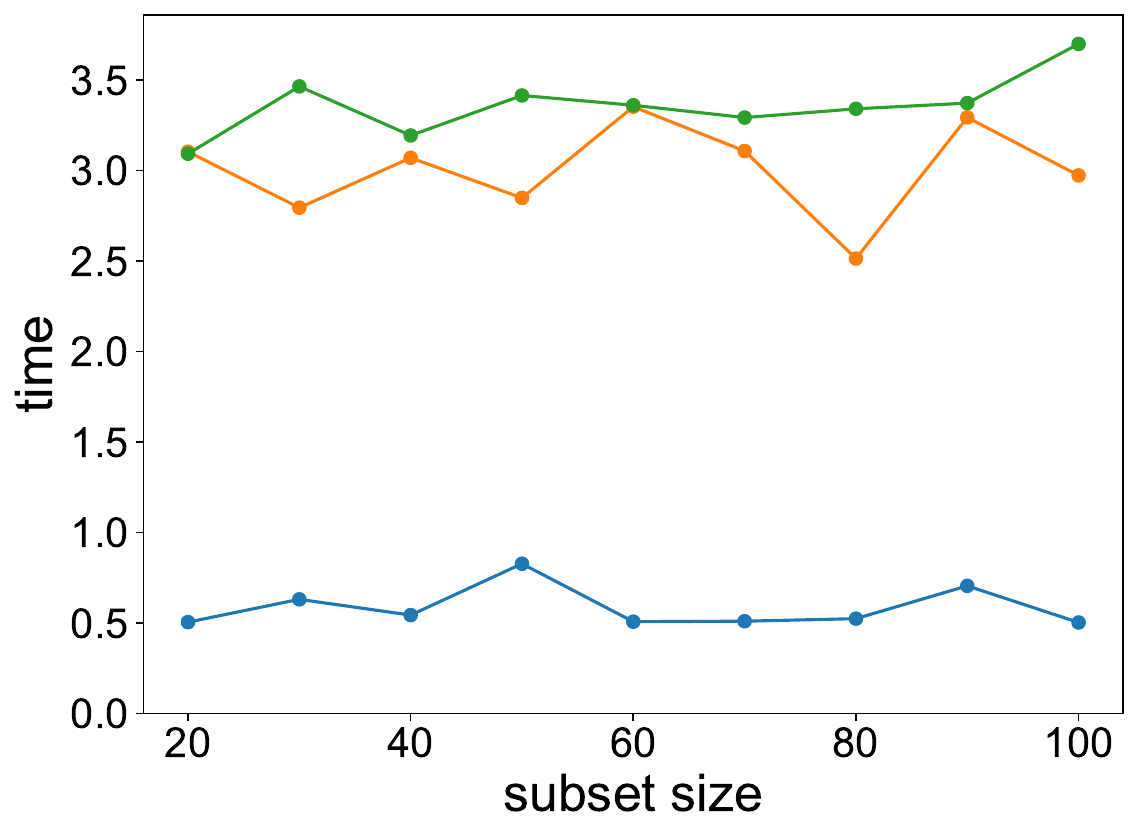}
  \caption{Performance of scaling methods applied to the linx relaxation on instances with $d=124$ and varying subset size $s$.}
  \label{fig:scaling-linx-124}
\end{figure}

In terms of relaxation quality, these results highlight that, for all instances, the double-scaled linx relaxation yields strictly tighter lower bounds for the standard MESP (as a minimization problem) than the g-scaled linx relaxation, which in turn strictly outperforms o-scaling, i.e., $\Opt_{\textup{linx-d}}> \Opt_{\textup{linx-g}}> \Opt_{\textup{linx-o}}$. Moreover, in most cases, particularly for instances with $d=124$, the improvement from double-scaling over g-scaling exceeds the improvement from g-scaling over o-scaling.  
Regarding computational performance, \cref{fig:scaling-linx-90,fig:scaling-linx-124} show that o-scaling is the fastest method, followed by g-scaling and then double-scaling (with the exception of the anomalies at $d=90$, $s=50$ and $s=80$). Although g-scaling incurs a noticeable increase in solution time relative to o-scaling, the additional cost of double-scaling over g-scaling is comparatively modest, especially for the larger instances with $d=124$.

\subsection{Comparison of the scaled linx relaxations and the \texorpdfstring{$\Gamma$}{Gamma} relaxations}

We next compare the quality of the bounds obtained from the scaled linx relaxations with the bounds obtained from $\Gamma$, $\Gamma^c$ and $\Gamma^*$ relaxations. 
These results are presented in \cref{fig:scaling-5algos-90,fig:scaling-5algos-124,tab:scaling-5-d90,tab:scaling-5-d124}. 
Comparing the $\Gamma$ relaxation and its variants, we observe that, for fixed $d$, as the subset size $s$ increases, the quality of the bound $\Opt_{\Gamma^c}$ 
improves, while the one from $\Opt_{\Gamma}$ 
deterioriates for the larger instances. As expected, $\Opt_{\Gamma^*}$ 
is always at least as strong as both $\Opt_{\Gamma}$ and $\Opt_{\Gamma^c}$, 
and in certain cases it 
strictly outperforms the best of the two 
(see e.g., $d=90$ and $s=70$, or $d=124$ and $s=60$). 
When the quality of $\Gamma$ relaxation bounds are compared against the linx relaxation strengthened with various scaling techniques, we observe that neither dominates the other.
For example, for the instances with $d=90$, $\Opt_{\textup{linx-d}}$ outperforms $\Opt_{\Gamma^*}$ when the subset size parameter is $s=50,60,70,80$, and $\Opt_{\textup{linx-d}}$ is almost the same as $\Opt_{\Gamma^*}$ when $s=40$. Furthermore, for the instances with $d=124$, $\Opt_{\textup{linx-d}}$ achieves the best quality among all six bounds at almost all subset size parameters $s$ shown in \cref{fig:scaling-5algos-90,fig:scaling-5algos-124}, except for $s=20$ (in this instance $\Opt_{\Gamma}$ and $\Opt_{\Gamma^*}$ perform better). 
In terms of solution time, 
the $\Opt_{\Gamma^c}$ relaxation is generally faster than $\Opt_{\Gamma}$, though this speed advantage sometimes comes at the expense of bound quality, especially for $d=90$. In contrast, the solution time for $\Opt_{\Gamma^*}$ shows more pronounced variability and can be the most computationally expensive method. 
Our proposed $\Opt_{\textup{linx-d}}$ demonstrates a favorable solution time profile. It is consistently faster than $\Opt_{\Gamma}$ and more stable than $\Opt_{\Gamma^*}$. Although it typically requires more time than $\Opt_{\Gamma^c}$,  this is expected to some extent, as both $\Opt_{\Gamma}$ and $\Opt_{\Gamma^c}$ are solved as standard convex minimization problems using Mirror Descent, whereas the other bounds are computed via the SP algorithm (PF-NE-EG) from \cite{shen_parameter-free_2026}. 

\begin{figure}[htbp]
  \centering
  \includegraphics[scale=.35]{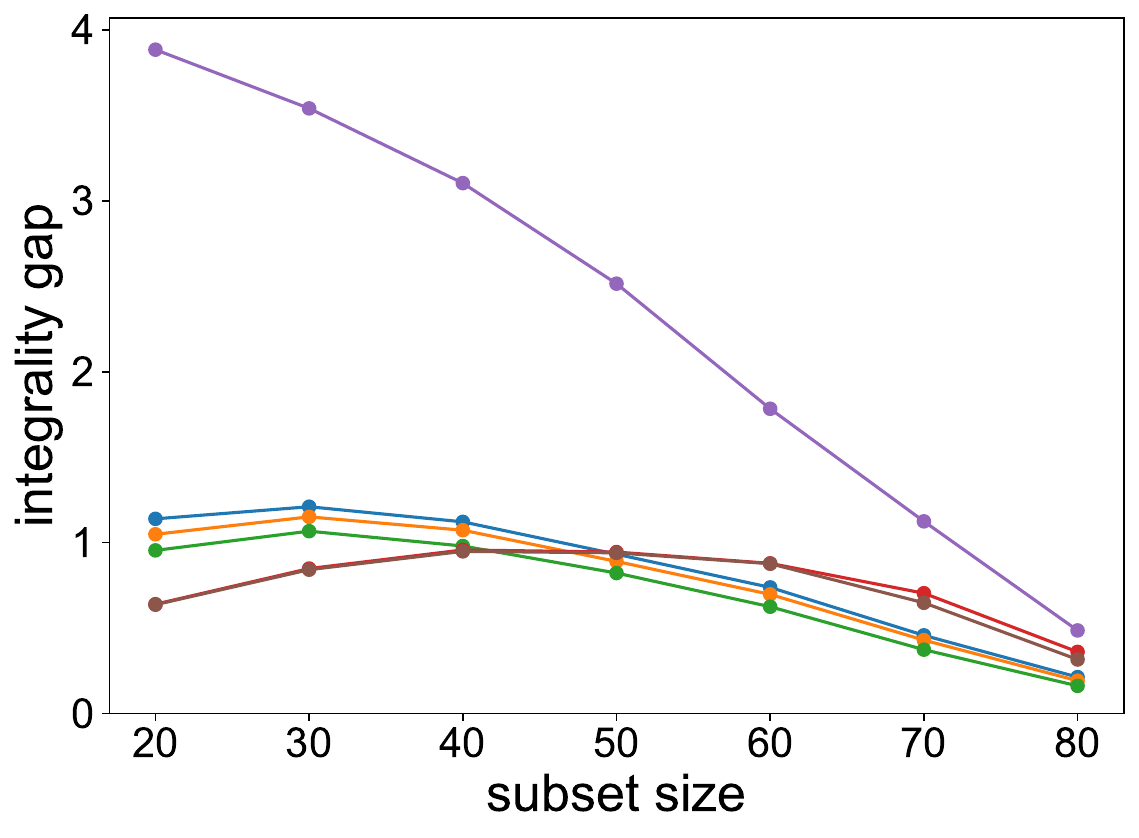}
  \includegraphics[scale=.35]{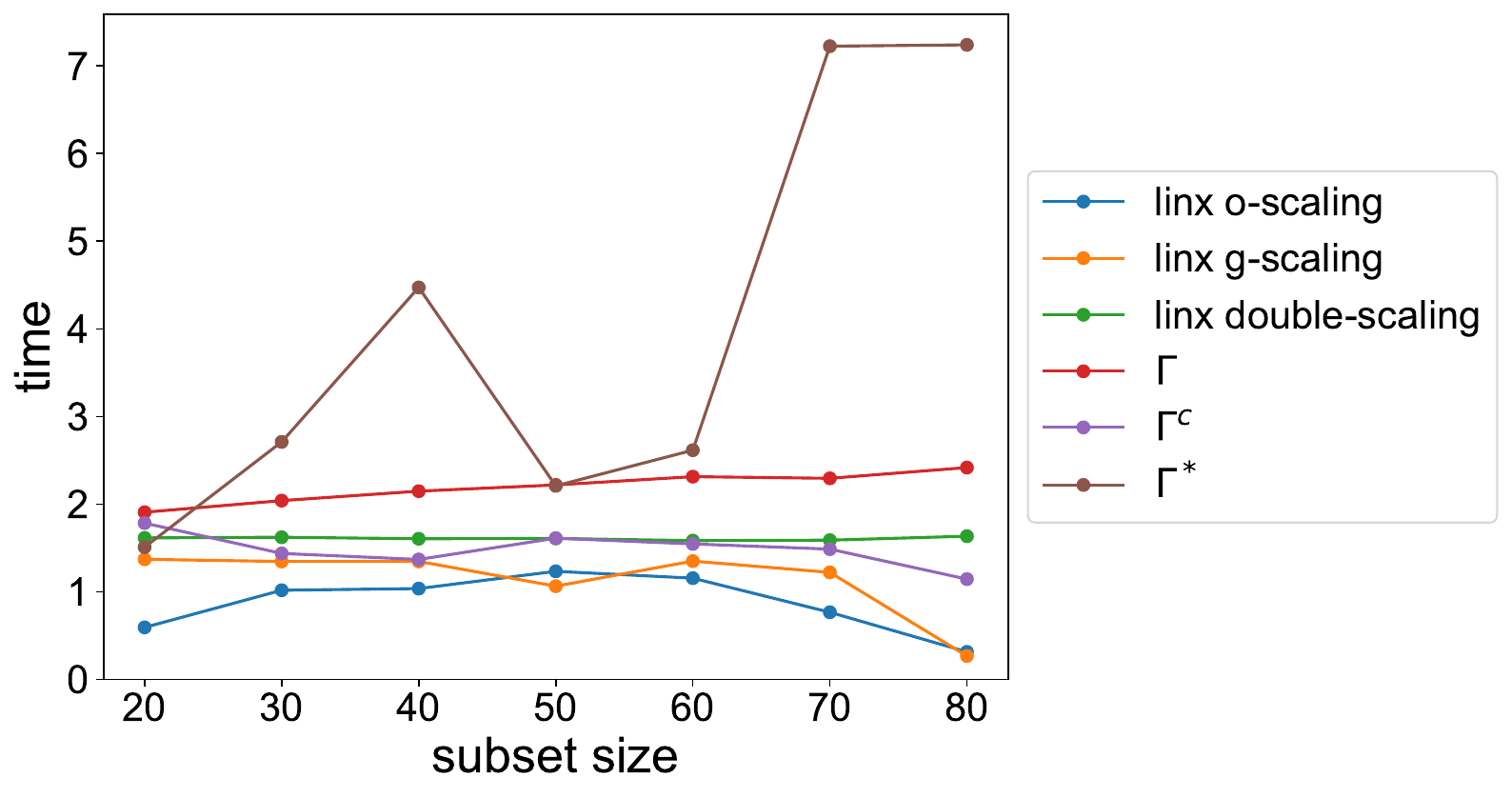}
  \caption{Performance of different MESP relaxations and bound enhancement techniques on instances with $d=90$.}
  \label{fig:scaling-5algos-90}
\end{figure}

\begin{figure}[htbp]
  \centering
  \includegraphics[scale=.35]{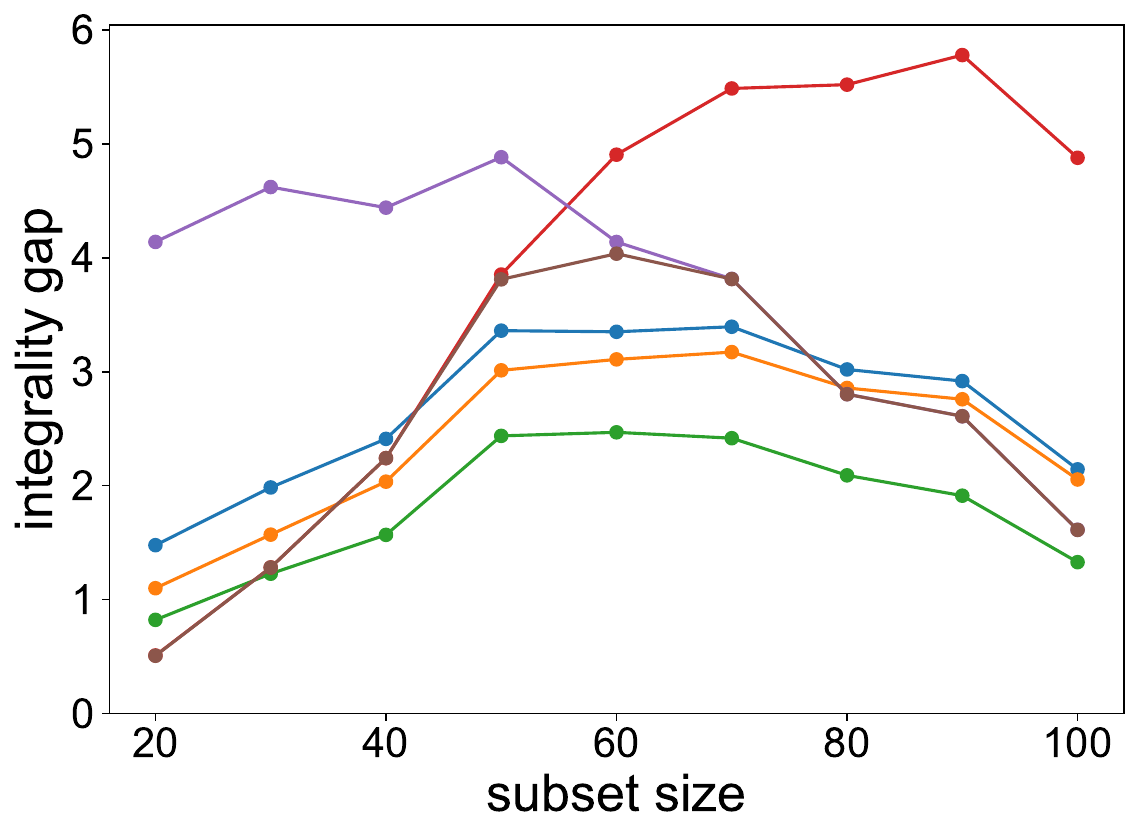}
  \includegraphics[scale=.35]{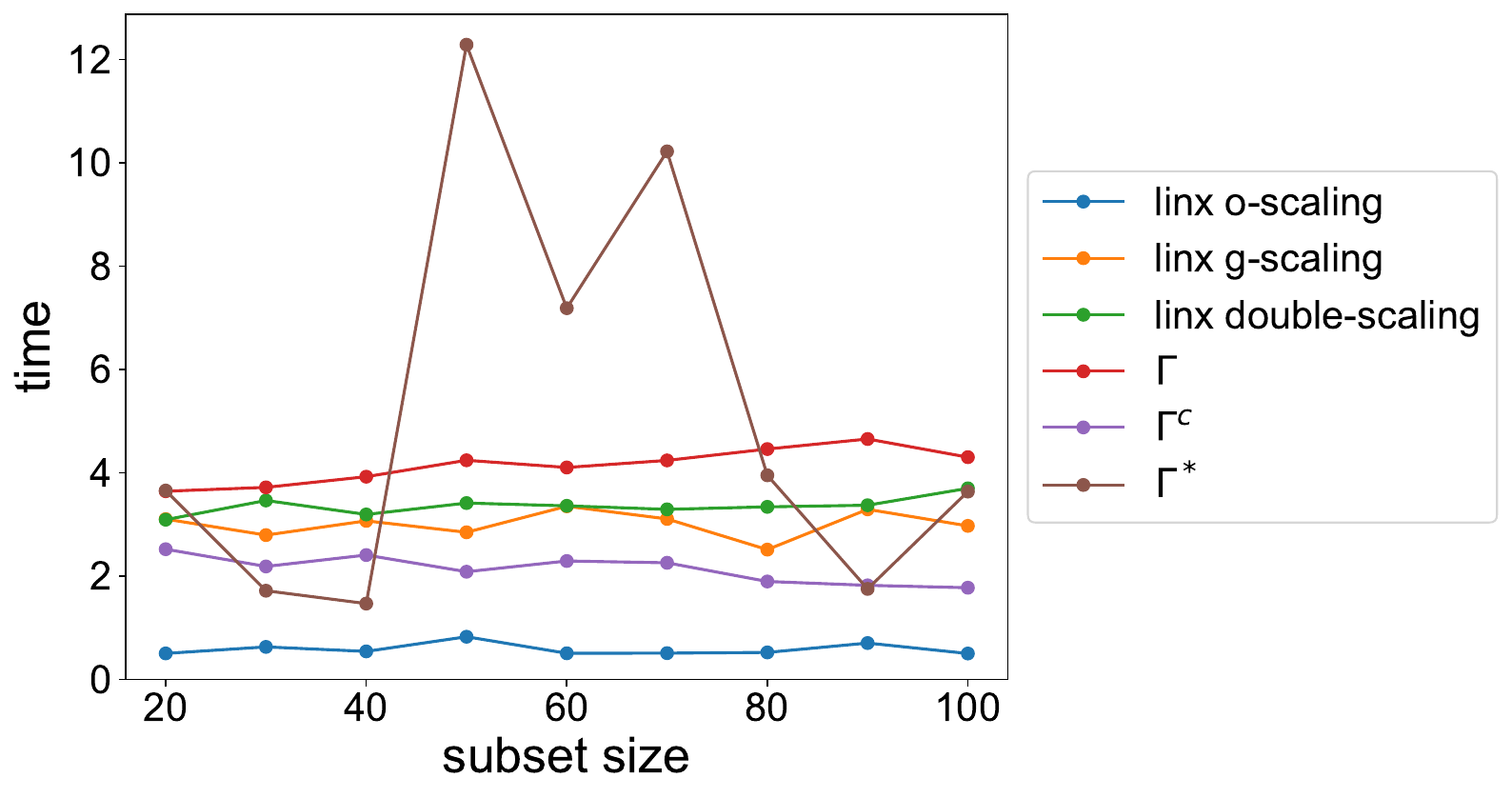}
  \caption{Performance of different MESP relaxations and bound enhancement techniques on instances with $d=124$.}
  \label{fig:scaling-5algos-124}
\end{figure}

%% file: appendix.tex
\newpage
\section{Relegated proofs and examples}
\label{sec:app:proofs}

\subsection{Proof of \texorpdfstring{\cref{cor:f+Omega_s:conv}}{Corollary~\protect\ref{cor:f+Omega_s:conv}}}
\begin{proof}
As $F_s(\bX)=f_s(\bm\lambda(\bX))$, the relation $t\ge F_s(\bX)$ holds if and only if $t\ge f_s(\bm\xi)$, where $\bm\xi\coloneqq \bm\lambda(\bX) \in\R^d_+\cap \K_{o,d}$. Thus, the given representation of $\epi(F_s)$ is justified. 

Next, recall that $\dom(F_s)=\dom(G_s)$ and this domain is simply the set of positive semidefinite matrices $\bX$ with $\rank(\bX) \ge s$, which is convex. 
Given a set $\Omega\subseteq\R^d\times\S^d$, let us define
\[
    \cF(\Omega) \coloneqq \set{ { (t,\bm\xi,\bX)\in\R\times\R^d\times\S^d }
    :
      \begin{array}{>{\displaystyle}l}
        t \geq f_s(\bm\xi), \\
        (\bm\xi,\bX)\in\Omega  
      \end{array}
    } .
\] 
Then, $\epi(F_s)=\Proj_{t,\bX}(\cF(\overline{\Omega}))$.  
Moreover,  the set $\cC \coloneqq \set{{ (t,\bm\xi)\in\R\times \K_{o,d} } :\ t\geq f_s(\bm\xi) }$ is closed, and using this set $\cC$ in \cref{prop:conv-general} we conclude that 
$\cl\conv(\cF(\overline{\Omega})) = \conv(\cF(\overline{\Omega})) = \cF(\widehat{\Omega})$. 
Then, as projection and convex hull operations commute, i.e., for any set $\cD$ we have $\Proj(\conv(\cD))=\conv(\Proj(\cD))$, we deduce that 
$ 
\conv(\epi(F_s)) 
= \conv\left( \Proj_{t,\bX}(\cF(\overline{\Omega})) \right)
= \Proj_{t,\bX}\left(\conv(\cF(\overline{\Omega})) \right)
=\Proj_{t,\bX}\left(\cF(\widehat{\Omega})\right)
.
$ 
Moreover, by definition of $\cF(\widehat{\Omega})$, we conclude
\begin{align*}
& \Proj_{t,\bX}\left(\cF(\widehat{\Omega})\right) \\
&= \set{(t,\bX)\in\R\times\S^d:~ \exists \bm\xi \textup{ s.t. }
        t \geq f_s(\bm\xi),~ (\bm\xi,\bX)\in\widehat{\Omega} } \\
&= \set{(t,\bX)\in\R\times\S^d_+:~ \exists \bm\xi \textup{ s.t. }
        t \geq f_s(\bm\xi),~ \bm\xi\majorize \bm\lambda(\bX), ~\bm\xi \in\K_{o,d} \cap\R^d_+ } \\
&= \set{(t,\bX)\in\R\times\S^d_+:~ 
        t \geq \inf_{\bm\xi}\set{f_s(\bm\xi):~ \bm\xi\majorize \bm\lambda(\bX), ~\bm\xi \in\K_{o,d} \cap\R^d_+} } \\  
&= \set{(t,\bX)\in\R\times\S^d_+:~ 
        t \geq g_s\left( \bm\lambda(\bX) \right) = G_s(\bX) } \\     
&= \epi(G_s).           
\end{align*}
This then justifies both the representation of $\epi(G_s)$ and the claim $\conv(\epi(F_s))=\epi(G_s)$.
\end{proof}

\subsection{Proof of \texorpdfstring{\cref{prop:P-relax-integrality}}{Proposition~\protect\ref{prop:P-relax-integrality}}}
\begin{proof}
Since $\overline{\Omega}\subseteq\widehat{\Omega}$, the ``$\subseteq$'' direction is immediate. To show the reverse direction, consider any $(t,\x)\in\cP(\widehat{\Omega}) \cap \left(\R\times\{0,1\}^d\right)$. Then, $\x\in\cX$ and there exists $\bm\xi\in\K_{o,d}$ such that $\bm\xi\majorize \bm\lambda(\bV\Diag(\x)\bV^\top)$ and $t\ge f_s(\bm\xi)$. Define $\bm\xi':=\bm\lambda(\bV\Diag(\x)\bV^\top)$ so that $\bm\xi'\in\K_{o,d}$ and $\bm\xi\majorize\bm\xi'$. 
Thus, $(\bm\xi',\bV\Diag(\x)\bV^\top)\in\overline{\Omega}$. 
As $\bm\xi\majorize\bm\xi'$,  
by \cref{lem:monotone} we have 
$f_s(\bm\xi)\geq f_s(\bm\xi')$, 
thus we have $t\geq f_s(\bm\xi)\geq f_s(\bm\xi')$. Therefore, $(t,\bm x)\in\cP(\overline{\Omega})$. 
\end{proof}

\renewcommand{\arraystretch}{1}   
\setlength{\arraycolsep}{.5em}

\begin{example}\label{ex:P+Omega_s:not_conv}
In this example, we illustrate that even when $\varphi(\cdot)=-\log(\cdot)$, there exist instances of \eqref{eq:mesp-general} where $\conv\left(\cP(\overline{\Omega})\cap(\R\times\set{0,1}^d)\right) \neq \cP(\widehat{\Omega})$.

  Consider an instance of  \eqref{eq:mesp-general} where $d=4$, $s=2$, $\cX=\{\x\in[0,1]^d:\ones^\top \x=s\}$, and 
  \begin{align*}
    \bV = \begin{pmatrix}
      1 & 1 &   &   \\
        & 1 &   &   \\
        &   & 1 &   \\
        &   &   & 1 
    \end{pmatrix},\ 
    \bC = \begin{pmatrix}
      1 & 1 &   &   \\
      1 & 2 &   &   \\
        &   & 1 &   \\
        &   &   & 1 
    \end{pmatrix}. 
  \end{align*}
  Then, $\cP(\overline{\Omega})\cap(\R\times\{0,1\}^d)$ consists of the following points: 
  \begin{align*}
    t \geq 0,&\qquad \x = (1,1,0,0) \\
    t \geq  0,&\qquad \x = (1,0,1,0) \\ 
    t \geq 0,&\qquad \x = (1,0,0,1) \\ 
    t \geq -\log2,&\qquad \x = (0,1,1,0) \\ 
    t \geq -\log2,&\qquad \x = (0,1,0,1) \\ 
    t \geq 0,&\qquad \x = (0,0,1,1).
  \end{align*}
  Thus, 
  $
  	\min_{t,\x}\{t-x_1:\, (t,\x)\in \conv(\cP(\overline{\Omega}))\}
	= \min_{t,\x}\{t-x_1:\, (t,\x)\in \cP(\overline{\Omega})\} = -1.
  $ 
 In contrast, we claim that $\min_{t,\x}\{t-x_1:\, (t,\x)\in \cP(\widehat{\Omega})\}<-1$. Consider the point
$\bar t=-\log(\tfrac{4+\sqrt2}{4})$ and $\bar \x = (1,  \frac12,  \frac14,  \frac14)$. We claim that $(\bar t, \bar \x)\in \cP(\widehat{\Omega})$. First, note that $\bar\x\in\cX$. Moreover,
  \begin{align*}
     \bar \bX &:=\bV\Diag(\bar \x)\bV^\top = \begin{pmatrix}
      \frac32 & \frac12 &   &   \\
      \frac12 & \frac12 &   &   \\
        &   & \frac14 &   \\
        &   &   & \frac14
    \end{pmatrix} 
    \quad \textup{and} \quad
    \bm\lambda(\bar \bX) = \left(\tfrac{2+\sqrt2}{2},\tfrac{2-\sqrt2}{2},\tfrac14,\tfrac14\right). 
  \end{align*}
  Thus, by taking $\bar{\bm\xi} := \left(\tfrac{2+\sqrt2}{2},\tfrac{3-\sqrt2}{2},0,0\right)$ we observe that $\bar{\bm\xi} \majorize \bm\lambda(\bar \bX)$, $\bar{\bm\xi} \in\K_{o,d}\cap\R^d_+$, $\|\bar{\bm\xi} \|_0=s$, and $\bar t = f_s (\bar{\bm\xi}) = -\sum_{i\in [s]} \log(\bar\xi_i)= -\log(\tfrac{2+\sqrt2}{2})-\log(\tfrac{3-\sqrt2}{2})= -\log\left(\tfrac{2+\sqrt2}{2} \cdot \tfrac{3-\sqrt2}{2} \right) = -\log(\tfrac{4+\sqrt2}{4})<0$. Hence, $\min_{t,\x}\{t-x_1:\, (t,\x)\in \cP(\widehat{\Omega})\}\leq \bar{t} - \bar{x}_1 = -\log(\tfrac{4+\sqrt2}{4}) -1 <-1$. 
  As $\min_{t,\x}\{t-x_1:\, (t,\x)\in \cP(\widehat{\Omega})\} < \min_{t,\x}\{t-x_1:\, (t,\x)\in \conv(\cP(\overline{\Omega})\cap(\R\times\set{0,1}^d))\}$ we conclude that $\conv(\cP(\overline{\Omega})\cap(\R\times\set{0,1}^d)) \neq \cP(\widehat{\Omega})$.
\end{example}

\renewcommand{\arraystretch}{1.5}   
\setlength{\arraycolsep}{1.5pt}

\subsection{Proof of \texorpdfstring{\cref{prop:Q+Omega_s:not_conv}}{Proposition~\protect\ref{prop:Q+Omega_s:not_conv}}}
\label{sec:proof-prop-Q+Omega_s}

\begin{proof}
Suppose $\bC$ with $\rank(\bC)=d$ and $2\le s\le d$ is given. 
Let us construct $(\bar t,\bar{\bm\xi},\bar\x)\in\cQ(\widehat{\Omega})$ that is not in $\conv(\cQ(\overline{\Omega}))$. 
  We first pick an arbitrary $\bar\x\in[0,1]^d$ such that $\ones^\top\bar\x=s$, and define $\bar\bX:=\bV\Diag(\bar\x)\bV^\top$. We will next identify a vector $\bar{\bm\xi}\in\K_{0,d}$ that majorizes $\bm\lambda (\bar\bX)$. Let us define $\bar{\bm\xi}$ as follows: 
  \begin{align*}
    \bar\xi_j := \begin{cases}
      \tr(\bar\bX) - (s-1)\epsilon, & \textup{if } j\in\set{1}, \\
      \epsilon, & \textup{if } j\in(1,s], \\
      0, & \textup{if } j\in(s,d], \\
    \end{cases}
  \end{align*}
  where $0<\epsilon\leq\lambda_s(\bar\bX)$ will be determined later. Note that as $\bar\x\ge0$, we have
  \[
  	\lambda_s(\bar\bX)=\lambda_s(\bV\Diag(\bar\x)\bV^\top)=\lambda_s(\Diag(\sqrt{\bar\x}) \bC \Diag(\sqrt{\bar\x}) )>0,
  \] 
  where the concluding inequality follows from $\bar\x\in\cX\subseteq\{\x\in[0,1]^d:\, \ones^\top\x=s\}$ which implies that $\rank(\Diag(\bar\x))\geq s$ and consequently using Sylvester's rank inequality together with $\rank(\bC)=d$ we deduce that $\rank(\Diag(\sqrt{\bar\x}) \bC \Diag(\sqrt{\bar\x}) )\geq\rank(\Diag(\bar\x))\geq s$. Thus, we can indeed pick such an $\epsilon \in (0,\lambda_s(\bar\bX)]$. Then we have $\ones^\top\bar{\bm\xi}=\ones^\top\bm\lambda(\bar\bX)=\tr(\bar\bX)$ and 
  \begin{align*}
    \sum_{i\in[j,d]}\bar\xi_i \leq \sum_{i\in[j,d]}\lambda_i(\bar\bX), \quad\forall j\in[d],
  \end{align*}
  i.e., $\bar{\bm\xi} \majorize \bm\lambda(\bar\bX)$. Finally, define $\bar t :=f_s(\bar{\bm\xi})<+\infty$ (because $\bar\xi_j>0$ for $j\in[s]$). 
  Then, as $\bar{\bm\xi}\in\K_{0,d}$, 
  we have $(\bar t,\bar{\bm\xi},
  \bar\x)\in\cQ(\widehat{\Omega})$. 
  Assume for contradiction that there exists some convex combination weights $\alpha_\ell$ such that $(\bar t,\bar{\bm\xi},
  \bar\x)=\sum_{\ell}\alpha_{\ell}(\bar t_{\ell},\bar{\bm\xi}^{\ell},
  \bar\x^{\ell})$ where $(\bar t_{\ell},\bar{\bm\xi}^{\ell}, 
  \bar\x^{\ell})\in\cQ(\overline{\Omega})$ for all $\ell$. In particular, $\epsilon=\bar{\xi}_s=\sum_{\ell}\alpha_{\ell}\bar{\xi}_s^{\ell}$, which implies that at least one of the $\bar{\bm\xi}^{\ell}$'s satisfies $\bar\xi_{s}^{\ell}\leq\epsilon$. On the other hand, as $(\bar t_{\ell},\bar{\bm\xi}^{\ell},
  \bar\x^{\ell})\in\cQ(\overline{\Omega})$ we deduce 
  \begin{align*}
   	\bar\xi_{s}^{\ell} & =\lambda_s(\bV\Diag(\bar\x^{\ell})\bV^\top) \\
	& = \lambda_s\! \left(\Diag(\sqrt{\bar\x^\ell}) \bC \Diag(\sqrt{\bar\x^\ell}) \right) \\
	& \geq \lambda_s\! \left( \lambda_d(\bC) \Diag(\bar\x^\ell) \right) \\
	& = \lambda_d(\bC)\ \lambda_s\! \left( \Diag(\bar\x^\ell) \right) \\
	& = \lambda_d(\bC)\ \bar x_{[s]}^{\ell}, 
  \end{align*}
  where the inequality follows from $\bC \succeq \lambda_d(\bC) \bI$ and thus $ \Diag(\sqrt{\bar\x^\ell})\bC  \Diag(\sqrt{\bar\x^\ell}) \succeq \lambda_d(\bC) \Diag(\bar\x^\ell)$. 
  Finally, to complete the proof we claim that $\bar x^{\ell}_{[s]}\geq\frac{1}{d}$. If not, we will have $\bar x^{\ell}_{[s]}<\frac{1}{d}$, which will lead to $\ones^\top\bar\x^{\ell}<(s-1)\times1+(d-s+1) \times \frac{1}{d}\leq s$, a contradiction. Therefore, putting things together we observe that $\epsilon\geq\bar\xi_s^{\ell}\geq\lambda_d(\bC)\frac{1}{d}$. Then, we can set  $0<\epsilon<\min\{\frac{1}{d}\lambda_d(\bC),\lambda_s(\bar\bX)\}$ and this leads to a contradiction. Hence, we conclude that $(\bar t,\bar{\bm\xi}, 
  \bar\x)\notin\conv(\cQ(\overline{\Omega}))$ for $\epsilon$ selected as such. 
\end{proof}

\subsection{Proof of \texorpdfstring{\cref{lem:exact-s-1}}{Lemma~\protect\ref{lem:exact-s-1}}}
\label{sec:proof-lem:exact-s-1}
\begin{proof}
  Recall that, by \cref{rem:sparse}, including the constraint $\|\bm\xi\|_0\leq s$ in the convex relaxation given by \eqref{eq:Gamma_rel} does not change its optimum value. Then, when $s=1$, we can simplify the relaxation given in \eqref{eq:Gamma_rel} as
  \begin{align*}
    & \min_{\x\in\cX,\bm\xi\in\R^d_+} \set{\varphi(\xi_1): 
      \begin{array}{>{\displaystyle}l} 
        \bm\xi\majorize\bm\lambda\left(\bV\Diag(\x)\bV^\top\right) \\
        \bm\xi \in \K_{o,d},\ \|\bm\xi\|_0\leq 1 
      \end{array}
    } \\
    &= \min_{\x\in\cX,\xi_1\in\R_+} \set{\varphi(\xi_1): 
      \xi_1 = \tr\left(\bV\Diag(\x)\bV^\top\right) 
    } \\ 
    &= \min_{\x\in\cX} \set{\varphi\left(\tr\left(\bV\Diag(\x)\bV^\top\right)\right) } \\
    &= \varphi\left(\max_{\x\in\cX} \set{\tr\left(\bV\Diag(\x)\bV^\top\right)} \right) \\
    &= \varphi\left(\max_{\x\in\cX} \set{\tr(\Diag(\x)\bC)} \right), 
  \end{align*}
  where the first equality follows from the majorization relation and $\|\bm\xi\|_0\leq 1$, and the third equality follows from the assumption that $\varphi(\cdot)$ is a nonincreasing function. 
  Thus, when $s=1$, solving \eqref{eq:Gamma_rel} reduces to simply solving the linear program given by $\max_{\x} \set{\tr(\Diag(\x)\bC): x\in\cX }$. 
  Since $\bA$ is totally unimodular and $\b$ is an integer vector such that $\cX\neq\emptyset$, by \cite[Theorem III.1.2.2]{nemhauser_integer_1988}, the optimal solution of this linear program will be integral. Then, by \cref{prop:relaxation-quality}, this implies that any optimum solution of the relaxation \eqref{eq:Gamma_rel} is an optimal solution of the original problem \eqref{eq:mesp-general} and vice versa. 
\end{proof}

\subsection{Proof of \texorpdfstring{\cref{prop:linx-complementary}}{Proposition~\protect\ref{prop:linx-complementary}}}
\label{sec:linx-complementary}
\begin{proof}
Recall $\cX_{d-s}\coloneqq\{\check\x\in[0,1]^d:\ones^\top\check\x=d-s\}$. Then, starting from the linx relaxation of the complementary problem given in \eqref{eq:linx-complementary} and applying double scaling technique to it results in 
  \begin{align*}
    &\min_{\check\x\in\cX_{d-s}} \max_{\bm\gamma,\bm\mu\in\R^d_{++}} \left\{-\frac12\log\det \left(\bC^{-1}\Diag(\bm\gamma)\Diag(\check\x)\bC^{-1}+\Diag(\bm\mu)(\bI_d-\Diag(\check\x)) \right) - \log\det\bC \vphantom{\sum_{i\in[d]}}\right. \\
    &\hspace{10em} \left. + \frac12\sum_{i\in[d]}\check{x}_i\log\gamma_i + \frac12\sum_{i\in[d]}(1-\check{x}_i)\log\mu_i \right\} \\
    &= \min_{\check\x\in\cX_{d-s}} \max_{\bm\gamma,\bm\mu\in\R^d_{++}} \left\{-\frac12\log\det\left( \Diag(\bm\gamma)\Diag(\check\x)+\bC\Diag(\bm\mu)(\bI_d-\Diag(\check\x))\bC \right) \vphantom{\sum_{i\in[d]}}\right. \\
    &\hspace{10em}\left. + \frac12\sum_{i\in[d]}\check{x}_i\log\gamma_i + \frac12\sum_{i\in[d]}(1-\check{x}_i)\log\mu_i \right\} \\
    &= \min_{\x\in\cX} \max_{\bm\gamma,\bm\mu\in\R^d_{++}} \left\{-\frac12\log\det \left( \Diag(\bm\gamma)\Diag(\ones-\x)+\bC\Diag(\bm\mu)\Diag(\x)\bC \right) \vphantom{\sum_{i\in[d]}}\right. \\
    &\hspace{10em}\left. + \frac12\sum_{i\in[d]}(1- x_i)\log\gamma_i + \frac12\sum_{i\in[d]} x_i\log\mu_i \right\}, 
  \end{align*}
  where the first equality follows from the multiplicative property of the determinant and the second equality from a change of variables of $\x\:=\ones - \check\x$ (note that $\check\x\in\cX_{d-s}$ if and only if $\x\in\cX$). As a result, the last problem above is exactly the same problem as the one obtained from double-scaling applied to linx relaxation \eqref{eq:double-scaling}. 
\end{proof}

\subsection{Proof of \texorpdfstring{\cref{lem:distinct-eigenvalues}}{Lemma~\protect\ref{lem:distinct-eigenvalues}}}
\label{sec:distinct-eigenvalues}
\begin{proof}
  First, consider $m=d$. 
  Let $\bar\nu_1 > \bar\nu_2 > \cdots > \bar\nu_d$. 
  By the Courant-Fischer Theorem, 
  \begin{align*}
    \lambda_i(\bM(\bar{\bm\nu}))
    &= \max_{S\subseteq\R^d,\dim(S)=i}\min_{\a\in S,\a\neq\bm0}\frac{\a^\top\bM(\bar{\bm\nu})\a}{\a^\top\a} \\
    &= \max_{S\subseteq\R^d,\dim(S)=i}\min_{\a\in S,\a\neq\bm0}\frac{\a^\top\bB^\top\Diag(\exp(\bar{\bm\nu}))\bB\a}{\a^\top\a} \\
    &= \max_{S\subseteq\R^d,\dim(S)=i}\min_{\bB^{-1}\a\in S,\a\neq\bm0}\frac{\a^\top\Diag(\exp(\bar{\bm\nu}))\a}{\a^\top\bB^{-\top}\bB^{-1}\a} \\
    &\geq \min_{\a\in\operatorname{span}\set{\e_1,\dots,\e_i}\setminus\set{0}}\frac{\a^\top\Diag(\exp(\bar{\bm\nu}))\a}{\a^\top(\bB\bB^\top)^{-1}\a} \\
    &\geq {e^{\bar\nu_i}}{\lambda_d(\bB\bB^\top)}, 
  \end{align*}
  {
  where the last inequality follows from the facts that for any $\bar{\nu}$ with nonincreasing coordinates and any vector $\a$ supported on only the first $i$ coordinates, we have $\a^\top\Diag(\exp(\bar{\bm\nu}))\a \geq e^{\bar\nu_i} \sum_{\ell\in[i]} a_\ell^2 =  e^{\bar\nu_i} \|\a\|_2^2$ and that $\a^\top(\bB\bB^\top)^{-1}\a \le \lambda_1\left( (\bB\bB^\top)^{-1} \right) \|\a\|_2^2 = \frac{1}{\lambda_d (\bB\bB^\top)} \|\a\|_2^2$. }
  
  Similarly, another form of the Courant-Fischer Theorem leads to 
  \begin{align*}
    \lambda_i(\bM(\bar{\bm\nu}))
    &= \min_{S\subseteq\R^d,\dim(S)=d-k+1}\max_{\a\in S,\a\neq\bm0}\frac{\a^\top\bM(\bar{\bm\nu})\a}{\a^\top\a} \\
    &\leq \max_{\a\in\operatorname{span}\set{\e_i,\dots,\e_d}\setminus\set{0}}\frac{\a^\top\Diag(\exp(\bar{\bm\nu}))\a}{\a^\top(\bB\bB^\top)^{-1}\a} 
    \leq {e^{\bar\nu_i}}{\lambda_1(\bB\bB^\top)}. 
  \end{align*}
 
  Therefore, when $\bar\nu_1, \dots, \bar\nu_d$ are sufficiently dispersed such that $\frac{\bar\nu_{i+1}}{\bar\nu_i}>\log\frac{\lambda_1(\bB\bB^\top)}{\lambda_d(\bB\bB^\top)}$, the eigenvalues are distinct. For the case where $m>d$, it suffices to take $\bar\nu_i\ll1$ for $i\in[d+1,m]$ and apply matrix perturbation theory. 
\end{proof}

\subsection{Proof of \texorpdfstring{\cref{prop:g-scaling-Gamma-invariant}}{Theorem~\protect\ref{prop:g-scaling-Gamma-invariant}}}
\label{sec:g-scaling-Gamma-invariant}

\begin{proof}
Recall $\objgg(\x,\bm\rho) = G_s\!\big(\bV\Diag(\exp({\bm\rho}))\Diag(\x)\bV^\top\big) + \x^\top \bm\rho$. Let $\x^*\in\arg\min_{\x\in\cX} \objgg(\x,\bm0)$ and define
\[
\bM^*:=\bV\Diag(\x^*)\bV^\top =\bQ^*\Diag(\bm\lambda^*){\bQ^*}^\top .
\]
Note $\bm\lambda^*\ge\bm0$. Let $\bm\beta^*:=\bm\beta(\bm\lambda^*)$ according to \eqref{eq:beta-map}. In particular, the definition of $\bm\beta^*$ guarantees 
$\bm\beta^*\geq\bm\lambda^*$ and $\bm\beta^*>\bm0$, because by \cref{lem:index-k} we have $\beta^*_i=\lambda^*_i$ for $i\in[k]$, $\beta^*_i\geq\lambda^*_{k+1}\geq\lambda^*_i$ for $i\in[k+1,d]$, and moreover, $\beta^*_i\geq\lambda^*_{k+1}\geq\lambda^*_s>0$ for $i\in[d]$ since $k+1\leq s$ and $\|\bm\lambda^*\|_0=\rank(\bM^*)\geq\|\x^*\|_0=s$. 

The first part of the proof will be dedicated to establishing the identity $\x^* + \x^* \circ \g^* = 0$ for a subgradient $\g^*$ of $\objgg(\x,\bm0)$ at $\x^*$. While the underlying logic is similar to \cite[Theorem 6.iv]{chen_generalized_2024}, we give a direct proof using standard convex optimization arguments instead of relying on generalized differentiability. 
The key insight, however, lies in the second part, building upon the convex-concavity structure of $\objgg$ we established in \cref{prop:gamma-g-scaling-convexity}.

As $G_s(\bX)=g_s(\bm\lambda(\bX))$ is a convex spectral function and $\varphi(\cdot)=-\log(\cdot)$, by \cref{lem:subdifferential}, a subgradient of $G_s$ at $\bM^*$ is given by
\[
\bY^* \coloneqq -\,\bQ^*\Diag(\bm\beta^*)^{-1}{\bQ^*}^\top.
\]
By the chain rule for the linear map $\x \mapsto \bV\Diag(\x)\bV^\top$, we obtain a subgradient of $\x \mapsto \objgg(\x,\bm0)$ at $\x^*$ as
\[
\g^* \coloneqq \diag\!\big(\bV^\top \bY^* \bV\big)
= -\diag\!\big(\bV^\top \bQ^*\Diag(\bm\beta^*)^{-1}{\bQ^*}^\top \bV\big).
\]
Since $\x^*$ is the minimizer of $\objgg(\x,\bm0)$, which is a convex function over $\cX$,  there exists a subgradient $\bar\g$ of $\objgg(\x,\bm0)$ at $\x^*$ satisfying the variational inequality
\[
\bar\g^\top(\x-\x^*) \ge 0, \qquad \forall \x\in\cX.
\]
We claim that indeed we can take $\bar\g=\g^*$. 
By \cref{lem:subdifferential}, all subgradients of $G_s$ at $\bM^*$ have the form
\[
{\bY(\bm\eta)} \coloneqq -\,\bQ^*(\Diag(\bm\beta^*)^{-1} + \Diag(\bm\eta)){\bQ^*}^\top,
\]
{for some} $\bm\eta\in\R^d_+$ and $\eta_i=0$ for all $i\in[r]$ where $r=\rank(\bM^*)$; that is $\bY^*- { \bY(\bm\eta) } \succeq 0$ and $\bY^*-{ \bY(\bm\eta) } $ is supported entirely on the nullspace of $\bM^*$, i.e., $(\bY^*- { \bY(\bm\eta) } )\u=0$ for all $\u\in\range(\bM^*)$.  
Also, using the chain rule, we conclude that there exists $\bar{\bm\eta}\in\R^d_+$ and $\bar{\eta}_i=0$ for all $i\in[r]$ such that
\[
\bar{\g} 
= -\diag\!\big(\bV^\top \bQ^* ( \Diag(\bm\beta^*)^{-1} + \Diag(\bar{\bm\eta}) ) {\bQ^*}^\top \bV\big).
\]
Now, by defining {$\bar{\bY} \coloneqq \bY(\bar{\bm\eta})$ and}  $S^*\coloneqq \set{i\in[d]:\, x^*_i>0 }$, we observe that $\bM^* = \sum_{i\in S^*} x^*_i \bV \e_i \e_i^\top \bV^\top$ and thus $\range(\bM^*) = \spann\{\bV \e_i: i\in S^*\}$. Then, together with $(\bY^*-\bar{\bY})\u=0$ for all $\u\in\range(\bM^*)$, we deduce for all $i\in S^*$ that $(\bY^*-\bar{\bY})\bV\e_i=0$ and so 
$
0= \e_i^\top \bV^\top (\bY^*-\bar{\bY})\bV\e_i= \e_i^\top \bV^\top \bQ^*  
{\Diag(\bar{\bm\eta})} 
 {\bQ^*}^\top \bV \e_i
$. 
Hence, we conclude that $g^*_i =\bar{g}_i$ for all $i\in S^*$. In addition, for any $i\not\in S^*$, we have $g^*_i - \bar{g}_i = \e_i^\top \bV^\top (\bY^*-\bar{\bY})\bV\e_i \ge 0$ as  $\bY^*-\bar{\bY}\succeq 0$. Then, for any $\x\in\cX$ we arrive at
\[
(\g^*-\bar{\g})^\top (\x-\x^*) = \sum_{i\not\in S^*} (g^*_i - \bar{g}_i) (x_i - x^*_i) =  \sum_{i\not\in S^*} (g^*_i - \bar{g}_i) x_i \ge 0,
\]
where the last equality follows from $x^*_i=0$ for all $i\not\in S^*$ and the last inequality follows from the definition of $\x\in\cX$ and  $g^*_i - \bar{g}_i\ge0$ for all $i\not\in S^*$. Therefore, we have
\[
(\g^*)^\top(\x-\x^*) \ge (\bar\g)^\top(\x-\x^*) \ge 0, \qquad \forall \x\in\cX.
\]

We will next construct a point $\widehat{\x}\in\cX$ that will enable our analysis. Define
\[
\bX^* \coloneqq \Diag(\x^*)^{1/2}\bV^\top \bQ^*\Diag(\bm\beta^*)^{-1}{\bQ^*}^\top \bV\Diag(\x^*)^{1/2}
\succeq 0,
\]
and let 
$ 
\widehat{\x} \coloneqq \diag(\bX^*).
$ 
Then, for all $i\in[d]$ we have
\[
\widehat{x}_i
= x_i^* \big(\bV^\top \bQ^*\Diag(\bm\beta^*)^{-1}{\bQ^*}^\top \bV\big)_{ii}
= - x_i^* g_i^*.
\]
We claim that $\widehat{\x} \in \cX$. 
First, $\widehat{x}_i \ge 0$ for all $i\in[d]$ since $\bX^*\succeq 0$ as $\bm\beta^*>\bm0$. 
Next, note that $\bX^*=\bZ^\top\bZ$, where $\bZ\coloneqq \Diag(\bm\beta^*)^{-1/2}{\bQ^*}^\top \bV\Diag(\x^*)^{1/2}$. As the nonzero eigenvalues of $\bX^*=\bZ^\top\bZ$ and the matrix $\bZ\bZ^\top= \Diag(\bm\beta^*)^{-1/2}{\bQ^*}^\top \bV\Diag(\x^*) \bV^\top \bQ^* \Diag(\bm\beta^*)^{-1/2} = \Diag(\bm\beta^*)^{-1} \Diag(\bm\lambda^*)$ are the same, we conclude that the eigenvalues of $\bX^*$ are 
\[
\lambda_i(\bX^*) = \frac{\lambda_i^*}{\beta_i^*},\qquad\forall i\in[d].
\]
Also, by construction of $\bm\beta^*$, we have $0 \le \lambda_i^*/\beta_i^* \le 1$, hence $\widehat{x}_i \le 1$.
Finally, using the definition of $\bm\beta^*$, we deduce
\[
\ones^\top \widehat{\x}
=
\tr(\bX^*)
=\sum_{i\in[d]} \lambda_i(\bX^*)
=\sum_{i\in[d]} \frac{\lambda_i^*}{\beta_i^*}
=s.
\]
Thus, indeed $\widehat{\x}\in\cX$.

Since $\widehat{\x}\in\cX$ and $\x^*$ minimizes $(\g^*)^\top \x$ over $\cX$, we have
$ 
(\g^*)^\top(\widehat{\x}-\x^*) \ge 0.
$ 
Substituting $\widehat{x}_i = -x_i^* g_i^*$ leads to
\begin{align*}
& 0 \le \sum_{i\in[d]} g_i^*(-x_i^* g_i^* - x_i^*)
= -\sum_{i\in[d]}  x_i^*(g_i^*)^2 - \sum_{i\in[d]}  x_i^* g_i^* \\
\iff & \sum_{i\in[d]}  x_i^*(g_i^*)^2 \le -\sum_{i\in[d]}  x_i^* g_i^* = \ones^\top \widehat{\x} = s.
\end{align*}
On the other hand,
\[
s^2
= \Big(\sum_{i\in[d]}  x_i^*(-g_i^*)\Big)^2
\le \Big(\sum_{i\in[d]}  x_i^*\Big)\Big(\sum_{i\in[d]} x_i^*(g_i^*)^2\Big)
= s \sum_{i\in[d]}  x_i^*(g_i^*)^2
\le s^2.
\]
Thus, equality holds throughout, and equality in Cauchy-Schwarz implies that $-g_i^*$ is constant over the support of $\x^*$. Let $-g_i^* = \eta$ for all $i$ such that $x_i^*>0$. 
Then
\[
s
= \sum_{i\in[d]}  \widehat{x}_i
=\sum_{i\in[d]}  x_i^*(-g_i^*)
= \eta \sum_{i\in[d]}  x_i^*
= \eta s,
\]
so $\eta=1$. Therefore, 
\begin{equation}\label{eq:opt-gamma-sol-structure}
\x^* \circ \g^* = -\x^*.
\end{equation}

{
Consider $\bm\rho \mapsto \objgg(\x^*,\bm\rho)$. We claim that using the same subgradient $\bY^*$, $\objgg(\x^*,\bm\rho)$ is differentiable with respect to $\bm\rho$, and its gradient at $\bm\rho=\bm0$ is
\[
\x^* + \x^* \circ \g^*.
\]
To see this, by fixing $\x=\x^*$ and defining
\[
\bX(\bm\rho) := \bV\Diag(\exp({\bm\rho}))\Diag(\x^*)\bV^\top,
\]
we have $\objgg(\x^*,\bm\rho) = G_s(\bX(\bm\rho)) + (\x^*)^\top \bm\rho$. As $\x^*$ has at least $s$ nonzero entries, $\rank(\bX(\bm\rho))\ge s$ for all $\bm\rho\in\R^d$ and so $\dom(G_s(\bX(\cdot))=\R^d$. By \cref{prop:gamma-g-scaling-convexity}, $G_s(\bX(\bm\rho))$ is concave in $\bm\rho$, and thus is continuous over its domain $\R^d$. 
Note also that the mapping $\bm\rho \mapsto \bX(\bm\rho)$ is smooth. Hence, by \cite[Theorem 10.6]{rockafellar_variational_1998}, we have
\begin{align*}
  \hat\partial_{\bm\rho} G_s(\bX(\bm\rho)) 
  \supseteq \set{\z\in\R^d: z_i=\left\langle\bY, \frac{\partial\bX}{\partial\rho_i}(\bm\rho)\right\rangle, \bY\in\partial G_s(\bX(\bm\rho))}, 
\end{align*}
where $\hat\partial_{\bm\rho} G_s(\bX(\bm\rho))$ denotes the regular subdifferential of $G_s(\bX(\cdot))$. Since $G_s(\bX(\cdot))$ is concave, its regular subdifferential is nonempty if and only if it is differentiable, in which case $\hat\partial_{\bm\rho} G_s(\bX(\bm\rho))=\{\nabla_{\bm\rho} G_s(\bX(\bm\rho))\}$ reduces to a singleton. Then, the above relation implies that this gradient can be computed by any subgradient in $\partial G_s(\bX(\bm\rho))$. For $\bm\rho=\bm0$, note that 
}
\[
\bY^* = -\,\bQ^*\Diag(\bm\beta^*)^{-1}{\bQ^*}^\top \in \partial G_s(\bX(\bm0)).
\]
A direct computation gives
\[
\frac{\partial \bX}{\partial \rho_i}(\bm0)
=
\bV \Diag(\x^* \circ \e_i)\bV^\top
=
x_i^* \bV \e_i \e_i^\top \bV^\top.
\]
Therefore, for all $i\in[d]$, we have
\[
\ip{ \bY^*,\, \frac{\partial \bX}{\partial \rho_i}(\bm0) }
=
x_i^* \ip{ \bY^*,\ \bV \e_i \e_i^\top \bV^\top } 
=
x_i^* (\bV^\top \bY^* \bV)_{ii}
= x_i^* g_i^*,
\]
where the last equality follows from $\g^* = \diag(\bV^\top \bY^* \bV)$.
This then leads to
$\nabla_{\bm\rho} G_s(\bX(\bm\rho))\big|_{\bm\rho=\bm0} = \x^*\circ\g^*$.
As $\objgg(\x^*,\bm\rho) = G_s(\bX(\bm\rho)) + (\x^*)^\top \bm\rho$, we arrive at
\[
\nabla_{\bm\rho} \objgg(\x^*,\bm0) = \x^*\circ\g^* + \x^*.
\]

Recalling from \eqref{eq:opt-gamma-sol-structure} that $\x^* + \x^* \circ \g^*=\bm0$, we arrive at $\nabla_{\bm\rho} \objgg(\x^*,\bm0) = \bm0$. As $\objgg(\x,\bm\rho)$ is concave in $\bm\rho\in\R^d$, this implies that $\bm\rho=\bm0$ maximizes it, i.e., 
\[
\objgg(\x^*,\bm\rho) \le \objgg(\x^*,\bm0), \qquad \forall \bm\rho\in\R^d.
\]

Since $\x^*$ minimizes $\objgg(\x,\bm0)$ over $\cX$, we have $\objgg(\x^*,\bm0) \le \objgg(\x,\bm0)$  for all $\x\in\cX$. 
Thus, 
\[
\objgg(\x^*,\bm\rho) \le \objgg(\x^*,\bm0) \le \objgg(\x,\bm0), \qquad \forall \bm\rho\in\R^d,
\]
i.e., $(\x^*,\bm0)$ is a saddle point of $\objgg(\x^*,\bm\rho)$ over $\cX\times\R^d$. By Sion’s theorem,
\[
\min_{\x\in\cX}\max_{\bm\rho\in\R^d} \objgg(\x,\bm\rho)
=\objgg(\x^*,\bm0)
=\min_{\x\in\cX} \objgg(\x,\bm0).
\]
\end{proof}

\subsection{Proof of \texorpdfstring{\cref{cor:linx-gamma-connection}}{Corollary~\protect\ref{cor:linx-gamma-connection}}}
\label{sec:linx-gamma-connection}

\begin{proof}
\begin{enumerate}[(i)]
\item 
Recall that $\Opt_{\linxo}  = \min_{\x\in\cX} \max_{\rho_0\in\R}  \set{ \objlo(\x,\rho_0)  } $, where $ \objlo(\x,\rho_0) $ is convex in $\x$ for every fixed $\rho_o$ and concave in $\rho_0$ for every fixed $\x\in\cX$.  As this saddle point function is continuous and admits a convex-concave structure,  all domains are closed and convex, and $\cX$ is bounded, there is at least one saddle point $(\x^*,\rho_0^*)$  associated with this problem. 
Then, $\Opt_{\linxo}  =\objlo(\x^*,\rho^*_0)$. As $\Opt_{\Gamma}  = \min_{\x\in\cX}  \objgg(\x,\bm0)$ and $\Opt_{\Gamma^c}  = \min_{\x\in\cX}  \objggc(\x,\bm0) $, we have 
\begin{align*}
\frac12 \left[ \Opt_{\Gamma} +\Opt_{\Gamma^c} \right] 
&\leq \frac12 \left[ \objgg(\x^*,\bm0) + \objggc(\x^*,\bm0) \right] \\
&\le \objld (\x^*, \kappa\ones,\bm0) \\
&= \objlo(\x^*,\kappa),
\end{align*}
where 
the second inequality follows from \cref{prop:connection} for a properly chosen $\kappa\in\R$, 
and the equality holds as $\objlo(\x,\kappa)=\objld(\x, \kappa\ones,\bm0)$. 
Now, since $(\x^*,\rho_0^*)$ is a saddle point associated with $\Opt_{\linxo}$, we have
\[
\objlo(\x^*,\rho_0) \leq \Opt_{\linxo} = \objlo(\x^*,\rho_0^*) \leq \objlo(\x,\rho_0^*),  \quad \forall \x\in\cX, \text{ and } \forall\rho_0 \in\R.
\]
Thus,
\[
\frac12 \left[ \Opt_{\Gamma} +\Opt_{\Gamma^c} \right] 
\le \objlo(\x^*,\kappa) 
\le \objlo(\x^*,\rho_0^*) 
= \Opt_{\linxo}.
\]

\item
Recall {from \eqref{eq:f-Gamma}, \eqref{eq:Opt-Gamma-c}-\eqref{eq:f-Gamma-c}, and \eqref{eq:double-linx-convex-concave}-\eqref{eq:d-scaling:reformulated}} that 
\begin{align*}
\Opt_{\Gamma\textup{-g}} & = \min_{\x\in\cX} \max_{\bm\rho\in\R^d}  \set{ \objgg(\x,\bm\rho)  } , \\
\Opt_{\Gamma^c\textup{-g}} & = \min_{\x\in\cX} \max_{\bm\omega\in\R^d}  \set{ \objggc(\x,\bm\omega)  } , \\
\Opt_{\linxd} & = \min_{\x\in\cX} \max_{\bm\rho,\bm\omega\in\R^d}  \set{ \objld(\x,\bm\rho,\bm\omega)  } ,
\end{align*}
and all these saddle point functions are convex-concave and continuous, all domains are closed and convex, and $\cX$ is bounded. Then, the saddle points associated with these problems must exist. Let $(\x^{\Gamma}, \bm\rho^{\Gamma})$, $(\x^{\Gamma^c}, \bm\omega^{\Gamma^c})$ and $(\x^*,\bm\rho^*,\bm\omega^*)$ be the saddle points associated with $\Opt_{\Gamma\textup{-g}}$, $\Opt_{\Gamma^c\textup{-g}}$ and $\Opt_{\linxd}$, respectively. Then, by definition of the saddle point we have, for all $\x\in\cX$,
\begin{align*}
\objgg(\x^{\Gamma},\bm\rho) \leq &\Opt_{\Gamma\textup{-g}} = \objgg(\x^{\Gamma},\bm\rho^{\Gamma}) \leq \objgg(\x,\bm\rho^{\Gamma}), \quad \forall \bm\rho\\
\objggc(\x^{\Gamma^c},\bm\omega) \leq &\Opt_{\Gamma^c\textup{-g}} = \objggc(\x^{\Gamma^c},\bm\omega^{\Gamma^c}) \leq \objggc(\x,\bm\omega^{\Gamma^c}),  \quad \forall \bm\omega \\
\objld(\x^*,\bm\rho,\bm\omega) \leq &\Opt_{\linxd} = \objld(\x^*,\bm\rho^*,\bm\omega^*) \leq \objld(\x,\bm\rho^*,\bm\omega^*),  \quad \forall \bm\rho,\bm\omega .
\end{align*}
Hence, summing up the first two relations above gives us, for all $\x\in\cX$, 
\begin{align*}
 \frac12\left[ \Opt_{\Gamma\textup{-g}} + \Opt_{\Gamma^c\textup{-g}} \right] 
 \le \frac12\left[ \objgg(\x,\bm\rho^{\Gamma}) + \objggc(\x,\bm\omega^{\Gamma^c}) \right] 
 \le  \objld(\x,\bm\rho^{\Gamma}+\kappa\ones,\bm\omega^{\Gamma^c}),
\end{align*}
where the last inequality follows from \cref{prop:connection} for a properly chosen $\kappa\in\R$.
Now, by plugging in $\x^*$ in this relation, we arrive at
\begin{align*}
 \frac12\left[ \Opt_{\Gamma\textup{-g}} + \Opt_{\Gamma^c\textup{-g}} \right] 
 \le \objld(\x^*,\bm\rho^{\Gamma}+\kappa\ones,\bm\omega^{\Gamma^c}) 
 \le \objld(\x^*,\bm\rho^*,\bm\omega^*) 
 = \Opt_{\linxd} ,
\end{align*}
where the last inequality and equality follow from the saddle point relations for $\Opt_{\linxd}$.
\end{enumerate}
 \end{proof}

\section{Further details on the numerical study}\label{sec:app:numerical}

\subsection{Implementation details}
\label{sec:numerical:implementation}

The experiments are all implemented in Python 3.9.21 on a server with a 3-Hz Intel Xeon Gold 5317 CPU and 124 GB memory. The code for the experiments is available at: \url{https://github.com/joyshen07/mesp-scaling}.

As the $\Gamma$ and $\Gamma^c$ relaxations, i.e., $\Opt_{\Gamma}$ and $\Opt_{\Gamma^c}$,  are standard convex minimization problems, we use the Mirror Descent algorithm \citep{nemirovski_problem_1983} 
for them.  
The required subgradients of these objective functions are computed explicitly  by \cite[Proposition 10]{li_best_2023}, and their boundedness follows from \cite[Lemma 4]{li_best_2023}.

{
The  convex-concave saddle point (SP) problems ($ \Opt_{\Gamma^*}$ and all scaled linx relaxations) are solved using the \emph{parameter-free non-ergodic extragradient} (PF-NE-EG) 
SP algorithm with non-monotone backtracking line search proposed in \cite[Algorithm 2]{shen_parameter-free_2026}. 
Recall that all three scaled linx relaxations results in SP problems of the form $\min_{\x\in\cX}\max_{\bm\rho\in \R^d} f(\x,\bm\rho)$, where $\cX$, i.e., the domain for $\x$ is bounded, but the domain for $\bm\rho$ is unbounded. 
Most classical SP algorithms assume bounded monotone operators (obtained by stacking the subgradient and supergradient) or global Lipschitz continuity, these assumptions fail in our setting, particularly for {g-scaling and} double-scaling, where gradients are neither uniformly bounded nor globally Lipschitz (see \cite[Remark 4]{shen_parameter-free_2026}).
Without the associated boundedness and continuity parameters not only the theoretical convergence guarantees of the classical SP algorithms fail, but also their stepsize selection becomes problematic. 
Moreover, the relatively few existing SP algorithms that can handle unbounded operators typically provide only ergodic (averaged iterate) convergence guarantees, which do not accurately reflect the superior non-ergodic (last-iterate) performance observed empirically. The PF-NE-EG algorithm overcomes these issues by being parameter-free and providing non-ergodic convergence guarantees.
}
{The details for computation of the subgradients associated with the corresponding SP functions and their bounds are given in \cref{sec:subgradient,sec:subgradient:d-scaled-linx}. 
} 

In our implementation of the PF-NE-EG 
algorithm with non-monotone backtracking line search, we set its parameter $\theta=0.9$, initial stepsize $\eta_0=0.1$, backtracking multiplier $\rho=0.9$, stepsize increment factor $\lambda_t=1+\frac1{\log(t+2)}$, and the starting solution to $\bar{\x}=\frac{s}{d}\ones$, and for scaled linx relaxations 
we initialize the scaling parameters with $\bar{\bm\rho}=\bar{\bm\omega}=\bm 0$ (i.e., $\bar{\bm\gamma}=\bar{\bm\mu}=\bar{\bm\Upsilon}=\ones$). 
For each relaxation, we ran this algorithm for a maximum number of $1000$ iterations.

\subsection{Supplementary numerical results}\label{sec:app:numerical:results}

In \cref{tab:scaling-linx-d90,tab:scaling-linx-d124,tab:scaling-5-d90,tab:scaling-5-d124}, the `LB' column shows lower bounds for each relaxation (so that they are indeed valid lower bounds of MESP) computed based on the solution output of the algorithm. 
The `gap' column is computed as the difference between the LB column, which itself is a lower bound provided by the corresponding relaxation, and the integer optimal values of MESP  provided in \cite{anstreicher_efficient_2020} for these problem instances. 
The `conv.' columns show a convergence measurement {for Mirror Descent  algorithm for convex minimization problems and  for the saddle point problems the PF-NE-EG algorithm, which is guaranteed to converge at a rate of $O(1/\sqrt{T})$ (see \cite{shen_parameter-free_2026}).} 
The `time' column reports the running time in seconds. 
The `\# iter.' column records the number of iterations until the algorithm stops, which is capped at $1000$ and can be less in the case of early stopping.

\begin{table}[htbp]
  \centering
  \resizebox{\columnwidth}{!}{
    \begin{tabular}{R|RRRRR|RRRRR|RRRRR}
      \toprule
      & \multicolumn{5}{c|}{linx o-scaling} & \multicolumn{5}{c|}{linx g-scaling} & \multicolumn{5}{c}{linx double-scaling} \\
      s & LB & gap & conv. & time & \# iter. & LB & gap & conv. & time & \# iter. & LB & gap & conv. & time & \# iter. \\
    \midrule  
  20 & -112.621 & 1.139 & 0.0000 &  0.59 &  489 & -112.531 & 1.049 & 0.0022 &  1.37 & 1000 & -112.436 & 0.954 & 0.0089 &  1.62 & 1000 \\
  30 & -162.749 & 1.210 & 0.0000 &  1.02 & 1000 & -162.689 & 1.150 & 0.0000 &  1.35 & 1000 & -162.606 & 1.067 & 0.0060 &  1.62 & 1000 \\
  40 & -211.090 & 1.121 & 0.0000 &  1.04 & 1000 & -211.041 & 1.072 & 0.0000 &  1.35 & 1000 & -210.949 & 0.980 & 0.0040 &  1.61 & 1000 \\
  50 & -258.092 & 0.932 & 0.0000 &  1.23 & 1000 & -258.050 & 0.890 & 0.0000 &  1.06 &  681 & -257.982 & 0.822 & 0.0038 &  1.61 & 1000 \\
  60 & -303.757 & 0.738 & 0.0000 &  1.16 &  912 & -303.716 & 0.697 & 0.0000 &  1.35 & 1000 & -303.644 & 0.625 & 0.0034 &  1.58 & 1000 \\
  70 & -347.928 & 0.457 & 0.0000 &  0.77 &  598 & -347.900 & 0.429 & 0.0000 &  1.22 &  742 & -347.844 & 0.373 & 0.0024 &  1.59 & 1000 \\
  80 & -390.210 & 0.213 & 0.0000 &  0.32 &  234 & -390.189 & 0.192 & 0.0000 &  0.27 &  155 & -390.158 & 0.161 & 0.0014 &  1.63 & 1000 \\
    \bottomrule
  \end{tabular}
  }
  \caption{Performance of scaling methods applied to the linx relaxation on instances with 
  $d=90$ and varying subset size $s$. 
  }
  \label{tab:scaling-linx-d90}
\end{table}

\begin{table}[htbp]
  \centering
  \resizebox{\columnwidth}{!}{
  \begin{tabular}{R|RRRRR|RRRRR|RRRRR}
    \toprule
      & \multicolumn{5}{c|}{linx o-scaling} & \multicolumn{5}{c|}{linx g-scaling} & \multicolumn{5}{c}{linx double-scaling} \\
    s & LB & gap & conv. & time & \# iter. & LB & gap & conv. & time & \# iter. & LB & gap & conv. & time & \# iter. \\
    \midrule
  20 &  -79.305 & 1.478 & 0.0000 &  0.51 &  210 &  -78.927 & 1.100 & 0.0010 &  3.10 & 1000 &  -78.649 & 0.822 & 0.0159 &  3.09 & 1000 \\
  30 & -108.684 & 1.984 & 0.0000 &  0.63 &  261 & -108.270 & 1.570 & 0.0013 &  2.80 & 1000 & -107.927 & 1.227 & 0.0120 &  3.47 & 1000 \\
  40 & -133.466 & 2.411 & 0.0000 &  0.54 &  221 & -133.090 & 2.035 & 0.0011 &  3.07 & 1000 & -132.623 & 1.568 & 0.0040 &  3.19 & 1000 \\
  50 & -152.858 & 3.360 & 0.0000 &  0.83 &  319 & -152.510 & 3.012 & 0.0006 &  2.85 & 1000 & -151.935 & 2.437 & 0.0021 &  3.42 & 1000 \\
  60 & -167.362 & 3.350 & 0.0000 &  0.51 &  205 & -167.120 & 3.108 & 0.0000 &  3.35 &  999 & -166.480 & 2.468 & 0.0025 &  3.36 & 1000 \\
  70 & -175.923 & 3.395 & 0.0000 &  0.51 &  201 & -175.700 & 3.172 & 0.0000 &  3.11 & 1000 & -174.944 & 2.416 & 0.0038 &  3.29 & 1000 \\
  80 & -178.111 & 3.020 & 0.0000 &  0.52 &  200 & -177.948 & 2.857 & 0.0000 &  2.51 &  695 & -177.181 & 2.090 & 0.0028 &  3.34 & 1000 \\
  90 & -174.180 & 2.918 & 0.0000 &  0.71 &  255 & -174.020 & 2.758 & 0.0000 &  3.29 & 1000 & -173.174 & 1.912 & 0.0038 &  3.37 & 1000 \\
 100 & -165.008 & 2.143 & 0.0000 &  0.50 &  187 & -164.919 & 2.054 & 0.0000 &  2.97 & 1000 & -164.194 & 1.329 & 0.0249 &  3.70 & 1000 \\
    \bottomrule
  \end{tabular}
  }
  \caption{Performance of scaling methods applied to the linx relaxation on instances with 
  $d=124$ and varying subset size $s$. 
  }
  \label{tab:scaling-linx-d124}
\end{table}

\begin{table}[htbp]
  \centering
  \rotatebox{90}{\begin{minipage}{\textheight}
  \resizebox{\textheight}{!}{ 
  \begin{tabular}{R|RRRRR|RRRRR|RRRRR|RRRRR}
    \toprule 
    & \multicolumn{5}{c|}{linx double-scaling} & \multicolumn{5}{c|}{$\Gamma$} & \multicolumn{5}{c|}{$\Gamma^c$} & \multicolumn{5}{c}{$\Gamma^*$} \\
    s & LB & gap & conv. & time & \# iter. & LB & gap & conv. & time & \# iter. & LB & gap & conv. & time & \# iter. & LB & gap & conv. & time & \# iter. \\
    \midrule
  20 & -112.436 & 0.954 & 0.0089 &  1.62 & 1000 & -112.121 & 0.639 & 0.0029 &  1.91 & 1000 & -115.366 & 3.884 & 0.0054 &  1.78 & 1000 & -112.118 & 0.636 & 0.0000 &  1.51 &  225 \\
  30 & -162.606 & 1.067 & 0.0060 &  1.62 & 1000 & -162.387 & 0.848 & 0.0067 &  2.04 & 1000 & -165.080 & 3.541 & 0.0128 &  1.44 & 1000 & -162.380 & 0.841 & 0.0000 &  2.71 &  352 \\
  40 & -210.949 & 0.980 & 0.0040 &  1.61 & 1000 & -210.926 & 0.957 & 0.0086 &  2.15 & 1000 & -213.072 & 3.103 & 0.0090 &  1.37 & 1000 & -210.917 & 0.948 & 0.0000 &  4.47 &  598 \\
  50 & -257.982 & 0.822 & 0.0038 &  1.61 & 1000 & -258.104 & 0.944 & 0.0044 &  2.22 & 1000 & -259.675 & 2.515 & 0.0006 &  1.61 & 1000 & -258.100 & 0.940 & 0.0000 &  2.21 &  312 \\
  60 & -303.644 & 0.625 & 0.0034 &  1.58 & 1000 & -303.897 & 0.878 & 0.0021 &  2.31 & 1000 & -304.802 & 1.783 & 0.0079 &  1.55 & 1000 & -303.895 & 0.876 & 0.0000 &  2.62 &  367 \\
  70 & -347.844 & 0.373 & 0.0024 &  1.59 & 1000 & -348.175 & 0.704 & 0.0014 &  2.30 & 1000 & -348.595 & 1.124 & 0.0041 &  1.49 & 1000 & -348.119 & 0.648 & 0.0000 &  7.22 & 1000 \\
  80 & -390.158 & 0.161 & 0.0014 &  1.63 & 1000 & -390.357 & 0.360 & 0.0022 &  2.42 & 1000 & -390.483 & 0.486 & 0.0026 &  1.15 & 1000 & -390.313 & 0.316 & 0.0001 &  7.24 & 1000 \\
    \bottomrule
  \end{tabular}
  }
  \caption{Performance of the linx and $\Gamma$ relaxation bounds on instances with $d=90$ and varying subset size $s$.}
  \label{tab:scaling-5-d90}
\vspace{2em}
  \resizebox{\textheight}{!}{ 
  \begin{tabular}{R|RRRRR|RRRRR|RRRRR|RRRRR}
    \toprule 
    & \multicolumn{5}{c|}{linx double-scaling} & \multicolumn{5}{c|}{$\Gamma$} & \multicolumn{5}{c|}{$\Gamma^c$} & \multicolumn{5}{c}{$\Gamma^*$} \\
    s & LB & gap & conv. & time & \# iter. & LB & gap & conv. & time & \# iter. & LB & gap & conv. & time & \# iter. & LB & gap & conv. & time & \# iter. \\
    \midrule
  20 &  -78.649 & 0.822 & 0.0159 &  3.09 & 1000 &  -78.336 & 0.509 & 0.0023 &  3.64 & 1000 &  -81.966 & 4.139 & 0.0003 &  2.52 & 1000 &  -78.334 & 0.507 & 0.0000 &  3.66 &  307 \\
  30 & -107.927 & 1.227 & 0.0120 &  3.47 & 1000 & -107.982 & 1.282 & 0.0005 &  3.72 & 1000 & -111.321 & 4.621 & 0.0001 &  2.19 & 1000 & -107.982 & 1.282 & 0.0000 &  1.72 &  148 \\
  40 & -132.623 & 1.568 & 0.0040 &  3.19 & 1000 & -133.297 & 2.242 & 0.0001 &  3.93 & 1000 & -135.495 & 4.440 & 0.0000 &  2.41 & 1000 & -133.296 & 2.241 & 0.0000 &  1.47 &  122 \\
  50 & -151.935 & 2.437 & 0.0021 &  3.42 & 1000 & -153.351 & 3.853 & 0.0002 &  4.24 & 1000 & -154.381 & 4.883 & 0.0000 &  2.09 & 1000 & -153.309 & 3.811 & 0.0000 & 12.29 &  988 \\
  60 & -166.480 & 2.468 & 0.0025 &  3.36 & 1000 & -168.917 & 4.905 & 0.0001 &  4.10 & 1000 & -168.151 & 4.139 & 0.0001 &  2.29 & 1000 & -168.048 & 4.036 & 0.0000 &  7.19 &  587 \\
  70 & -174.944 & 2.416 & 0.0038 &  3.29 & 1000 & -178.014 & 5.486 & 0.0001 &  4.24 & 1000 & -176.343 & 3.815 & 0.0000 &  2.26 & 1000 & -176.341 & 3.813 & 0.0000 & 10.22 &  783 \\
  80 & -177.181 & 2.090 & 0.0028 &  3.34 & 1000 & -180.611 & 5.520 & 0.0000 &  4.46 & 1000 & -177.894 & 2.803 & 0.0007 &  1.90 & 1000 & -177.893 & 2.802 & 0.0000 &  3.95 &  322 \\
  90 & -173.174 & 1.912 & 0.0038 &  3.37 & 1000 & -177.041 & 5.779 & 0.0001 &  4.66 & 1000 & -173.871 & 2.609 & 0.0000 &  1.82 & 1000 & -173.871 & 2.609 & 0.0000 &  1.75 &  141 \\
 100 & -164.194 & 1.329 & 0.0249 &  3.70 & 1000 & -167.743 & 4.878 & 0.0000 &  4.30 & 1000 & -164.478 & 1.613 & 0.0010 &  1.78 & 1000 & -164.477 & 1.612 & 0.0000 &  3.64 &  276 \\
    \bottomrule
  \end{tabular}
  }
  \caption{Performance of the linx and $\Gamma$ relaxation bounds on instances with $d=124$ and varying subset size $s$. }
  \label{tab:scaling-5-d124}
  \end{minipage}}
\end{table}

%% file: appendix_subgradient.tex
\section{Subgradient analysis for \texorpdfstring{$\Gamma$}{Gamma} relaxations}
\label{sec:subgradient}

{
We will use first-order methods to solve the $\Gamma$ relaxation and its variants. 
Recall that the equivalent definitions of $\Opt_{\Gamma}$ are given in \eqref{eq:opt-gamma:equivalence}, where $g_s$ and $h_s$ are defined in \eqref{eq:h-rep:non-sparse} and \eqref{def:h}, respectively. 
To analyze the subgradients involved with $\Gamma$ relaxations, 

in \cref{prop:subdifferential}, we show the exact characterization of the subdifferential set of $g_s=h_s$. 
\begin{proposition}\label{prop:subdifferential}
  Consider any $\bar{\bm\lambda}\in\R^d_+\cap\K_{o,d}$, where $r:=\|\bar{\bm\lambda}\|_0\geq s$. Then, the subdifferential of $g_s=h_s$ at $\bar{\bm\lambda}$ is given by 
  \begin{align}\label{eq:h-subdifferential}
    \partial g_s(\bar{\bm\lambda}) = \set{ \bm\mu\in\R^d:~ 
    \mu_i 
    \begin{cases} 
      = \varphi'(\bar{\lambda}_i), & \textup{if }i \in[k], \\
      = \varphi'(\frac{1}{s-k}\sum_{i\in[k+1,d]}\bar{\lambda}_i), & \textup{if }i \in (k,r], \\
      \leq \varphi'(\frac{1}{s-k}\sum_{i\in[k+1,d]}\bar{\lambda}_i), & \textup{if }i \in (r,d]. 
    \end{cases}
    }. 
  \end{align}
\end{proposition}

\begin{proof}
{Given a convex function $\varphi$, we let $\varphi^*$ represent its Fenchel conjugate.}
  We first compute the Fenchel conjugate of $g_s$: 
  \begin{align*}
    g_s^*(\bm\mu) 
    &= \max_{\bm\lambda\in\R^d_+}\set{\langle\bm\lambda,\bm\mu\rangle - g_s(\bm\lambda)} \\
    &= \max_{\bm\lambda\in\R^d_+,\bm\xi\in\R_+^d}\set{\langle\bm\lambda,\bm\mu\rangle - f_s(\bm\xi):\bm\xi\majorize\bm\lambda,\ \bm\xi\in\K_{o,d},\ \|\bm\xi\|_0\leq s} \\
    &= \max_{\bm\xi\in\R_+^d}\set{\langle\bP_{\bm\mu}\bm\xi,\bm\mu\rangle - f_s(\bm\xi): \bm\xi\in\K_{o,d},\ \|\bm\xi\|_0\leq s} \\
    &= \max_{\bm\xi\in\R_+^d}\set{\sum_{i\in[s]}(\mu_{[i]}\xi_i - \varphi(\xi_i)): \bm\xi\in\K_{o,d},\ \|\bm\xi\|_0\leq s} \\
    &= \sum_{i\in[s]}\varphi^*(\mu_{[i]}). 
  \end{align*}
  Here, the second step follows from the definition of $g_s$. 
  The third step follows from the property of majorization that $\bm\xi\majorize\bm\lambda$ if and only $\bm\lambda\in\conv\{\bP\bm\xi:\bP\textup{ is a permutation matrix}\}$ by {\cref{thm:char:majorization}}, and $\bP_{\bm\mu}$ denotes the permutation matrix such that $\bP_{\bm\mu}\bm\xi$ is sorted in the same order as $\bm\mu$ and therefore their inner product achieves the maximum according to the rearrangement inequality. The last step follows from the observation that the maximization problem without the ordering constraint $\max_{\bm\xi\in\R_+^d}\set{\sum_{i\in[s]}(\mu_{[i]}\xi_i - \varphi(\xi_i))}=\sum_{i\in[s]}\varphi^*(\mu_{[i]})$ is attained by $\xi^*_i\in\partial\varphi^*(\mu_{[i]})$ for all $i$, and moreover as $\varphi^*$ is convex we observe that $\xi^*_i$ is ordered in nonincreasing fashion. 
  
  Next, we compute the Fenchel conjugate of $g_s^*$ as follows. {Recall that $g_s$ is a proper, closed, and convex function, and so $g_s=g_s^{**}$.} Thus, for any $\bm\lambda\in\R^d_+\cap\K_{o,d}$, {where $r\coloneqq\|\bm\lambda\|_0\geq s$, using the preceding expression for $g_s^*(\bm\mu)$,} we have 
  \begin{align*}
    g_s(\bm\lambda)
    = g_s^{**}(\bm\lambda) 
    &= \max_{\bm\mu\in\R^d}\set{\ip{\bm\mu,\,\bm\lambda} - g_s^*(\bm\mu)} \\
    &= \max_{\bm\mu\in\R^d}\set{\sum_{i\in[r]}\mu_i\lambda_i - \sum_{i\in[s]}\varphi^*(\mu_{[i]})} \\
    &= \max_{\bm\mu\in\R^d}\set{\sum_{i\in[r]}\mu_{[i]}\lambda_i - \sum_{i\in[s]}\varphi^*(\mu_{[i]})} \\
    &= \max_{\bm\mu\in\R^d}\set{\sum_{i\in[r]}\mu_i\lambda_i - \sum_{i\in[s]}\varphi^*(\mu_i):
    \, \begin{array}{l} \bm\mu_{1:r}\in\K_{o,r},\\   \mu_i\leq\mu_r,\forall i\in(r,d] \end{array} } \\
    &\leq \max_{\bm\mu\in\R^d,\bm\zeta\in\R^d_+\cap\K_{o,d}}\set{\sum_{i\in[r]}\mu_i\zeta_i - \sum_{i\in[s]}\varphi^*(\mu_i): \, \begin{array}{l}  \bm\mu_{1:r}\in\K_{o,r},\\  \mu_i\leq\mu_r,\forall i\in(r,d],\\  \bm\zeta\majorize\bm\lambda,\ \|\bm\zeta\|_0\leq s \end{array} } \\
    &= \max_{\bm\mu\in\R^d,\bm\zeta\in\R^d_+\cap\K_{o,d}}\set{\sum_{i\in[s]}[\mu_i\zeta_i - \varphi^*(\mu_i)]: \, \begin{array}{l}  \bm\mu_{1:r}\in\K_{o,r},\\  \mu_i\leq\mu_r,\forall i\in(r,d],\\  \bm\zeta\majorize\bm\lambda,\ \|\bm\zeta\|_0\leq s \end{array} } \\
    &\leq \max_{\bm\zeta\in\R^d_+\cap\K_{o,d}}\set{\sum_{i\in[s]}\varphi(\zeta_i):\bm\zeta\majorize\bm\lambda,\ \|\bm\zeta\|_0\leq s} \\
    &= g_s(\bm\lambda). 
  \end{align*}
  The fourth equation follows from the rearrangement inequality, in the fifth equation we made a change of variable, {and the sixth equation follows from noting $\|\bm\zeta\|_0\le s \le r$}. The first inequality follows from the fact that $\bm\mu_{1:r}\in\K_{o,r}$ and the definition of majorization, because  
  \begin{align*}
    \sum_{i\in[r]}\mu_i\zeta_i
    &= \mu_r\sum_{j\in[r]}\zeta_i + \sum_{i\in[r-1]}(\mu_i-\mu_{i+1})\sum_{j\in[i]}\zeta_j \\
    &\geq \mu_r\sum_{j\in[r]}\lambda_i + \sum_{i\in[r-1]}(\mu_i-\mu_{i+1})\sum_{j\in[i]}\lambda_j 
    = \sum_{i\in[r]}\mu_i\lambda_i, 
  \end{align*}
  and it is satisfied at equality if and only if either $\mu_i=\mu_{i+1}$ or $\sum_{j\in[i]}\lambda_j=\sum_{j\in[i]}\zeta_j$ for each $i\in[r-1]$. 
  The second inequality {follows from the definition of the Fenchel conjugate and it} is satisfied at equality if and only if $\mu_i\in\partial\varphi(\zeta_i)$ for all $i\in[s]$. 
  {By \cref{prop:h-rep},} the last equality is achieved if and only if $\bm\zeta$ is given by
  \begin{align*}
    \zeta_i = \begin{cases}
      \lambda_i, & \textup{if } i\in[k], \\
      \frac{1}{s-k}\sum_{i\in(k,d]}\lambda_i, & \textup{if } i\in(k,s], \\
      0, & \textup{if } i\in(s,d]. 
    \end{cases}
  \end{align*}
  Thus, the first inequality is satisfied at equality if and only if $\mu_i=\mu_s$ for $i\in(s,r]$, because $\sum_{j\in[i]}\lambda_j<\sum_{j\in[s]}\zeta_j=\sum_{j\in[i]}\zeta_j$ for $i\in[s,r)$, which implies $\mu_i=\mu_{i+1}$ for $i\in[s,r)$. By \cite[Theorem 1.4.1]{hiriart-urruty_fundamentals_2001}, the subdifferential of $g_s=h_s$ is the set of maximizers of $\max_{\bm\mu\in\R^d}\set{\ip{\bm\mu,\, \bm\lambda} - g_s^*(\bm\mu)}$, which is exactly \eqref{eq:h-subdifferential} by our analysis.
\end{proof}

To obtain the subdifferential of $g_s\circ\bm\lambda=h_s\circ\bm\lambda$, we introduce lemmas on spectral functions. 

We are now ready to give the subdifferential of the function $G_s(\bX(\x))=g_s(\bm\lambda(\bX(\x)))$. 

\begin{lemma}\label{lem:subdifferential}
  The subdifferential of the function $G_s(\bX(\x))=g_s(\bm\lambda(\bX(\x)))$ at the point $\x$ is given by
  \begin{align*}
    \set{ \bX^*(\bQ\Diag(\y)\bQ^\top): 
    \begin{array}{>{\displaystyle}l} 
      y_i 
    \begin{cases} 
      = \varphi'(\lambda_i), & \textup{if }i \in[k], \\
      = \varphi'(\frac{1}{s-k}\sum_{i\in[k+1,d]}\lambda_i), & \textup{if }i \in (k,r], \\
      \leq \varphi'(\frac{1}{s-k}\sum_{i\in[k+1,d]}\lambda_i), & \textup{if }i \in (r,d]. 
    \end{cases} \\
      \bQ \textup{ is orthogonal},\ \bQ\Diag(\bm\lambda)\bQ^\top = \bX(\x)
    \end{array}
    }, 
  \end{align*}
  where $\bX^*(\cdot)$ is the adjoint operator of the linear mapping $\bX(\cdot)$,
  {$k$ is the special index for $\bm\lambda(\bX(\x))$ from \cref{lem:index-k}, and $r\coloneqq \|\bm\lambda(\bX(\x))\|_0$.} 
\end{lemma}

\begin{proof}
  By \cref{lem:conjugate-subdifferential}, we have
  \begin{align*}
    \partial G_s(\bX)
    &= \set{\bQ\Diag(\y)\bQ^\top:\y\in\partial g_s(\bm\lambda),\ \bQ \textup{ is orthogonal},\ \bQ\Diag(\bm\lambda)\bQ^\top = \bX}. 
  \end{align*}
 Then the conclusion follows from applying \cref{prop:subdifferential} for $\partial g_s(\bm\lambda)$ and \cite[Theorem 4.2.1]{hiriart-urruty_fundamentals_2001}. 
\end{proof}

Next, in \cref{lem:smooth-constant}, we provide an upper bound for the subgradient for $\bX(\x)=\bM\Diag(\x)\bM^\top$ for a generic matrix $\bM\in\S^d$. We will use \cref{lem:smooth-constant} in the discussion of specific relaxations of interest by substituting $\bM=\bV,\ \bW$ or $\bC$ later on. 

\begin{lemma}\label{lem:smooth-constant}
  Suppose $\bX(\x)=\bM\Diag(\x)\bM^\top$ for some full-rank matrix $\bM\in\R^{d\times d}$. Then, {all of} the subgradients of $G_s(\bX(\x))=g_s(\bm\lambda(\bX(\x)))$ at $\x$ is bounded above by $\cG\coloneqq \sqrt{d}\|\bM^\top\bM\|_2|\varphi'(\delta_s(\bM^\top\bM))|$, where $\delta_s(\bM^\top\bM)$ is defined in \eqref{eq:delta-constant}. 
\end{lemma}

\begin{proof}
  {Given $\bX(\x)$, let $\bQ$ be any orthogonal matrix such that $\bQ\Diag(\bm\lambda)\bQ^\top=\bX(\x)$, and let $\|\cdot\|_F$ denote the Frobenius norm.}  
  Then, by noting that $\bX^*(\bY)=\diag(\bM^\top\bY\bM)$ and using \cref{lem:subdifferential}, {any} subgradient of $G_s(\bX(\x))=g_s(\bm\lambda(\bX(\x)))$ at $\x$ is bounded by 
  \begin{align*}
    \|\bX^*(\bQ\Diag(\y)\bQ^\top)\|_2 
    &= \|\diag(\bM^\top\bQ\Diag(\y)\bQ^\top\bM)\|_2 \\
    &\leq \|\bm\lambda(\bM^\top\bQ\Diag(\y)\bQ^\top\bM)\|_2 \\
    &= \|\bM^\top\bQ\Diag(\y)\bQ^\top\bM\|_F \\
    &= \|\bQ^\top\bM\bM^\top\bQ\Diag(\y)\|_F \\
    &\leq \|\bQ^\top\bM\bM^\top\bQ\|_2\|\Diag(\y)\|_F \\
    &\leq \|\bM^\top\bM\|_2\|\y\|_2, 
  \end{align*}
  where 
  \begin{align}\label{eq:y_i:subdifferential}
    y_i = 
    \begin{cases} 
      \varphi'(\lambda_i(\bX(\x))), & \textup{if }i \in [k], \\
      \varphi'(\frac{1}{s-k}\sum_{i\in[k+1,d]}\lambda_i(\bX(\x))), & \textup{if }i\in [k+1,d], \\
    \end{cases}
  \end{align}
  {where $k$ is the special index from \cref{lem:index-k}.} 
  {Then, as $\bm\lambda(\bX(\x))\in\K_{o,d}$, from the definition of $k$ we have $k\le s-1$ and  so $\frac{1}{s-k}\sum_{i\in[k+1,d]}\lambda_i(\bX(\x))\geq\lambda_{k+1}(\bX(\x))\geq\lambda_s(\bX(\x))$ and thus}
  \begin{align}\label{eq:delta-constant}
    \min_{\x\in\cX}\set{\frac{1}{s-k}\sum_{i\in[k+1,d]}\lambda_i(\bX(\x))} \geq \min_{\x\in\cX\cap\{0,1\}^d}\lambda_s(\bX(\x)) \eqcolon\delta_s(\bM^\top\bM) > 0,
  \end{align}
  where in the first inequality we also used the concavity of the function $\sum_{i\in[k+1,d]}\lambda_i(\bX(\x))$ when we restrict the domain to binaries, and the last strict inequality follows from $\rank(\bX(\x))=s$. 
  By definition of $k$ in \cref{lem:index-k}, for any $i\in[k]$ we have $\lambda_i(\bX(\x))\geq\frac{1}{s-k}\sum_{j\in[k+1,d]}\lambda_j(\bX(\x)) \geq\delta_s(\bM^\top\bM)$. Then, since {$\varphi(\cdot)$ is convex and so} $\varphi'(\cdot)$ is nondecreasing and $0\ge \varphi'(\cdot)$, we conclude $0\geq y_i\geq\varphi'(\delta_s(\bM^\top\bM))$ for any $i\in[d]$. 
  Thus, $\|\y\|_2\leq\sqrt{d}|\varphi'(\delta_s(\bM^\top\bM))|$. Combined with the inequality $\|\bX^*(\bQ\Diag(\y)\bQ^\top)\|_2 \leq \|\bM^\top\bM\|_2\|\y\|_2$, this concludes the proof. 
\end{proof}

\begin{remark}\label{rem:smooth-constant}
  The constant $\delta_s(\bM^\top\bM)$ in \cref{lem:smooth-constant} can be difficult to evaluate in practice. Nevertheless, by Cauchy's interlacing theorem, for any $\x\in\cX\cap\{0,1\}^d$, the matrix $\bX(\x)=\bM\Diag(\x)\bM^\top$ for some full-rank matrix $\bM\in\R^{d\times d}$ satisfies $\lambda_s(\bX(\x)) \geq \lambda_d(\bM^\top\bM) > 0$. 
  Therefore, $\delta_s(\bM^\top\bM)\geq\lambda_d(\bM^\top\bM)$ gives a tractable lower bound for $\delta_s(\bM^\top\bM)$, which can in return be used to provide an upper bound on $\cG$ in \cref{lem:smooth-constant}. 
\end{remark}

Note that the objective function of the majorization-based $\Gamma$ relaxation \eqref{eq:Gamma_rel} (see also \eqref{eq:opt-gamma:equivalence}) is precisely 
$g_s(\bm\lambda(\bX(\x)))=\objgg(\x,\bm 0)$, where $\bX(\x) =\bV\Diag(\x)\bV^\top$. Then, applying \cref{lem:smooth-constant} with $\bM=\bV$ {and recalling that $\bC=\bV^\top\bV$, we conclude that the subgradients of this function is bounded above by} the constant $\cG_{\Gamma}\leq\sqrt{d}\|\bC\|_2|\varphi'(\delta_s(\bC))|$. 

Similarly, the objective function of the complementary $\Gamma^c$ relaxation \eqref{eq:opt-gamma-complementary}  
up to a constant difference of $-\log\det(\bC)$ is  $g_{d-s}(\bm\lambda(\bX(\check\x))) 
$, where $\bX(\check\x) =\bW\Diag(\check\x)\bW^\top$ and $\bW^\top\bW = \bC^{-1}$. Then, applying \cref{lem:smooth-constant} with $\bM=\bW$, we conclude that its subgradients are bounded above by the constant $\cG_{\Gamma^c} \leq\sqrt{d}\|\bC^{-1}\|_2|\varphi'(\delta_{d-s}(\bC^{-1}))|$. 

Finally recall that the objective function of the linx relaxation \eqref{eq:linx} 
up to a constant factor of $\frac12$ is given by $g_d(\bm\lambda(\bX(\x)))= 
\objld(\x,\bm0,\bm0)$, where $\bX(\x) = \bC\Diag(\x)\bC + \bI_d - \Diag(\x)$. Then, the norm of its subgradients can be bounded according to the following lemma. 

\begin{lemma}
  Suppose $\bX(\x) = \bC\Diag(\x)\bC + \bI_d - \Diag(\x)$. Then, any subgradient of $g_d(\bm\lambda(\bX(\x)))=
  \objld(\x,\bm0,\bm0)$ {at any point $\x\in\cX=\{\x\in[0,1]^d:\, \ones^\top\x = s\}$} is bounded above by $\cG_{\textup{linx}}=\sqrt{d}(\|\bC\|_2^2 + 1)|\varphi'(\delta_s(\bC^2))|$. 
\end{lemma}

\begin{proof}
  {Given $\bX(\x)$, let $\bQ$ be any orthogonal matrix such that $\bQ\Diag(\bm\lambda)\bQ^\top=\bX(\x)$.}  
Note that $\bX^*(\bY)=\diag(\bC\bY\bC-\bY)$. Thus, {by defining $\y$ as in  \eqref{eq:y_i:subdifferential}, defining $\bY\coloneqq \bQ\Diag(\y)\bQ^\top$, and} applying \cref{lem:subdifferential}, we have 
  \begin{align*}
    \|\bX^*(\bQ\Diag(\y)\bQ^\top)\|_2 
    = \|\diag(\bC\bY\bC - \bY)\|_2 
    \leq \|\diag(\bC\bY\bC)\|_2 + \|\diag(\bY)\|_2 
    \leq (\|\bC^2\|_2 + \|\bI_d\|_2)\|\y\|_2, 
  \end{align*}
  where the last step follows from the proof of \cref{lem:smooth-constant} with $\bM=\bC$ 
  and $\bM=\bI_d$, respectively. To bound $\|\y\|_2$, recall $|y_i|=|\varphi'(\lambda_i(\bX(\x)))|\leq|\varphi'(\lambda_d(\bX(\x)))|$ for all $i\in[d]$. Recall that $\lambda_d(\bX(\x))$ is a concave function of $\x$, and so its minimum must be attained at extreme points of $\cX$. {Since $\cX=\{\x\in[0,1]^d:\, \ones^\top\x = s\}$ has binary extreme points, we have} 
  \begin{align*}
    \min_{\x\in\cX}\lambda_d(\bX(\x)) 
    &= \min_{\x\in\cX\cap\{0,1\}^d}\lambda_d(\bX(\x)). 
  \end{align*}
  For any $\x\in\set{0,1}^d$, if $\ones^\top\x=s$, then $\rank(\bI_d-\Diag(\x))=d-s$. 
  By Weyl's inequality (see \citep[Corollary 4.3.5]{horn_matrix_1985})
  \begin{align*}
    \lambda_d(\bX(\x)) = \lambda_d(\bC\Diag(\x)\bC + \bI_d - \Diag(\x)) \geq \lambda_s(\bC\Diag(\x)\bC) \geq \delta_s(\bC^2), 
  \end{align*}
  where the last step follows from the definition of $\delta_s(\cdot)$ in \eqref{eq:delta-constant}. Therefore, $\lambda_i(\bX(\x))\geq\delta_s(\bC^2)$ for any $i\in[d]$. Following the proof of \cref{lem:smooth-constant} with $\bM=\bC$, this implies $\|\y\|_2\leq\sqrt{d}|\varphi'(\delta_s(\bC^2))|$. 
\end{proof}

\section{Subgradient analysis for the double-scaled linx relaxation}\label{sec:subgradient:d-scaled-linx}
We will use first-order methods to solve the double-scaled linx relaxation, i.e., \eqref{eq:d-scaling:reformulated}, and the other scaling variants of linx relaxation. Thus, we examine the gradients of $\objld(\x,\bm\rho,\bm\omega)$. 
\begin{lemma}\label{lem:linx-double-scaling-grad}
  Define $\bL:=\bL(\x,\bm\rho,\bm\omega)\coloneqq\bC\Diag(\exp(\bm\rho))\Diag(\x)\bC+\Diag(\exp(\bm\omega))(\bI_d-\Diag(\x))$.
  Then, for any $\x\in[0,1]^d$ and any $\bm\rho,\bm\omega$, the function $\objld(\x,\bm\rho,\bm\omega)$ satisfies 
  \begin{align*}
    \nabla_\x\objld(\x,\bm\rho,\bm\omega) &= -\frac12\exp(\bm\rho)\circ\diag(\bC\bL^{-1}\bC) + \frac12\exp(\bm\omega)\circ\diag(\bL^{-1}) + \frac12\bm\rho - \frac12\bm\omega, \\
    \nabla_{\bm\rho}\objld(\x,\bm\rho,\bm\omega) &= -\frac12\exp(\bm\rho)\circ\x\circ\diag(\bC\bL^{-1}{\bC}) + \frac12\x, \\
    \nabla_{\bm\omega}\objld(\x,\bm\rho,\bm\omega) &= -\frac12\exp(\bm\omega)\circ(\ones-\x)\circ\diag(\bL^{-1}) + \frac12(\ones-\x) .
  \end{align*}
\end{lemma}

\begin{proof}
  Recall that $\bE_{ii}=\Diag(\e_i)$. The partial derivative of the $\log\det$ term in $\objld(\x,\bm\rho,\bm\omega)$ w.r.t.\ $x_i$ is 
  \begin{align*}
    -\frac12\langle\bL^{-1}, \frac{\partial\bL}{\partial x_i}\rangle 
    &= -\frac12\langle\bL^{-1}, \bC\Diag(\exp(\bm\rho))\bE_{ii}\bC - \Diag(\exp(\bm\omega))\bE_{ii}\rangle \\
    &= -\frac12\tr(\bC\bL^{-1}\bC\Diag(\exp(\bm\rho))\bE_{ii}) + \frac12\tr(\bL^{-1}\Diag(\exp(\bm\omega))\bE_{ii}) \\
    &= -\frac12\exp(\rho_i)(\bC\bL^{-1}\bC)_{ii} + \frac12\exp(\omega_i)(\bL^{-1})_{ii}. 
  \end{align*}
  Thus, the gradient of $\objld(\x,\bm\rho,\bm\omega)$ w.r.t. $\x$ is 
  \begin{align*}
    \nabla_\x\objld(\x,\bm\rho,\bm\omega)
    = -\frac12\exp(\bm\rho)\circ\diag(\bC\bL^{-1}\bC) + \frac12\exp(\bm\omega)\circ\diag(\bL^{-1}) + \frac12\bm\rho - \frac12\bm\omega. 
  \end{align*}
  The partial derivative w.r.t.\ $\rho_i$ is 
  \begin{align*}
    \frac{\partial\objld}{\partial\rho_i}
    = -\frac12\langle\bL^{-1}, \frac{\partial\bL}{\partial\rho_i}\rangle + \frac12x_i
    &= -\frac12x_i\exp(\rho_i)(\bC\bL^{-1}\bC)_{ii} + \frac12x_i. 
  \end{align*}
  Thus, the gradient w.r.t.\ $\bm\rho$ is 
  \begin{align*}
    \nabla_{\bm\rho}\objld(\x,\bm\rho,\bm\omega) = -\frac12\exp(\bm\rho)\circ\x\circ\diag(\bC\bL^{-1}{\bC}) + \frac12\x. 
  \end{align*}
  The partial derivative w.r.t.\ $\omega_i$ is 
  \begin{align*}
    \frac{\partial\objld}{\partial\omega_i}
    &= -\frac12\langle\bL^{-1}, \frac{\partial\bL}{\partial\omega_i}\rangle + \frac12(1-x_i)
    = -\frac12(1-x_i)\exp(\omega_i)(\bL^{-1})_{ii} + \frac12(1-x_i). 
  \end{align*}
  Thus, the gradient w.r.t.\ $\bm\omega$ is 
  \begin{align*}
    \nabla_{\bm\omega}\objld(\x,\bm\rho,\bm\omega) = -\frac12\exp(\bm\omega)\circ(\ones-\x)\circ\diag(\bL^{-1}) + \frac12(\ones-\x). 
  \end{align*}
\end{proof}

%% file: ref.bib
@article{anstreicher_maximum-entropy_2018,
  author  = {Anstreicher, Kurt M.},
  title   = {Maximum-entropy sampling and the {Boolean} quadric polytope},
  journal = {Journal of Global Optimization},
  year    = {2018},
  volume  = {72},
  number  = {4},
  pages   = {603--618}
}

@article{anstreicher_efficient_2020,
  author  = {Anstreicher, Kurt M.},
  title   = {Efficient Solution of Maximum-Entropy Sampling Problems},
  journal = {Operations Research},
  volume  = {68},
  number  = {6},
  pages   = {1826--1835},
  year    = {2020}
}

@inproceedings{anstreicher_continuous_1996,
  author    = {Anstreicher, Kurt M. and Fampa, Marcia and Lee, Jon and Williams, Joy},
  editor    = {Cunningham, William H. and McCormick, S. Thomas and Queyranne, Maurice},
  title     = {Continuous relaxations for Constrained Maximum-Entropy Sampling},
  booktitle = {Integer Programming and Combinatorial Optimization},
  year      = {1996},
  publisher = {Springer Berlin Heidelberg},
  address   = {Berlin, Heidelberg},
  pages     = {234--248}
}

@article{anstreicher_using_1999,
  title   = {Using continuous nonlinear relaxations to solve constrained maximum-entropy sampling problems},
  volume  = {85},
  number  = {2},
  journal = {Mathematical Programming},
  author  = {Anstreicher, Kurt M. and Fampa, Marcia and Lee, Jon and Williams, Joy},
  year    = {1999},
  pages   = {221--240}
}

@inproceedings{anstreicher_masked_2004,
  address   = {Heidelberg},
  title     = {A Masked Spectral Bound for Maximum-Entropy Sampling},
  booktitle = {mODa 7 --- Advances in Model-Oriented Design and Analysis},
  publisher = {Physica-Verlag HD},
  author    = {Anstreicher, Kurt M. and Lee, Jon},
  editor    = {Di Bucchianico, Alessandro and Läuter, Henning and Wynn, Henry P.},
  year      = {2004},
  pages     = {1--12}
}

@article{burer_solving_2007,
  title   = {Solving maximum-entropy sampling problems using factored masks},
  journal = {Mathematical Programming},
  volume  = {109},
  number  = {2-3},
  pages   = {263--281},
  year    = {2007},
  author  = {Burer, Samuel and Lee, Jon}
}

@article{chen_mixing_2021,
  title   = {Mixing convex-optimization bounds for maximum-entropy sampling},
  journal = {Mathematical Programming},
  author  = {Chen, Zhongzhu and Fampa, Marcia and Lambert, Amélie and Lee, Jon},
  year    = {2021},
  volume  = {188},
  number  = {2},
  pages   = {539--568}
}

@article{chen_computing_2023,
  author  = {Chen, Zhongzhu and Fampa, Marcia and Lee, Jon},
  title   = {On Computing with Some Convex Relaxations for the Maximum-Entropy Sampling Problem},
  journal = {INFORMS Journal on Computing},
  volume  = {35},
  number  = {2},
  pages   = {368--385},
  year    = {2023}
}

@article{chen_generalized_2024,
  title   = {Generalized scaling for the constrained maximum-entropy sampling problem},
  journal = {Mathematical Programming},
  author  = {Chen, Zhongzhu and Fampa, Marcia and Lee, Jon},
  year    = {2024},
  volume  = {212},
  number  = {1},
  pages   = {177--216}
}

@book{fampa_maximum-entropy_2022,
  series     = {Springer Series in Operations Research and Financial Engineering},
  title      = {Maximum-Entropy Sampling: Algorithms and Application},
  shorttitle = {Maximum-Entropy Sampling},
  author     = {Fampa, Marcia and Lee, Jon},
  year       = {2022}, 
	publisher  = {Spring Nature}
}

@book{hiriart-urruty_fundamentals_2001,
  address   = {Berlin, Heidelberg},
  title     = {Fundamentals of Convex Analysis},
  publisher = {Springer Berlin Heidelberg},
  author    = {Hiriart-Urruty, Jean-Baptiste and Lemaréchal, Claude},
  year      = {2001}
}

@inproceedings{hoffman_new_2001,
  author    = {Hoffman, A. and Lee, J. and Williams, J.},
  editor    = {Atkinson, Anthony C. and Hackl, Peter and Müller, Werner G.},
  title     = {New Upper Bounds for Maximum-Entropy Sampling},
  booktitle = {mODa 6 --- Advances in Model-Oriented Design and Analysis},
  year      = {2001},
  publisher = {Physica-Verlag HD},
  address   = {Heidelberg},
  pages     = {143--153}
}

@book{horn_matrix_1985,
  place     = {Cambridge},
  title     = {Matrix Analysis (2nd Edition)},
  publisher = {Cambridge University Press},
  author    = {Horn, Roger A. and Johnson, Charles R.},
  year      = {2013}
}

@article{ko_exact_1995,
  title   = {An Exact Algorithm for Maximum Entropy Sampling},
  volume  = {43},
  number  = {4},
  journal = {Operations Research},
  author  = {Ko, Chun-Wa and Lee, Jon and Queyranne, Maurice},
  year    = {1995},
  pages   = {684--691}
}

@article{krause_near-optimal_2008,
  author    = {Krause, Andreas and Singh, Ajit and Guestrin, Carlos},
  title     = {Near-Optimal Sensor Placements in Gaussian Processes: Theory, Efficient Algorithms and Empirical Studies},
  journal   = {Journal of Machine Learning Research},
  year      = {2008},
  volume    = {9},
  pages     = {235--284},
  publisher = {JMLR.org}
}

@article{lee_constrained_1998,
  author    = {Jon Lee},
  journal   = {Operations Research},
  number    = {5},
  pages     = {655--664},
  publisher = {INFORMS},
  title     = {Constrained Maximum-Entropy Sampling},
  volume    = {46},
  year      = {1998}
}

@article{lee_linear_2003,
  author  = {Lee, J. and Williams, J.},
  title   = {A linear integer programming bound for maximum-entropy sampling},
  journal = {Mathematical Programming},
  year    = {2003},
  volume  = {94},
  number  = {2--3},
  pages   = {247--256}
}

@article{lewis_convex_1996,
  author  = {Lewis, A. S.},
  title   = {Convex Analysis on the Hermitian Matrices},
  journal = {SIAM Journal on Optimization},
  volume  = {6},
  number  = {1},
  pages   = {164--177},
  year    = {1996},
}

@article{li_best_2023,
  author  = {Li, Yongchun and Xie, Weijun},
  title   = {Best Principal Submatrix Selection for the Maximum Entropy Sampling Problem: Scalable Algorithms and Performance Guarantees},
  journal = {Operations Research},
  volume  = {72},
  number  = {2},
  pages   = {493-513},
  year    = {2024}
}

@inproceedings{li_augmented_2024,
  author    = {Li, Yongchun},
  editor    = {Megow, Nicole and Basu, Amitabh},
  title     = {The Augmented Factorization Bound for Maximum-Entropy Sampling},
  booktitle = {Integer Programming and Combinatorial Optimization},
  year      = {2025},
  publisher = {Springer Nature Switzerland},
  address   = {Cham},
  pages     = {412--426}
}

@article{lin_eigenvalue_2011,
  author    = {Minghua Lin and Henry Wolkowicz},
  title     = {An eigenvalue majorization inequality for positive semidefinite block matrices},
  journal   = {Linear and Multilinear Algebra},
  volume    = {60},
  number    = {11-12},
  pages     = {1365--1368},
  year      = {2012},
  publisher = {Taylor \& Francis},
}

@book{marshall_inequalities_2011,
  title     = {Inequalities: Theory of Majorization and Its Applications},
  author    = {Marshall, Albert W and Olkin, Ingram and Arnold, Barry C},
  year      = {2011},
  publisher = {Springer}
}

@inbook{nemhauser_integer_1988,
  author    = {Nemhauser, George and Wolsey, Laurence},
  publisher = {John Wiley \& Sons, Ltd},
  title     = {Integral Polyhedra},
  booktitle = {Integer and Combinatorial Optimization},
  chapter   = {III.1},
  pages     = {533-607},
  year      = {1988}
}

@inproceedings{nikolov_randomized_2015,
  author    = {Nikolov, Aleksandar},
  title     = {Randomized Rounding for the Largest Simplex Problem},
  year      = {2015},
  publisher = {Association for Computing Machinery},
  address   = {New York, NY, USA},
  booktitle = {Proceedings of the Forty-Seventh Annual ACM Symposium on Theory of Computing},
  pages     = {861--870},
  numpages  = {10},
  location  = {Portland, Oregon, USA},
  series    = {STOC '15}
}

@article{shewry_maximum_1987,
  author    = {M. C. Shewry and H. P. Wynn},
  title     = {Maximum entropy sampling},
  journal   = {Journal of Applied Statistics},
  volume    = {14},
  number    = {2},
  pages     = {165--170},
  year      = {1987},
  publisher = {Taylor \& Francis}
}

@book{nemirovski_problem_1983,
  author    = {Nemirovski, Arkadi and Yudin, David},
  title     = {Problem Complexity and Method Efficiency in Optimization},
  publisher = {Wiley},
  year      = {1983},
  series    = {Wiley-Interscience Series in Discrete Mathematics}
}

@book{kilinc-karzan_essential_2025, 
  author    = {K{\i}l{\i}n{\c{c}}-Karzan, Fatma and Nemirovski, Arkadi},
  title     = {Essential Mathematics for Convex Optimization},
  publisher = {Cambridge University Press},
  address   = {Cambridge},
  year      = {2025}
}

@article{dudley_second_1977,
  author    = {Dudley, R. M.},
  title     = {On Second Derivatives of Convex Functions},
  journal   = {Mathematica Scandinavica},
  volume    = {41},
  number    = {1},
  pages     = {159--174},
  year      = {1977},
  publisher = {Mathematica Scandinavica}
}

@book{rockafellar_convex_1970,
  author    = {Rockafellar, R. Tyrrell},
  title     = {Convex Analysis},
  publisher = {Princeton University Press},
  year      = {1970},
}

@article{fampa_recent_2026,
  author  = {Fampa, Marcia and Lee, Jon},
  title   = {Recent Advances in Maximum-Entropy Sampling},
  journal = {Kuwait Journal of Science},
  volume  = {53},
  number  = {1},
  pages   = {100527},
  year    = {2026},
}

@book{rockafellar_variational_1998,
  author    = {R. Tyrrell Rockafellar and Roger J. B. Wets},
  title     = {Variational Analysis},
  series    = {Grundlehren der mathematischen Wissenschaften},
  volume    = {317},
  publisher = {Springer},
  address   = {Berlin, Heidelberg},
  year      = {1998},
  pages     = {XII + 736},
  edition   = {1},
}

@article{magnus_differentiating_1985,
  author  = {Jan R. Magnus},
  title   = {On Differentiating Eigenvalues and Eigenvectors},
  journal = {Econometric Theory},
  year    = {1985},
  volume  = {1},
  number  = {2},
  pages   = {179--191},
}

@misc{shen_parameter-free_2026,
  title        = {Parameter-Free Non-Ergodic Extragradient Algorithms for Solving Monotone Variational Inequalities},
  author       = {Shen, Lingqing and K\i{}l\i{}n\c{c}-Karzan, Fatma},
  year         = {2026},
  eprint       = {2604.07662},
  archivePrefix= {arXiv},
  url          = {https://arxiv.org/abs/2604.07662}
}
